\documentclass[UTF-8]{article}
\usepackage{amssymb}
\usepackage{mathrsfs}
\usepackage{graphicx}
\usepackage{amsmath}
\usepackage{amsthm}
\allowdisplaybreaks[4]
\usepackage{dsfont}
\usepackage{authblk}
\usepackage{subcaption}
\usepackage{geometry}
\usepackage{booktabs}
\usepackage{algorithm}
\usepackage{algorithmic}
\usepackage{xcolor}
\usepackage{hyperref}
\numberwithin{equation}{section}
\newtheorem{definition}{Definition}
\newtheorem{assumption}{Assumption}

\newtheorem{lemma}{Lemma}
\newtheorem{proposition}{Proposition}
\newtheorem{theorem}{Theorem}
\newtheorem{remark}{Remark}

\newcommand{\dd}{\mathsf {d\kern -0.07em l}} 
\newcommand{\bgeqn}{\begin{eqnarray}}
\newcommand{\edeqn}{\end{eqnarray}}
\newcommand{\bgeq}{\begin{eqnarray*}}
\newcommand{\edeq}{\end{eqnarray*}}
\newcommand{\bec}{\begin{center}}

\def\R{\mathbb{R}}

\title{Adaptive Distributionally Robust Optimal Control with Bayesian Ambiguity Sets}
\author[1,2]{Wentao Ma\thanks{Email: \texttt{mwtmwt7@stu.xjtu.edu.cn}}}
\author[1,2]{Zhiping Chen\thanks{Email: \texttt{zchen@mail.xjtu.edu.cn}}}
\author[3]{Huifu Xu\thanks{Email: 
\texttt{hfxu@se.cuhk.edu.hk}}}
\author[4]{Enlu Zhou\thanks{Email: 
\texttt{enlu.zhou@isye.gatech.edu}}}

\affil[1]{School of Mathematics and Statistics, Xi'an Jiaotong University, Xi'an, Shaanxi, P. R. China}
\affil[2]{Research Center for Optimization Technology and Intelligent Game, Xi'an International Academy for Mathematics and Mathematical Technology, Xi'an, P. R. China}
\affil[3]{Department of Systems Engineering and Engineering Management, The Chinese University of Hong Kong, Hong Kong}
\affil[4]{H. Milton Stewart School of Industrial and Systems Engineering, Georgia Institute of Technology, Atlanta, USA}
\date{}

\begin{document}
	\maketitle
\begin{abstract}
In stochastic optimal control (SOC), uncertainty may arise from incomplete knowledge of the true probability distribution of the underlying environment, which is known as Knightian or epistemic uncertainty. Distributionally robust optimal control (DROC) models are subsequently proposed to tackle this source of uncertainty. While such models are effective in some practical applications, 
most existing DROC models are offline and can be overly conservative when data are scarce.
Moreover, they cannot be applied to the case when samples are generated episodically.
Motivated by the Bayesian SOC framework recently proposed by Shapiro et al.~\cite{shapiro2025episodic}, we propose an adaptive DROC model in which the ambiguity set is updated via Bayesian learning from new data.
Under some moderate conditions,
we derive a tractable risk-averse reformulation, establish consistency of the optimal value function and optimal policy for an infinite-horizon SOC and establish a finite-sample posterior credibility guarantee for the policy value induced by the proposed episodic Bayesian DROC model.
We also
study the stability 
and statistical robustness
of the proposed model 
with respect to sample perturbations that often arise in data-driven environments. 
To solve the episodic Bayesian DROC model, we propose a Bellman-operator cutting-plane (BOCP) algorithm
that is computationally efficient and provably convergent.
Numerical results on an inventory control problem demonstrate the effectiveness, adaptivity, and robust performance of the proposed model and algorithm.
\end{abstract}


\section{Introduction}
Stochastic optimal control (SOC) provides a principled framework for sequential decision-making under uncertainty, where system dynamics evolve according to stochastic processes.
In classical SOC models, the underlying randomness is often assumed to follow a known probability distribution. For instance, a standard infinite-horizon discrete-time
SOC problem \cite{bertsekas1996stochastic,hernandez2012further}
seeks an optimal policy $\pi$ (a decision rule mapping states to actions) over the feasible policy set $\Pi$ that minimizes the expected long-term discounted cost:
\begin{equation}\label{eq-whole}
	\min_{\pi \in \Pi}\mathbb{E}^{\pi}_{P^c} \left[ \sum_{t=1}^{\infty} \gamma^{t-1}\mathcal{C}(s_t,a_t, \xi_t)\right],
\end{equation}
where the expectation is taken with respect to the known joint distribution of the random variables involved \cite{bertsekas2012dynamic}. Here,
\(s_t\in\mathcal{S}\subseteq\mathbb{R}^m\) and \(a_t\in\mathcal{A}\subseteq\mathbb{R}^n\) denote the system state and action at time $t$, respectively,
\(\mathcal{S}\) and \(\mathcal{A}\) are nonempty closed sets. The cost function \(\mathcal{C}: \mathcal{S} \times \mathcal{A} \times \Xi \to\mathbb{R}\) represents the immediate cost, with \(\gamma \in [0, 1)\) being the discount factor. The randomness is modeled by an independent and identically distributed (i.i.d.) sequence \(\{\xi_t\}_{t\geq1}\) of random vectors with support set $\Xi\subseteq\mathbb{R}^k$ and 
probability distribution \(P^c \in \mathscr{P}(\Xi)\), where \(\mathscr{P}(\Xi)\) denotes the set of all Borel probability measures on \((\Xi,\mathcal{B}(\Xi))\). The dynamics follow \(s_{t+1} = g(s_t, a_t,\xi_t)\), where \(g:\mathcal{S}\times\mathcal{A}\times\Xi\to\mathcal{S}\) is the transition function.

The SOC model has been extensively applied across various domains \cite{bertsekas1996stochastic}. The formulation \eqref{eq-whole} fits within the framework of dynamic stochastic programming by treating the state-action pair $(s_t, a_t)$ as the decision, exemplified in inventory management \cite{shapiro2021lectures} and hydrothermal planning \cite{guigues2023risk}. Moreover, \(g\) together with the distribution of \(\xi_t\) induces the transition kernel of a Markov decision process (MDP) \cite{lin2022bayesian}. For a detailed discussion on the connections between SOC and MDP frameworks, readers may refer to Chapter 3.5 of \cite{puterman2014markov}.

It is well known \cite{bertsekas1996stochastic,wang2023bayesian} that the optimal value function \(V^*: \mathcal{S} \to\mathbb{R}\) 
and the optimal policy \(\pi^*\) for \eqref{eq-whole} can be derived by the Bellman equation:
\begin{equation}\label{true}
	V^{*}(s) = \inf_{a \in \mathcal{A}} \mathbb{E}_{P^c} \left[ \mathcal{C}(s,a, \xi) + \gamma V^{*}(g(s, a, \xi)) \right].
\end{equation}
However, in some practical applications, particularly data-driven problems, the true distribution \(P^c\) is unknown and consequently
solving \eqref{true} 
directly is infeasible.
This introduces an additional layer of uncertainty beyond the intrinsic randomness of \(\xi\); 
see, e.g., \cite{wang2024aleatoric,ma2024bayesian} and the references therein.
Such uncertainty about $P^c$ is referred to as
\emph{epistemic (Knightian) uncertainty} \cite{dibiasi2021measuring}, which may arise from incomplete information about \(P^c\) due to limited, noisy, or contaminated observations. In contrast, the irreducible randomness in the system represented by \(\xi\), even when \(P^c\) is known, is termed \emph{aleatoric uncertainty}.
The distributional uncertainty prevents exact evaluation of the expectation in \eqref{true} and motivates data-driven approximations and robust control strategies.

A common approach approximates the unknown distribution by a nominal distribution $\hat{P}_N^\mathrm{E}$\footnote{We use the subscript $N$ to indicate that the variable is constructed from the sample set with $N$ observations.}
derived via methods such as Monte Carlo sampling \cite{cooper2012performance,haskell2016empirical} (also known as sample average approximation (SAA) 
in stochastic programming), leading to the empirical Bellman equation: 
\begin{equation}\label{SAA}
	V_N^{\mathrm{E}}(s) = \inf_{a \in \mathcal{A}} \mathbb{E}_{\hat{P}_N^\mathrm{E}} [\mathcal{C}(s,a, \xi) + \gamma V_{N}^{\mathrm{E}}(g(s, a, \xi))],
\end{equation}
where $V^{\mathrm{E}}_N$ denotes the associated value function. Although convergence results for \eqref{SAA} are well established in \cite{haskell2016empirical}, the empirical approach can be sensitive to data scarcity and model misspecification, which may yield policies that perform poorly in the true environment.
To mitigate approximation errors due to limited sample size, 
distributionally robust optimization (DRO) has subsequently
been integrated into the SOC framework. By constructing an ambiguity set $\mathcal{P}\subseteq\mathscr{P}(\Xi)$ around the empirical distribution, distributionally robust optimal control (DROC) optimizes against the worst-case distribution within $\mathcal{P}$ \cite{xu2010distributionally,yang2020wasserstein}. The DROC problem yields the distributionally robust Bellman equation:
\begin{equation}\label{DR}
	V^{\mathrm{DRO}}(s) = \inf_{a \in \mathcal{A}}\sup_{{P}\in \mathcal{P}} \mathbb{E}_{P} [\mathcal{C}(s,a, \xi) + \gamma V^{\mathrm{DRO}}(g(s, a, \xi))],
\end{equation}
where $V^{\mathrm{DRO}}$ denotes the corresponding value function. Various constructions of $\mathcal{P}$ have been explored, including moment-based sets \cite{delage2010distributionally,van2015distributionally} and Wasserstein distance-based sets \cite{kim2023distributional,hanasusanto2015distributionally,taskesen2023distributionally} for linear-quadratic control, and total variation distance-based sets \cite{tzortzis2019infinite} for infinite-horizon average cost control.
By duality principles, many distributionally robust formulations admit equivalent or closely related risk-averse interpretations \cite{guigues2023risk,shapiro2021lectures}.
Despite their robustness, many existing DROC models are calibrated from offline data and may not adapt efficiently to newly observed samples \cite{shapiro2025episodic}. 

Another important research direction for addressing distributional uncertainty is Bayesian adaptive control \cite{osband2013more,strens2000bayesian}, which integrates learning of the unknown dynamics with policy optimization by continuously updating distribution estimates from observed data via Bayesian inference.
Such approaches typically assume that $P^c$ belongs to a parametric family $\{{P}_\theta : \theta \in \Theta\}$ with parameter space
$\Theta\subseteq\mathbb{R}^d$. That is, there exists an unknown parameter $\theta^c\in\Theta$ such that $P^c \equiv P_{\theta^c}$.
One starts from a prior distribution $\mu_0$ based on the available information and then updates it at each episode using newly observed data. 
The resulting posterior distribution \(\mu_N\) quantifies the epistemic uncertainty and can be viewed as a belief (information) state, augmenting the physical state to form an extended state space. 
Dynamic programming techniques applied to this augmented state space yield policies that treat distributional uncertainty as a partially observed state \cite{lin2022bayesian,rieder1975bayesian}. Further, a unified framework for SOC models based on the Bayesian composite risk approach was proposed in \cite{ma2024bayesian}.
Nevertheless, such approaches often become computationally intractable due to the infinite-dimensional nature of the belief state. Recently, Shapiro et al.~\cite{shapiro2025episodic} proposed an adaptive episodic approximation that focuses on the Bayesian average counterpart of \eqref{true} with respect to the posterior $\mu_N$:
\begin{equation}\label{BE}
	V_N^{\mathrm{B}}(s) = \inf_{a\in \mathcal{A}} \mathbb{E}_{\mu_N}\mathbb{E}_{{P}_{\theta}} [\mathcal{C}(s,a,\xi) + \gamma V_{N}^{\mathrm{B}}(g(s,a,\xi))],
\end{equation}
where $V^{\mathrm{B}}_N$ denotes the value function associated with \eqref{BE}. Despite its adaptivity,
this method can reduce to a point-estimation-based technique in some cases. For instance, when the parametric family depends affinely on $\theta$ (e.g., finite mixture family \cite{ma2025bayesian}), the value function induced by \eqref{BE} coincides with that of \eqref{SAA} with the nominal distribution chosen as $\hat{P}_N^\mathrm{E}:=P_{\hat{\theta}_N}$, where $\hat{\theta}_N=\mathbb{E}_{\mu_N}[\theta]$ is the posterior mean \cite{chen2024bayesian}.
Consequently, this method may inherit part of the sensitivity to data variability exhibited by the empirical method \eqref{SAA}. Moreover, when prior structural knowledge is limited, reliance on a specific low-dimensional family may increase the risk of model misspecification in practice, especially when the underlying data exhibit features such as multimodality, skewness, or heavy tails \cite{xie2021nonparametric}. 
Such features may induce persistent model misspecification even with large datasets.

More broadly, vulnerability to imperfect data is a generic issue in data-driven optimal control. In practice, data-driven SOC and DROC models depend critically on finite historical data and may be affected by poor data quality, limited sample sizes, data contamination, or adversarial perturbations \cite{xu2021quantitative,pichler2022quantitative}. Such effects can distort the estimated uncertainty model and, in turn, compromise the robustness and stability of the resulting policies. While related questions have been extensively studied in static DRO problems \cite{lam2019recovering}, their dynamic counterparts in DROC remain relatively underexplored, especially from the perspectives of quantitative stability and statistical robustness.

These considerations motivate two complementary lines of analysis: (i) developing an episodic, data-adaptive DROC formulation to handle epistemic uncertainty; and (ii) establishing stability and statistical robustness under data perturbations and contamination.
To this end, we propose an episodic Bayesian DROC framework and study the associated Bellman equation:
\begin{equation}\label{eq:BDR-Bellman}
	V_N(s) = \inf_{a \in \mathcal{A}}\sup_{{P}\in \mathcal{P}_{\mu_N}} \mathbb{E}_{P} [\mathcal{C}(s,a, \xi) + \gamma V_{N}(g(s, a, \xi))],
\end{equation}
where $V_N$ denotes the associated value function and $\mathcal{P}_{\mu_N}$ is a Bayesian ambiguity set induced by the posterior $\mu_N$, which is updated using newly observed data after each episode. 
Unlike standard DROC with a fixed ambiguity set, our formulation combines the robustness of \eqref{DR} with the Bayesian adaptivity of \eqref{BE}, thereby quantifying and mitigating epistemic uncertainty through adaptive posterior learning while preserving distributional robustness.
To make \eqref{eq:BDR-Bellman} operational and theoretically well-founded, we focus on the following key questions:
(a) how to explicitly construct the ambiguity set $\mathcal{P}_{\mu_N}$ and derive a tractable reformulation of \eqref{eq:BDR-Bellman}; (b) whether the induced Bellman operator admits a unique fixed point and the corresponding optimal value function is well-defined; (c) how the resulting value function and the associated optimal policy respond statistically to changes in the observations; and (d) how to efficiently compute the optimal value function and policy, particularly in the multi-episode setting where the ambiguity set is repeatedly updated.
The main contributions of this study include:
\begin{itemize}

\item \textbf{Modeling framework.} We introduce 
an adaptive Bayesian DROC model
where the ambiguity set is updated episodically.
An important feature of the model is that the center of
the ambiguity set is updated via Bayes' rule using the samples observed in each episode and its radius is calibrated via a $(1-\alpha)$-credible region (Proposition~\ref{prop-ambiguity}, Remark~\ref{rem-1}).
The Bayesian DROC model reduces to a Bayesian SOC model when $\alpha=1$ (Remark~\ref{rem-2}). 
Unlike traditional Bayesian models, which typically assume that the true distribution belongs to or can be well approximated by a parametric family, we consider instead the case where the true distribution has a known finite support set but an unknown probability vector, which is learned from data.
This makes the proposed model applicable to arbitrary finite-support distributions without imposing specific shape restrictions such as unimodality, symmetry, or light tails.
The framework can also be extended to settings where the true distribution is continuous, by first discretizing it and then applying the proposed method to obtain an approximate solution.
Finally, the box-type ambiguity set yields an explicit mean-risk reformulation of the worst-case expectation (Theorem \ref{thm-reformulation}) as a convex combination of an expectation term and a CVaR term under suitable reference measures.
This yields a coherent risk-measure interpretation of the induced Bellman operator and clarifies how robustness against epistemic uncertainty translates into risk aversion in the SOC objective.

\item \textbf{Theoretical analysis.} 
We conduct theoretical analysis from three perspectives: episodic asymptotic convergence analysis of the ambiguity set and the resulting robust value function as the number of episodes $N$ tends to infinity (Theorems \ref{thm-bc1} and \ref{thm-bc-2});
quantitative stability analysis of the value function and optimal policies with respect to perturbations in the ambiguity set (Theorems~\ref{thm-stab-value} and \ref{thm-sta-policy});
and statistical robustness analysis of the statistical estimators of the value function under potential contamination of sample data (Theorem \ref{thm-qsr}). Unlike 
standard analysis in the MDP/SOC literature under the sup-norm \cite{shapiro2025episodic}, all analyses are performed under a weighted supremum norm to accommodate potentially unbounded costs.
The episodic convergence analysis addresses not only the consistency of the value function but also establishes a finite-sample posterior credibility guarantee for the associated policy value (Theorem~\ref{thm-finite-sample-guarantee}), i.e., the robust value function obtained by solving \eqref{eq:BDR-Bellman} serves as a posterior-credible upper bound on the true value of the corresponding policy at a prescribed credibility level.
Notably, a $(1-\alpha)$ posterior credible ambiguity set calibrated from the one-step model already yields the same $(1-\alpha)$ posterior guarantee for the infinite-horizon discounted problem, with no horizon-dependent adjustment.
The quantitative stability analysis yields explicit bounds on the value function in terms of perturbations in the ambiguity set. Finally, the statistical robustness analysis quantifies the impact of contaminated samples on the induced distribution of the estimated value function.

\item \textbf{Solution method.}
We develop a Bellman-operator cutting-plane (BOCP) algorithm for the infinite-horizon Bayesian DROC problem and establish a monotonicity property (Lemma~\ref{lem:valid}) and uniform convergence on compact state spaces (Theorem~\ref{thm:cp-episode}).
The proposed method iteratively refines the value-function approximation via a sequence of cutting-plane updates to the Bellman operator.
To further reduce computational cost in the multi-episode setting, we develop a provably safe warm-start mechanism that reuses cutting planes across posterior updates via an explicit validity test (Proposition \ref{prop:warm-max}).
We validate the approach on an inventory control problem, illustrating its convergence, robustness properties, and adaptivity across episodes.
\end{itemize}

The rest of the paper is 
structured as follows. Section~2 introduces the episodic Bayesian DROC model with posterior-driven box-type ambiguity sets via Bayesian adaptive learning. We also derive an equivalent mean-risk reformulation of the resulting Bellman operator.
Section~3 
establishes the asymptotic convergence of the episodic Bayesian DROC value functions and optimal policies to their true counterparts as the number of episodes grows, and also presents a finite-sample posterior credibility guarantee for the associated policy value based on posterior credible ambiguity sets.
Section~4 studies stability and statistical robustness of the proposed model under data perturbations and contamination.
Section~5 develops a computational method based on a Bellman-operator cutting-plane (BOCP) scheme.
Section~6 reports numerical results for the episodic Bayesian DROC model, illustrating the effectiveness and applicability.

Throughout the paper, we use the following notation. For a metric space $(\mathbb{X}, \dd)$, we write $\dd(x, S)$ for the distance from a point $x\in\mathbb{X}$ to a set $S\subseteq\mathbb{X}$, $\mathbb{D}\left(S_1, S_2 ; \dd\right)$ for the excess of a set $S_1\subseteq\mathbb{X}$ over another set $S_2\subseteq\mathbb{X}$ associated with distance $\dd$, i.e.,
$$
\mathbb{D}\left(S_1, S_2 ; \dd\right)=\sup _{x \in S_1} \dd\left(x, S_2\right)=\sup _{x \in S_1} \inf _{y \in S_2} \dd(x, y)
$$
and $\mathbb{H}\left(S_1, S_2 ;\dd\right)$ for the Hausdorff distance between the two nonempty compact sets, that is,
$$
\mathbb{H}\left(S_1, S_2 ; \dd\right)=\max \left\{\mathbb{D}\left(S_1, S_2 ; \dd\right), \mathbb{D}\left(S_2, S_1 ; \dd\right)\right\} .
$$

\section{Episodic Bayesian DROC Model}
In this section, we present the episodic Bayesian DROC framework \eqref{eq:BDR-Bellman}, construct a Bayesian ambiguity set based on a posterior credible region, and establish a mean-risk reformulation
that connects the episodic Bayesian DROC objective to a risk-averse SOC formulation. Since many practical MDP/SOC implementations rely on finite or discretized representations of uncertainty \cite{wang2023bayesian}, we adopt a finite support set for the associated random vectors
(see, e.g., \cite{carpentier2012dynamic,pfeiffer2018two}). Motivated by this and for clarity of exposition, we make the following basic assumption.

\begin{assumption}
\label{ass-finite-support}
 Assume that \(\Xi:=\{\xi^{1},\ldots,\xi^{J}\}\subseteq\mathbb{R}^k\) is a finite, known support set with $J\in \mathbb{N}_+$.
\end{assumption}

Under Assumption \ref{ass-finite-support}, each distribution $P\in\mathscr{P}(\Xi)$ can be represented by the probability vector $(p_j)_{j=1}^J:=(p_1,\ldots,p_J)$ with \(p_{j} = {P}(\xi = \xi^{j})\) for \(j = 1, \ldots, J\). 
With a slight abuse of notation for brevity, we use $P$ to denote both the distribution and its probability vector, as they are in one-to-one correspondence.
Hence, under this identification, we may write $\mathscr{P}(\Xi):=\{P\in[0,1]^J:\sum_{j=1}^J p_j=1\}$.
When Bayesian learning is considered, we treat the unknown distribution as a random probability vector $P\in\mathscr P(\Xi)$ with a posterior distribution, which we will introduce in the next subsection.
To address the uncertainty in the true distribution $P^c$, we adopt an ambiguity set centered at a reference distribution. This approach is consistent with widely used DRO formulations for discrete distributions, which construct ambiguity sets around a reference distribution using distances such as $L_1$, $L_\infty$ norms \cite{jiang2018risk,huang2017study}, $\chi^2$-divergence \cite{philpott2018distributionally}, Burg entropy \cite{wang2016likelihood}, and total variation distance \cite{rahimian2022effective}. 
Specifically, following \cite{mehrotra2014models}, we consider the box-type ambiguity set
\begin{equation}\label{eq-ambiguity-set}
		\mathcal{P}=\left\{P\in \mathscr{P}(\Xi):\hat{P}^l \leq P\leq \hat{P}^u\right\},
	\end{equation}
where $\hat{P}^l=\hat{P}-\hat{r}$ and $\hat{P}^u=\hat{P}+\hat{r}$ denote the componentwise lower and upper bounds associated with a reference distribution $\hat{P}=(\hat p_j)_{j=1}^J:=(\hat{p}_{1},\ldots,\hat{p}_{J})$ and tolerance vector $\hat{r}=(\hat r_j)_{j=1}^J:=(\hat{{r}}_{1},\ldots,\hat{{r}}_{J})$. Notably, when $\hat{r}_1=\ldots=\hat{r}_J$, the box-type ambiguity set \eqref{eq-ambiguity-set} reduces to an $L_\infty$ norm set. 
In what follows, we show how to construct and update the ambiguity set $\mathcal{P}_{\mu_N}$ of the form \eqref{eq-ambiguity-set} using the posterior $\mu_N$.

Assumption \ref{ass-finite-support} plays a critical role in both the formulation of the episodic Bayesian DROC model and the subsequent analysis. When $\xi$ is continuously distributed, one common approach is to discretize it using optimal quantization \cite{pflug2014multistage}. This involves partitioning the support set $\Xi$ into mutually exclusive regions via Voronoi tessellation \cite{li2024discretization}, and then 
using the corresponding cell centers as support points for a discrete approximation. Under suitable conditions, it is possible to derive bounds on the resulting approximation errors; see \cite{guo2019distributionally} for results pertaining to one-stage stochastic programs. 
It should be noted that this approach is generally applicable only when the support set is bounded\footnote{In cases where the support set is unbounded, one may truncate the tails under suitable tightness conditions.}. A comprehensive investigation of these discretization techniques and the associated error analysis is beyond the scope of this paper and is left for future research.

\subsection{Episodic Bayesian ambiguity sets}

We now proceed to discuss the construction of an adaptive reference distribution and the corresponding tolerance within the ambiguity set \eqref{eq-ambiguity-set}, guided by Bayesian learning principles.
From a Bayesian viewpoint, the unknown \(P^c\) is treated as a random distribution, initially characterized by a prior distribution \(\mu_0\) in the absence of new observations. The prior reflects initial beliefs about \(P^c\). 
Here, we adopt the conjugate Dirichlet distribution as the prior. In the absence of prior information, different kinds of noninformative priors (see \cite{gelman1995bayesian})
can be used as special cases of the Dirichlet distribution.
Under this setting, the unknown $P^c$ has the prior density $\mu_0(P)=B(\tau_0)^{-1}\prod_{j=1}^{J}{p}_j^{\tau_{j,0}-1}$. Here, we set $\tau_{j,0}>1$ for $j=1,\ldots,J$ with $B(\tau_0)$ serving as a normalizing constant. 
Upon observing a data sequence \(\vec{\hat{\xi}}_N:=(\hat{\xi}_1,\ldots,\hat{\xi}_{N})\), beliefs about $P^c$ are updated episodically, resulting in the posterior distribution \(\mu_N(P)\). We assume that the observations are drawn i.i.d. from the true \(P^c\).
Under Bayesian inference, $\mu_N$ also evolves as a Dirichlet distribution with updated parameters $\tau_N:=(\tau_{j,N})_{j=1}^J$, defined by \begin{equation}\label{eq-update-tau}
 \tau_{j,N}=\tau_{j,0}+\sum_{i=1}^N\mathds{1}_{\{\hat{\xi}_i=\xi^j\}},
\end{equation}
where $\mathds{1}_{\{\hat{\xi}_i=\xi^j\}}$ equals $1$ if $\hat{\xi}_i=\xi^j$ and $0$ otherwise. Explicitly, the update of the posterior distribution is as follows:
\begin{equation}\label{posterior}
	\mu_N(P)
 :=\mu(P|\vec{\hat{\xi}}_N)
 =B(\tau_N)^{-1}\prod_{j=1}^{J}{p}_j^{\tau_{j,N}-1}.
\end{equation}
This posterior $\mu_N$ reflects the epistemic uncertainty about $P^c$ given the cumulative data $\vec{\hat{\xi}}_N$.
Based on this posterior belief $\mu_N$, we seek to construct an adaptive ambiguity set $\mathcal{P}_{\mu_N}$ in the form of \eqref{eq-ambiguity-set} with a prescribed posterior credibility. Specifically, following the standard Bayesian credible regions discussed in \cite{berger2013statistical,gelman1995bayesian}, given the dataset $\vec{\hat{\xi}}_N$ and a credible level $\alpha\in(0,1)$, we say that a posterior-based set $\mathcal{P}_{\mu_N}$ is a $(1-\alpha)$-posterior credible region for $P^c$ if
$$\mathbb{P}_{\mu_N}(P^c\in{\mathcal{P}_{\mu_N}})\geq1-\alpha,$$ where $\mathbb{P}_{\mu_N}$ denotes the posterior probability measure induced by $\mu_N$. 

{\color{black} It is important to distinguish a posterior credible region in the Bayesian setting from a classical confidence region in the frequentist setting. A confidence region evaluates the long-run coverage of a random set under repeated sampling, with $P^c$ treated as fixed but unknown.
In our framework, however, the ambiguity set is updated sequentially as new data are accumulated across episodes.
Accordingly, the relevant probabilistic notion here is posterior credibility conditional on the observed data, rather than repeated-sampling coverage.
While the data-generating distribution $P^c$ remains fixed but unknown, our uncertainty about it after observing $\vec{\hat{\xi}}_N$ is represented by the posterior distribution $\mu_N$.
For this reason, we 
formulate
the ambiguity set through posterior credibility rather than frequentist confidence.
For broader discussions on the differences and trade-offs between frequentist and Bayesian approaches, we refer the reader to \cite{efron2021computer,gupta2019near}.
}

This posterior-credibility-based perspective has been widely used in constructing ambiguity sets for DRO; see, e.g., \cite{chen2025data,li2022risk,li2025bayesian}. Following this line of literature, the next proposition provides a concrete construction of such an ambiguity set under our setting.
The proposed construction is computationally convenient and will be crucial to the mean-risk reformulation in Theorem~\ref{thm-reformulation}.
It is motivated by tractable coordinate-wise bounds used in \cite{li2020confidence,nilim2004robust}, together with a local normal approximation of the posterior around the posterior mode. To obtain a joint credibility level $1-\alpha$ across all $J$ coordinates, we apply a Bonferroni correction with $\alpha'=\alpha/J$.

\begin{proposition}[Approximate posterior credible region]\label{prop-ambiguity}
Fix $\alpha\in(0,1)$ and let $\alpha'=\alpha/J$. For $j=1,\ldots,J$, define
\begin{equation}
\hat{p}_{j,N}=\frac{\tau_{j,N}-1}{\sum_{k=1}^J\tau_{k,N}-J},\ 
\hat{r}_{j,N}:=z_{1-\frac{\alpha'}{2}}\sqrt{\frac{\hat{p}_{j,N}(1-\hat{p}_{j,N})}{\sum_{k=1}^J\tau_{k,N}-J}},
\end{equation}
where $\tau_{j,N}$ are the posterior parameters in \eqref{eq-update-tau} and $z_{1-\frac{\alpha'}{2}}$ is the $(1-\alpha'/2)$-quantile of the standard normal distribution.
Consider the box-type ambiguity set
\begin{equation}\label{ambiguity}
\mathcal{P}_{\mu_N}=\mathfrak{P}(\hat{P}_{N},\hat{r}_N):=\Big\{P\in \mathscr{P}(\Xi): \hat{p}_{j,N}-\hat{r}_{j,N}\leq p_j\leq \hat{p}_{j,N}+\hat{r}_{j,N},\ j=1,\ldots,J\Big\}.
\end{equation}
Then $\mathcal{P}_{\mu_N}$ is an approximate $(1-\alpha)$ posterior credible region for $P^c$. That is, 
\begin{equation}
 \mathbb{P}_{\mu_N}(P^c\in{\mathfrak{P}(\hat{P}_{N},\hat{r}_N)})\geq1-\alpha+o(1),
\end{equation}
 where $o(1)$ denotes a quantity that vanishes as $N\to\infty$.
\end{proposition}

\begin{remark}\label{rem-1}
Let $\hat{P}_0=(\frac{\tau_{j,0}-1}{\sum_{k=1}^{J}\tau_{k,0}-J})_{j=1}^J$ denote the prior-only reference distribution (Dirichlet mode) and $\hat{P}_N^\mathrm{E}:=(\frac{\sum_{i=1}^N\mathds{1}_{\{\hat{\xi}_i=\xi^j\}}}{N})_{j=1}^J$ the data-only empirical distribution. Then the reference distribution $\hat{P}_N$ proposed in Proposition \ref{prop-ambiguity}, as the center of the Bayesian ambiguity set, admits the convex decomposition:
$$\hat{P}_N=\omega_N\hat{P}_0+(1-\omega_N)\hat{P}_N^\mathrm{E},$$ 
where $\omega_N=\frac{\sum_{k=1}^{J}\tau_{k,0}-J}{\sum_{k=1}^{J}\tau_{k,N}-J}$ and $\omega_N\downarrow0$ as $N\to\infty$. 
This representation makes explicit that $\hat{P}_N$ is a Bayesian adjusted estimator that blends the decision-maker’s prior $\tau_0$ with the data $\vec{\hat{\xi}}_N$, and that the influence of the prior decays as more data are accumulated.
Prior beliefs can be elicited from historical data, expert judgment or domain knowledge, allowing the ambiguity set to reflect credible information beyond the sample. 
Compared with purely data-driven approaches, such adaptability of the proposed Bayesian ambiguity set helps improve robustness by incorporating prior information in addition to sample frequencies.
\end{remark}

\begin{remark}[Connection to episodic Bayesian SOC]\label{rem-2}
Under Assumption~\ref{ass-finite-support}, the mapping 
$P\mapsto \mathbb{E}_{P}[Z(\xi)]$ is affine in the probability vector $P$ for any integrable function $Z(\xi)$.
Thus, the Bayesian-average objective satisfies
\[
\mathbb{E}_{\mu_N}\mathbb{E}_{P}\left[\mathcal{C}(s,a, \xi) + \gamma V_{N}(g(s, a, \xi))\right]
=
\mathbb{E}_{\bar P_N}\left[\mathcal{C}(s,a, \xi) + \gamma V_{N}(g(s, a, \xi))\right],
\]
where $\bar P_N:=\mathbb{E}_{\mu_N}[P]=\Big(\frac{\tau_{j,N}}{\sum_{k=1}^J\tau_{k,N}}\Big)_{j=1}^J$ is the Dirichlet posterior mean (which also coincides with the posterior predictive distribution of $\xi$ \cite{gelman1995bayesian}). 
Moreover, if we choose the center of the ambiguity set as $\bar P_N$ and let the credibility level degenerate by setting $\alpha=1$ (i.e., $1-\alpha=0$), 
then one may take the corresponding posterior credible region to be the singleton set $\{\bar P_N\}$, namely, the box radius collapses to $0$.
In this degenerate case, the distributionally robust Bellman equation \eqref{eq:BDR-Bellman} reduces to
\[
V_N(s)=\inf_{a\in\mathcal A}\mathbb{E}_{\bar P_N}\left[\mathcal{C}(s,a,\xi)+\gamma V_N\big(g(s,a,\xi)\big)\right],
\]
which coincides with the episodic Bayesian Bellman equation \eqref{BE} in \cite{shapiro2025episodic}.
Furthermore, 
the center $\hat P_N$ is the Dirichlet posterior \emph{mode} under parameters $\tau_N$, but it can also be interpreted as the Dirichlet posterior \emph{mean}
under shifted parameters $(\tau_N- 1)$:
\[
\hat p_{j,N}=\frac{\tau_{j,N}-1}{\sum_{k=1}^J\tau_{k,N}-J}
=\frac{\tilde\tau_{j,N}}{\sum_{k=1}^J\tilde\tau_{k,N}},
\qquad \tilde\tau_{j,N}:=\tau_{j,N}-1.
\]
Therefore, the episodic Bayesian SOC model in \cite{shapiro2025episodic} with prior
$\tilde\tau_0=\tau_0- 1$ coincides with our episodic Bayesian DROC model with prior
$\tau_0$ in the degenerate case $\alpha=1$.
\end{remark}

The following result illustrates that as the number of episodes $N$ grows, the Bayesian posterior of $P$ becomes increasingly concentrated around $P^c$. Let $\|\cdot\|$ denote the Euclidean norm. The convergence is in the almost surely (a.s.) sense with respect to the data-generating distribution, i.e., for almost every sequence \(\{\hat{\xi}_1, \hat{\xi}_2,\ldots \}\).

\begin{lemma}[Bayesian consistency]\label{lem-bc}
The reference distribution \(\hat{P}_N\) a.s. converges to \(P^c\). Furthermore, for any \(\epsilon > 0\), 
\begin{equation}\label{eq-bc}
		\lim_{N \to \infty} \int_{\|P - P^c\| \leq \epsilon} \mu(P|\vec{\hat{\xi}}_N) dP = 1,\ a.s.
\end{equation}
\end{lemma}

\begin{proof}
	The first part of the claim follows from 
 the law of large numbers, that is, $\frac{\sum_{i=1}^N\mathds{1}_{\{\hat{\xi}_i=\xi^j\}}}{N} \to p_{j}^c$ 
 as $N\to\infty$ almost surely
 for all $j=1,\ldots,J$.
 Thus $\hat{p}_{j,N} \to p_j^{c}$ and $\hat{P}_N\to P^c$ almost surely.	
	For the second part, 
 let \(\bar P_N:=\mathbb E_{\mu_N}[P]\) and \(\Sigma_N:=\mathrm{Cov}_{\mu_N}(P)\).
By Chebyshev's inequality,
\begin{equation}\label{bayescon}
	\int_{\|P-P^c\|\le\epsilon}\mu(P|\vec{\hat\xi}_N)dP\ge 1-\frac{\mathrm{tr}(\Sigma_N)+\|\bar P_N-P^c\|^2}{\epsilon^2}.
	\end{equation}
According to the Dirichlet posterior \eqref{posterior}, we have
$$
	\bar p_{j,N}=\frac{\tau_{j,0}+\sum_{i=1}^N\mathds{1}_{\{\hat{\xi}_i=\xi^j\}}}{\sum_{k=1}^J\tau_{k,0}+N}, \ \mathrm{tr}(\Sigma_N)=\sum_{j=1}^J\frac{(\tau_{j,0}+\sum_{i=1}^N\mathds{1}_{\{\hat{\xi}_i=\xi^j\}})(\sum_{k=1}^J\tau_{k,0}+N-\tau_{j,0}-\sum_{i=1}^N\mathds{1}_{\{\hat{\xi}_i=\xi^j\}})}{(\sum_{k=1}^J\tau_{k,0}+N)^2(\sum_{k=1}^J\tau_{k,0}+N+1)} .
	$$
Consequently, by the law of large numbers, it follows that $\mathrm{tr}(\Sigma_N)\to0$ and $\bar p_{j,N} \to p_j^{c}$ almost surely as $N\to\infty$. These facts and \eqref{bayescon} ensure the posterior consistency as stated in \eqref{eq-bc}.
\end{proof}

Lemma~\ref{lem-bc} establishes posterior consistency and the convergence of the Bayesian reference distribution $\hat P_N$ to $P^c$.
Since the ambiguity set is defined as a box around $\hat P_N$ with $\hat r_N$, we then consider the convergence for the ambiguity set $\mathcal P_{\mu_N}=\mathfrak P(\hat P_N,\hat r_N)$.
To this end, we use the Hausdorff metric to measure the distance between compact subsets of the probability simplex.

\begin{lemma}[Hausdorff Lipschitz continuity of Bayesian ambiguity sets]\label{lem-amb-cont}
Let $\mathcal U:=\{(P,r)\in\mathscr P(\Xi)\times\mathbb R_+^J:\mathfrak P(P,r)\neq\emptyset\}$.
Then for any $(\hat P_N,\hat r_N),(\hat P'_N,\hat r'_N)\in\mathcal U$,
\bgeqn 
\mathbb H\left(\mathfrak P(\hat P_N,\hat r_N),\mathfrak P(\hat P'_N,\hat r'_N);\|\cdot\|_\infty\right)
\le\|\hat P_N-\hat P'_N\|_\infty+\|\hat r_N-\hat r'_N\|_\infty.
\edeqn 
In particular, under the Bayesian construction in Proposition~\ref{prop-ambiguity},
\[
\mathbb H\left(\mathfrak P(\hat P_N,\hat r_N),\mathfrak P(P^c,0);\|\cdot\|_\infty\right)\to0
\quad \text{a.s.}
\]
\end{lemma}

\begin{proof}
Let $s(y,C)$ denote the support function of a compact set $C$. By H\"ormander formula (see e.g.,
Theorem II.~18 of \cite{castaing1977convex}),
\bgeq 
\mathbb H\left(\mathfrak P(\hat P_N,\hat r_N),\mathfrak P(\hat P'_N,\hat r'_N);\|\cdot\|_\infty\right)
=\sup_{\|y\|_1\leq 1} \left|s(y,\mathfrak P(\hat P_N,\hat r_N))- s(y,\mathfrak P(\hat P'_N,\hat r'_N))\right|. 
\edeq
Thus
\bgeq 
\mathbb H\left(\mathfrak P(\hat P_N,\hat r_N),\mathfrak P(\hat P'_N,\hat r'_N);\|\cdot\|_\infty\right)
&\leq& 
\mathbb H\left(\mathfrak P(\hat P_N,\hat r_N),\mathfrak P( \hat{P}_N,\hat r'_N);\|\cdot\|_\infty\right)
+ \mathbb H\left(\mathfrak P(\hat P_N,\hat r'_N),\mathfrak P(\hat P'_N,\hat r'_N);\|\cdot\|_\infty\right)\\
&\leq& \|\hat r_N-\hat r'_N\|_\infty +
\sup_{\|y\|_1\leq 1} \left|s(y,\mathfrak P(\hat P_N,\hat r'_N))- s(y,\mathfrak P(\hat P'_N,\hat r'_N))\right| \\
&=& \|\hat r_N-\hat r'_N\|_\infty +
\sup_{\|y\|_1\leq 1}\left| y^T(\hat{P}_N-\hat{P}'_N)\right|\\
&\leq & \|\hat r_N-\hat r'_N\|_\infty + \|\hat P_N-\hat P'_N\|_\infty.
\edeq
Finally, the convergence follows from Lemma~\ref{lem-bc} since $(\hat P_N,\hat r_N)\to(P^c,0)$ almost surely as $N\to\infty$.
\end{proof}

In light of the Bayesian credible region and Lemmas \ref{lem-bc} and \ref{lem-amb-cont}, as episode $N$ tends to infinity, $\mathcal{P}_{\mu_N}=\mathfrak P(\hat P_N,\hat r_N)$ almost surely converges to the singleton $\mathfrak P(P^c,0)=\{P^c\}$. This shows that Bayesian learning dynamically shrinks the ambiguity set to the true distribution as data accumulate.

\begin{remark}
In the preceding discussion, $\hat P_N$ and $\hat r_N$ denote the reference distribution and tolerance radius of the Bayesian ambiguity set constructed from the sample set $\vec{\hat\xi}_N$.
To clarify the functional dependence on the empirical data, we write $\hat P_N$ and $\hat r_N$ as functions of $\hat{P}_N^\mathrm{E}$ as follows:
\[
\hat p_{j,N}(\hat{P}_N^\mathrm{E}):=\frac{\tau_{j,0}-1+N \hat{p}_{j,N}^\mathrm{E}}{\sum_{k=1}^J \tau_{k,0}-J+N},
\ 
\hat r_{j,N}(\hat{P}_N^\mathrm{E}):=
z_{1-\frac{\alpha'}{2}}\sqrt{\frac{\hat{p}_{j,N}(\hat{P}_{N}^\mathrm{E})(1-\hat{p}_{j,N}(\hat{P}_{N}^\mathrm{E}))}{\sum_{k=1}^J\tau_{k,0}-J+N}}.
\]
This representation reveals that $\hat{P}_N$ is an affine function of $\hat{P}_N^\mathrm{E}$ (also shown in Remark \ref{rem-1}) and $\hat{r}_N$ is continuous in $\hat{P}_N^\mathrm{E}$.
Let $\Phi_N(\hat{P}_N^\mathrm{E}):=\mathfrak P(\hat P_N(\hat{P}_N^\mathrm{E}),\hat r_N(\hat{P}_N^\mathrm{E}))$. Then Lemma~\ref{lem-amb-cont} implies that $\Phi_N$ is Hausdorff continuous with respect to the $\|\cdot\|_\infty$ metric. Hence the ambiguity set $\mathcal{P}_{\mu_N}$ depends continuously on the empirical distribution of the observed data $\vec{\hat\xi}_N$.
\end{remark}

\subsection{Bellman equation and reformulation}
With the ambiguity set defined as in \eqref{ambiguity}, 
we can write the episodic Bayesian DROC Bellman equation
\eqref{eq:BDR-Bellman} as 
\begin{equation}\label{drmdp}
V_N(s)=\inf_{a\in\mathcal{A}}\sup_{{P}\in\mathfrak{P}(\hat{P}_{N},\hat{r}_N)}\mathbb{E}_{P}\left[\mathcal{C}(s,a,\xi)+\gamma V_N(g(s,a,\xi))\right].
\end{equation}
Let $\hat{V}^{*}_N$ denote the value function satisfying 
\eqref{drmdp}. The corresponding optimal policy can then be defined as:
\begin{equation}\label{drmdp-policy}
	\hat{\pi}_N^*(s) \in \arg\inf_{a \in \mathcal{A}} \sup_{{P}\in\mathfrak{P}(\hat{P}_{N},\hat{r}_N)}\mathbb{E}_{P}[\mathcal{C}(s,a,\xi) + \gamma \hat{V}_N^{*}(g(s,a,\xi))].
\end{equation}
The next theorem states that the distributionally robust counterpart \eqref{drmdp} can be equivalently reformulated as a mean-risk Bellman equation within the SOC framework. 
Specifically, we show that the reformulation takes the form of a convex combination of the expected cost and the Conditional Value at Risk (CVaR) of the cost. This result provides a risk-averse interpretation of the robust control formulation under the Bayesian ambiguity set \eqref{ambiguity}.

\begin{theorem}\label{thm-reformulation}
	Let ${P}^l_N$ and ${P}^{u-l}_N$ represent the probability measures induced by the lower and upper bounds $\hat{P}^l_N=\max\{0,\hat{P}_N-\hat{r}_N\}$ and $\hat{P}^u_N=\min\{1,\hat{P}_N+\hat{r}_N\}$ with respect to $\hat{P}_N$ and $\hat{r}_N$ as
\[p^l_{j,N}\equiv{P}^l_N(\xi= \xi^{j})=\frac{\hat{p}_{j,N}^l}{\mathbf{L}},\ p^{u-l}_{j,N}\equiv{P}^{{u-l}}_N(\xi= \xi^{j})=\frac{\hat{p}_{j,N}^u-\hat{p}_{j,N}^l}{\mathbf{U}-\mathbf{L}}\]
for \(j = 1, \ldots, J\), where
$\mathbf{L}=\sum_{k=1}^J\hat{p}_{k,N}^l\in(0,1)$ and $\mathbf{U}=\sum_{k=1}^J\hat{p}_{k,N}^u>1$.
Then, for any $s\in\mathcal{S}$ and $a\in\mathcal{A}$,
\[\begin{split}
	&
 \sup_{{P}\in\mathfrak{P}(\hat{P}_{N},\hat{r}_N)}\mathbb{E}_{P}\left[\mathcal{C}(s,a,\xi)+\gamma V_N(g(s,a,\xi))\right]\\&=\mathbf{L}\mathbb{E}_{{P}^l_N}\left[\mathcal{C}(s,a,\xi)+\gamma V_N(g(s,a,\xi))\right]+(1-\mathbf{L}){\rm{CVaR}}^{(\mathbf{U}-1)/(\mathbf{U}-\mathbf{L})}_{{P}_N^{u-l}}\left[\mathcal{C}(s,a,\xi)+\gamma V_N(g(s,a,\xi))\right].
\end{split}\]
\end{theorem}

\begin{proof}
The result is similar to \cite[Theorem 5]{jiang2018risk} with $L_\infty$ norm ambiguity sets. For completeness, we provide a proof for the more general box-type ambiguity set considered here.
Without loss of generality, by relabeling the support points if necessary, we may assume that
$$\mathcal{C}(s,a,\xi^i)+\gamma V_N(g(s,a,\xi^i))\leq \mathcal{C}(s,a,\xi^j)+\gamma V_N(g(s,a,\xi^j))\text{ for all }i\leq j\text{ and }i,j \in\{1,\ldots,J\}.$$
Due to the finiteness of the support set, $\sup_{{P}\in\mathfrak{P}(\hat{P}_{N},\hat{r}_N)}\mathbb{E}_{P}\left[\mathcal{C}(s,a,\xi)+\gamma V_N(g(s,a,\xi))\right]$ is equal to the optimal value of the following optimization problem:
	\begin{equation}
		\begin{split}
			\sup_{P\in\mathbb{R}^J}&\sum_{j=1}^Jp_{j} h(\xi^j)\\
			\text{s.t.}\ & \sum_{j=1}^Jp_j=1, \ \hat{p}_{j,N}^l\leq p_j\leq \hat{p}_{j,N}^u,\ j=1,\ldots,J,
		\end{split}\label{3.4}
	\end{equation}
where $h(\xi^j):=\mathcal{C}(s,a,\xi^j)+\gamma V_N(g(s,a,\xi^j))$,
and
${P}\in\mathfrak{P}(\hat{P}_{N},\hat{r}_N)$ is described by the constraints in \eqref{3.4}. The Lagrangian dual problem corresponding to problem \eqref{3.4} is
	\[\inf_{d^u,d^l\geq0,d^0}\sup_{P\in\mathbb{R}^J}\sum_{j=1}^J
 h(\xi^j)
p_j+\sum_{j=1}^Jd^u_j(\hat{p}_{j,N}^u-p_j)+\sum_{j=1}^Jd_j^l(p_j-\hat{p}_{j,N}^l)+d^0\left(1-\sum_{j=1}^Jp_j\right),\]
	where $d^u\geq0$, $d^l\geq 0$ and $d^0$ are the dual variables corresponding to the upper and lower bounded constraints and the equality constraint, respectively. 
 Thus we have
	\begin{align*}	 			&\inf_{d^u,d^l\geq0,d^0}\sup_{P\in\mathbb{R}^J}\sum_{j=1}^Jh(\xi^j)p_j+\sum_{j=1}^Jd^u_j(\hat{p}_{j,N}^u-p_j)+\sum_{j=1}^Jd_j^l(p_j-\hat{p}_{j,N}^l)+d^0\left(1-\sum_{j=1}^Jp_j\right)\\
		=& \inf_{d^u,d^l\geq0,d^0}d^0+\sum_{j=1}^J(d^u_j\hat{p}_{j,N}^u-d_j^l\hat{p}_{j,N}^l)+\sum_{j=1}^J\sup_{p_j\in\mathbb{R}}\left(h(\xi^j)+d_j^l-d^u_j-d^0\right)p_j\\
		=&\inf_{d^u,d^l\geq0,d^0}d^0+\sum_{j=1}^J(d^u_j\hat{p}_{j,N}^u-d_j^l\hat{p}_{j,N}^l)\\
		&\text{s.t.}\ h(\xi^j)+d_j^l-d^u_j-d^0=0,\ j=1,\ldots,J	\\
		=&\inf_{d^0}d^0+	\sum_{j=1}^J\inf_{d^u_j,d_j^l\geq0}\left\{d^u_j\hat{p}_{j,N}^u-d_j^l\hat{p}_{j,N}^l\right\}\\
		&\text{s.t.}\ h(\xi^j)+d_j^l-d^u_j-d^0=0,\ j=1,\ldots,J	\\
		=&\inf_{d^0}d^0+	\sum_{j=1}^{j^*}(h(\xi^j)-d^0)\hat{p}_{j,N}^l+\sum_{j=j^*+1}^J (h(\xi^j)-d^0)\hat{p}_{j,N}^u\\
		=&\inf_{d^0}d^0+	\sum_{j=1}^J (h(\xi^j)-d^0)\hat{p}_{j,N}^l+\sum_{j=j^*+1}^J (h(\xi^j)-d^0)(\hat{p}_{j,N}^u-\hat{p}_{j,N}^l)\\
 =&\mathbf{L}\mathbb{E}_{{P}^l_N}\left[h(\xi)\right]+(1-\mathbf{L})\inf_{d^0}\left\{d^0+\frac{\mathbf{U}-\mathbf{L}}{1-\mathbf{L}}\mathbb{E}_{{P}_N^{u-l}}\left[(h(\xi)-d^0)^+\right]\right\}\\			
 =&\mathbf{L}\mathbb{E}_{{P}^l_N}[h(\xi)]+(1-\mathbf{L}){\rm{CVaR}}^{(\mathbf{U}-1)/(\mathbf{U}-\mathbf{L})}_{{P}_N^{u-l}}[h(\xi)],
	\end{align*}
 where $[x]^+ := \max\{x,0\}$, 
 $j^*:= \max\{j\in\{1,\ldots,J\} :h(\xi^j)\leq d^0 \}$ if $h(\xi^1)\leq d^0 $, and $j^*:= 0$ otherwise. The last equality follows from the standard minimization representation of CVaR in \cite{rockafellar2000optimization}.
\end{proof}

Theorem \ref{thm-reformulation} shows that the Bellman equation for episodic Bayesian DROC admits a risk-averse SOC interpretation in terms of a convex combination of an expectation (under ${P}^l_N$) and CVaR (under ${P}^{u-l}_N$).
Define the coherent risk measure $\rho_N(\cdot):=\lambda_N\mathbb{E}_{{P}^l_N}\left[\cdot\right]+(1-\lambda_N){\rm{CVaR}}^{\upsilon_N}_{{P}^{u-l}_N}\left[\cdot\right]$, where $\lambda_N=\mathbf{L}$ and $\upsilon_N=\frac{\mathbf{U}-1}{\mathbf{U}-\mathbf{L}}$. Then the Bayesian distributionally robust Bellman equation \eqref{drmdp} can be rewritten as
\begin{equation}\label{drmdp2}
	V_N(s)=\inf_{a\in\mathcal{A}}\sup_{{P}\in\mathfrak{P}(\hat{P}_{N},\hat{r}_N)}\mathbb{E}_{P}\left[\mathcal{C}(s,a,\xi)+\gamma V_N(g(s,a,\xi))\right]=\inf_{a\in\mathcal{A}}\rho_N\left(\mathcal{C}(s,a,\xi)+\gamma V_N(g(s,a,\xi))\right).
\end{equation}
Consequently, the corresponding optimal policy \eqref{drmdp-policy} can also be determined by
\begin{equation}
	\hat{\pi}_N^*(s) \in \arg\inf_{a \in \mathcal{A}} \rho_{N}[\mathcal{C}(s,a,\xi) + \gamma \hat{V}_N^{*}(g(s,a,\xi))].
\end{equation}

With the above reformulation, we now outline the solution approach for episodic Bayesian DROC.
We extend the episodic adaptation scheme in \cite{shapiro2025episodic} to our distributionally robust setting.
Since the Bayesian distributionally robust solution serves as an approximation of the true SOC model, we apply the policy $\hat{\pi}^*_{N}$ for the next single episode.
At the end of the episode, a new realization $\hat{\xi}_{N+1}$ is observed and aggregated into the dataset $\vec{\hat{\xi}}_{N+1}=(\vec{\hat{\xi}}_{N},\hat{\xi}_{N+1})$. 
This observation updates the posterior distribution to $\mu_{N+1}$, which in turn refines the ambiguity set parameters $\hat{P}_{N+1}$ and $\hat{r}_{N+1}$ for the subsequent episode. 
The policy is then dynamically redetermined by solving the Bellman equation \eqref{drmdp2} with the updated parameters. Repeating this process produces an adaptive policy sequence $\hat{\pi}^* = \{\hat{\pi}_1^*(\cdot), \hat{\pi}_2^*(\cdot), \ldots\}$.
At the outset of the initial episode, we assume access to historical data $\vec{\hat{\xi}}_0$. In scenarios where no such data exist, the algorithm instead initializes purely from the noninformative prior distribution. This procedure is summarized in Algorithm~\ref{alg:A}.
\begin{algorithm}[h]
	\caption{Episodic Bayesian DROC}
	\begin{algorithmic}[1]\label{alg:A}
		\STATE Input: initial state $s_0$; initial prior distribution $\mu_0(P)$ computed using a historical batch of data $\vec{\hat{\xi}}_0$.
		\FOR{Episode $N = 0,1, 2, \ldots$}
		\STATE Construct the Bayesian ambiguity set \eqref{ambiguity} corresponding to the current belief $\mu_{N}$.
		\STATE Solve the Bellman equation \eqref{drmdp2} to obtain value function $\hat{V}^{*}_N(\cdot)$ and the corresponding optimal policy $\hat{\pi}_N^*$. 
		\STATE Execute action $a_N:=\hat{\pi}_{N}^*(s_{N})$ and transition $s_{N+1} = g(s_{N},a_N, \hat{\xi}_{N+1})$.
		\STATE Observe the new realization $ \hat{\xi}_{N+1}$ and update the sample set $\vec{\hat{\xi}}_{N+1}=(\vec{\hat{\xi}}_{N},\hat{\xi}_{N+1})$ and the posterior distribution according to
\begin{equation}
\mu(P|\vec{\hat{\xi}}_{N+1})=B(\tau_{N+1})^{-1}\prod_{j=1}^{J}{p}_j^{\tau_{j,N+1}-1},
\end{equation}
where
\[
\tau_{j,N+1} =
\begin{cases}
	\tau_{j,N} & \text{if } \hat{\xi}_{N+1}\neq\xi^j, \\
	\tau_{j,N}+1 & \text{if } \hat{\xi}_{N+1}=\xi^j.
\end{cases}
\]
\STATE Treat the updated posterior $\mu(P|\vec{\hat{\xi}}_{N+1})$ as the new prior $\mu_{N+1}$ for the next episode.
		\ENDFOR
		\RETURN Optimal policy sequence $\hat{\pi}^* = \{\hat{\pi}_1^*(\cdot), \hat{\pi}_2^*(\cdot), \ldots\}$.
	\end{algorithmic}
\end{algorithm}
\begin{remark}
We implicitly assume unit-length episodes where each policy implementation corresponds to a single decision step, with each step yielding one observation. However, the framework can be readily extended to longer episodes: for any specified positive integer $K$, the selected policy remains fixed throughout a $K$-step episode. During such extended periods, $K$ observations are aggregated to update the posterior distribution, followed by a subsequent resolution of the Bellman equation under the revised uncertainty quantification.
\end{remark}

It is important to note that Algorithm~\ref{alg:A} is only a conceptual framework: solving \eqref{drmdp2} exactly in Step~4 is generally intractable, and the procedure is written as an open-ended loop without an explicit stopping criterion. In Section~5, we will propose both a practical algorithm for approximately solving the Bellman equation~\eqref{drmdp2} and an associated termination rule.

\section{Convergence Analysis}
In this section, we analyze how the sequence of value functions and optimal policies evolves episode by episode.
We establish asymptotic convergence as the number of episodes grows and provide finite-sample results for the proposed episodic Bayesian DROC approach, which are essential in practice since real-world datasets are typically finite.

\subsection{Weighted supremum norm setting for unbounded costs}
The classical contraction analysis for discounted SOC/MDPs is often developed on the space of bounded value functions equipped with the sup-norm $\|\cdot\|_\infty$.
However, when the state space is unbounded and the one-stage cost grows with the state, the optimal value function may fail to be bounded and $\|V\|_\infty$ can be infinite.
To accommodate such unbounded-cost settings, we work in a {\em weighted supremum norm space} defined as follows.

Let $w:\mathcal S\to[1,\infty)$ be a continuous weight function and define
\[
\mathscr{V}
(\mathcal S)
:=\Big\{V:\mathcal S\to\mathbb R\ \big|\ \|V\|_w<\infty\Big\},
\quad
\|V\|_w:=\sup_{s\in\mathcal S}\frac{|V(s)|}{w(s)}.
\]
Then $(\mathscr{V}(\mathcal S),\|\cdot\|_w)$ is a Banach space (see, e.g., \cite{bertsekas2022abstract}). When $w\equiv 1$, $\|\cdot\|_w$ reduces to the standard sup-norm $\|\cdot\|_\infty$, and the bounded-cost setting is recovered as a special case.
We impose the following standard assumption.

\begin{assumption}\label{ass1}
The action set $\mathcal A$ is compact.
For each $\xi\in\Xi$, $\mathcal C(s,a,\xi)$ and $g(s,a,\xi)$ are continuous in $(s,a)$.
Moreover, there exist constants $\bar{\mathcal C}_w>0$ and $\kappa_w\ge 0$ such that for all $(s,a,\xi)\in\mathcal S\times\mathcal A\times\Xi$,
\begin{equation}\label{eq:growth_cost}
|\mathcal C(s,a,\xi)|\le \bar{\mathcal C}_w w(s),
\end{equation}
and
\begin{equation}\label{eq:drift_w}
w\big(g(s,a,\xi)\big)\le \kappa_w w(s).
\end{equation}
Finally, the discount factor $\gamma$ satisfies
\begin{equation}\label{eq:gamma_kappa}
\gamma\kappa_w<1.
\end{equation}
\end{assumption}

Assumption~\ref{ass1} is standard for discounted Markov decision processes with unbounded costs under a weighted supremum norm.
Condition~\eqref{eq:growth_cost} controls the growth of the one-stage cost relative to the weight function $w(\cdot)$, while
\eqref{eq:drift_w} is a (pointwise) drift condition ensuring that the weighted supremum norm remains stable under transitions.
Together with \eqref{eq:gamma_kappa}, these conditions imply that the Bellman operators are contractions under $\|\cdot\|_w$.
This weighted supremum norm approach is classical in dynamic programming and Markov control with unbounded rewards/costs; see, e.g., the early development in \cite{wessels1977markov} and the monograph treatments in \cite{hernandez2012further,bertsekas2022abstract}.

\subsection{Asymptotic convergence of value function and optimal policy}
Recall that \eqref{drmdp-policy} in Section 2.2 assumes the existence of a solution $\hat{V}^*_N$ to the Bellman equation \eqref{drmdp}. However, the underlying rationale has not yet been fully established. In the following, we first demonstrate the existence and uniqueness of $\hat{V}^*_N$.

We define the Bellman operators \(\mathcal{L}\), \(\hat{\mathcal{L}}_N, \text{ and }\mathcal{L}^\pi :\mathscr{V}(\mathcal S) \to\mathscr{V}(\mathcal S)\) respectively as
\begin{equation}\label{eq-sa-true}
	\mathcal{L}(V)(s) := \inf_{a \in \mathcal{A}} \mathbb{E}_{P^c} \left[ \mathcal{C}(s,a, \xi) + \gamma V (g(s, a, \xi))\right], \ \forall V \in \mathscr{V}(\mathcal S),
\end{equation}
\begin{equation}\label{eq-dro-true}
	\hat{\mathcal{L}}_N(V)(s) := \inf_{a \in \mathcal{A}} \rho_{N} \left[ \mathcal{C}(s,a, \xi) + \gamma V(g(s, a, \xi)) \right], \ \forall V \in \mathscr{V}(\mathcal S),
\end{equation}
\begin{equation}\label{eq-pi}
	\mathcal{L}^\pi(V)(s) :=\mathbb{E}_{P^c} \left[ \mathcal{C}(s,\pi(s), \xi) + \gamma V (g(s, \pi(s), \xi))\right], \ \forall V \in \mathscr{V}(\mathcal S).
\end{equation}

We now show that the above Bellman operators are all contraction mappings in weighted supremum norm $\|\cdot\|_w$ with factor $\gamma_w:=\gamma \kappa_w$. For any function $h:\Xi\to\mathbb R$, define $\|h\|_{\xi,\infty}:=\max_{\xi\in\Xi}|h(\xi)|$. 

\begin{lemma}[Contraction of Bellman operators under $\|\cdot\|_w$]\label{lem-beo-3}
The operators $\mathcal L$, $\hat{\mathcal{L}}_N$, and $\mathcal{L}^\pi$ are $\gamma_w$-contractions under the norm $\|\cdot\|_{w}$.
That is, for any $V, V^{\prime} \in \mathscr{V}(\mathcal{S})$,
\[
\left\|\hat{\mathcal{L}}_N V-\hat{\mathcal{L}}_N V^{\prime}\right\|_{w} \leq \gamma_w\left\|V-V^{\prime}\right\|_{w},
\quad
\left\|\mathcal{L}^\pi V-\mathcal{L}^\pi V^{\prime}\right\|_{w} \leq \gamma_w\left\|V-V^{\prime}\right\|_{w},
\quad
\left\|\mathcal{L} V-\mathcal{L} V^{\prime}\right\|_{w} \leq \gamma_w\left\|V-V^{\prime}\right\|_{w}.
\]
\end{lemma}

\begin{proof}
For any given $V,V'\in\mathscr{V}(\mathcal S)$ and $s\in\mathcal S$, by definition,
\[
\begin{aligned}
 \big|(\hat{\mathcal L}_N V)(s)-(\hat{\mathcal L}_N V')(s)\big|
\le &\sup_{a\in\mathcal A}
\Big|
\rho_N\big(\mathcal C(s,a,\xi)+\gamma V(g(s,a,\xi))\big)
-\rho_N\big(\mathcal C(s,a,\xi)+\gamma V'(g(s,a,\xi))\big)
\Big|
\\
\le& \sup_{a\in\mathcal A,{\xi\in\Xi}}\gamma |V(g(s,a,\xi))-V'(g(s,a,\xi))|.
\end{aligned}
\]
Here the last inequality comes from the 1-Lipschitz continuity of $\rho_N (\cdot)$. For each $\xi\in\Xi$,
\[
|V(g(s,a,\xi))-V'(g(s,a,\xi))|
\le \|V-V'\|_w w(g(s,a,\xi))
\le \kappa_w\|V-V'\|_w w(s),
\]
where the last inequality follows from \eqref{eq:drift_w}. Hence
\[
\big|(\hat{\mathcal L}_N V)(s)-(\hat{\mathcal L}_N V')(s)\big|
\le \gamma_w\|V-V'\|_w w(s).
\]
Dividing by $w(s)$ and taking supremum over $s$ yields the contraction for $\hat{\mathcal L}_N$.
The proofs for $\mathcal{L}$ and $\mathcal L^\pi$ are similar with $\rho_N$ replaced by $\mathbb E_{P^c}$.
\end{proof}

Since \(\mathcal{L}\), \(\hat{\mathcal{L}}_N\) and $\mathcal{L}^\pi$ are $\gamma_w$-contractions, they admit unique fixed points \(V^*\), \(\hat{V}_N^*\) and $V^{\pi}$, respectively.

\begin{lemma}\label{lem:V_weight_bounded}
Under Assumption~\ref{ass1}, for any $V\in\mathscr{V}(\mathcal S)$,
\[
\|\hat{\mathcal L}_N V\|_w \le \bar{\mathcal C}_w + \gamma\kappa_w\|V\|_w,
\quad
\|\mathcal L V\|_w \le \bar{\mathcal C}_w + \gamma\kappa_w\|V\|_w,
\quad
\|\mathcal L^\pi V\|_w \le \bar{\mathcal C}_w + \gamma\kappa_w\|V\|_w.
\]
In particular, the fixed points satisfy
\[
\|\hat V_N^*\|_w \le \frac{\bar{\mathcal C}_w}{1-\gamma\kappa_w},
\quad
\|V^*\|_w \le \frac{\bar{\mathcal C}_w}{1-\gamma\kappa_w},
\quad
\|V^\pi\|_w \le \frac{\bar{\mathcal C}_w}{1-\gamma\kappa_w}.
\]
\end{lemma}

\begin{proof}
Fix $V\in\mathscr{V}(\mathcal S)$ and $s\in\mathcal S$.
By Assumption~\ref{ass1}, $|\mathcal C(s,a,\xi)|\le \bar{\mathcal C}_w w(s)$ and
$|V(g(s,a,\xi))|\le \|V\|_w w(g(s,a,\xi))\le \kappa_w\|V\|_w w(s)$.
Hence for any $a\in\mathcal A$,
\[
|\mathcal C(s,a,\xi)+\gamma V(g(s,a,\xi))|
\le \big(\bar{\mathcal C}_w+\gamma\kappa_w\|V\|_w\big) w(s).
\]
Since $\rho_N$ is monotone and positively homogeneous, we obtain
\[
|(\hat{\mathcal L}_N V)(s)|
\le \sup_{a\in\mathcal A}\left|\rho_N\big(\mathcal C(s,a,\xi)+\gamma V(g(s,a,\xi))\big)\right|
\le \big(\bar{\mathcal C}_w+\gamma\kappa_w\|V\|_w\big) w(s),
\]
which implies $\|\hat{\mathcal L}_N V\|_w\le \bar{\mathcal C}_w+\gamma\kappa_w\|V\|_w$.
Finally, taking $V=\hat V_N^*$ and using $\hat V_N^*=\hat{\mathcal L}_N\hat V_N^*$ yields
$\|\hat V_N^*\|_w\le \bar{\mathcal C}_w/(1-\gamma\kappa_w)$. The bounds for $\mathcal L$ and $\mathcal L^\pi$ follow similarly with $\rho_N$ replaced by $\mathbb E_{P^c}$.
\end{proof}

Under Assumption~\ref{ass1}, since $\mathcal A$ is compact and $\cal C$ and $g$ are continuous in $(s,a)$ for each $\xi\in\Xi$, the Bellman operators preserve continuity; see, e.g., \cite[Chapter~6]{puterman2014markov}.
We are now ready to present the convergence results of the optimal value function $\hat{V}_N^*$ and policy $\hat{\pi}^*_N$.
\begin{theorem}\label{thm-bc1}
Suppose that Assumption~\ref{ass1} holds. Then the unique value function $\hat{V}_N^*$ of the episodic Bayesian DROC Bellman equation converges to the true optimal value function $V^*$ in the weighted supremum norm:
\[
\lim_{N\to\infty}\|\hat{V}_N^* - V^*\|_{w} = 0,
\quad a.s.
\]
\end{theorem}

\begin{proof}
Using the fixed point identities $\hat V_N^*=\hat{\mathcal L}_N\hat V_N^*$ and $V^*=\mathcal L V^*$,
\[
\|\hat V_N^*-V^*\|_w
=
\|\hat{\mathcal L}_N\hat V_N^*-\mathcal L V^*\|_w
\le
\|\hat{\mathcal L}_N\hat V_N^*-\hat{\mathcal L}_N V^*\|_w
+
\|\hat{\mathcal L}_N V^*-\mathcal L V^*\|_w.
\]
By Lemma~\ref{lem-beo-3},
\[
\|\hat{\mathcal L}_N\hat V_N^*-\hat{\mathcal L}_N V^*\|_w
\le \gamma_w\|\hat V_N^*-V^*\|_w.
\]
Therefore,
\begin{equation}\label{eq:weighted_master}
\|\hat V_N^*-V^*\|_w
\le
\frac{1}{1-\gamma_w} 
\|\hat{\mathcal L}_N V^*-\mathcal L V^*\|_w.
\end{equation}

It remains to show $\|\hat{\mathcal L}_N V^*-\mathcal L V^*\|_w\to 0$ a.s.
Define $C_{s,a}(\xi):=\mathcal C(s,a,\xi)+\gamma V^*(g(s,a,\xi))$.
By Lemma~\ref{lem:V_weight_bounded} and \eqref{eq:drift_w},
\[
|V^*(g(s,a,\xi))|
\le \|V^*\|_w w(g(s,a,\xi))
\le \kappa_w\|V^*\|_w w(s)
\le \frac{\kappa_w\bar{\mathcal C}_w}{1-\gamma_w} w(s),
\]
hence
\begin{equation}\label{eq:Csa_growth}
|C_{s,a}(\xi)|
\le \bar{\mathcal C}_w w(s)+\gamma\frac{\kappa_w\bar{\mathcal C}_w}{1-\gamma_w}w(s)
=
\frac{\bar{\mathcal C}_w}{1-\gamma_w} w(s).
\end{equation}

Using the mean--CVaR representation $\rho_N=\lambda_N\mathbb E_{P^l_N}+(1-\lambda_N)\mathrm{CVaR}^{\upsilon_N}_{P^{u-l}_N}$,
\[
\begin{aligned}
&\frac{1}{w(s)}\sup_{a\in\mathcal A}
\Big|
\rho_N\big(C_{s,a}(\xi)\big)-\mathbb E_{P^c}\big[C_{s,a}(\xi)\big]
\Big|\\
\le&
\frac{1}{w(s)}\sup_{a\in\mathcal A}
\left(
\Big|\lambda_N\mathbb E_{P^l_N}[C_{s,a}(\xi)]-\mathbb E_{P^c}[C_{s,a}(\xi)]\Big|
+
(1-\lambda_N)\big|\mathrm{CVaR}_{P_N^{u-l}}^{\upsilon_N}[C_{s,a}(\xi)]\big|
\right)\\
\le&
\frac{1}{w(s)}\sup_{a\in\mathcal A}
\left(
\|C_{s,a}\|_{\xi,\infty} \| \hat P_N^l-\hat P_N\|_1+\|C_{s,a}\|_{\xi,\infty} \| \hat P_N-P^c\|_1
+
(1-\lambda_N)\|C_{s,a}\|_{\xi,\infty}
\right).
\end{aligned}
\]
By \eqref{eq:Csa_growth}, $\|C_{s,a}\|_{\xi,\infty}\le \frac{\bar{\mathcal C}_w}{1-\gamma_w}w(s)$.
Moreover, from the definition of the box bounds we have $1-\lambda_N\le \sum_{j=1}^J\hat r_{j,N}$.
Thus
\[
\|\hat V_N^*-V^*\|_w\leq\frac{1}{1-\gamma_w}\|\hat{\mathcal L}_N V^*-\mathcal L V^*\|_w
\le
\frac{\bar{\mathcal C}_w}{(1-\gamma_w)^2}\left(
\|\hat P_N-P^c\|_1+2\sum_{j=1}^J\hat r_{j,N}
\right).
\]
By Lemma~\ref{lem-bc} and Proposition~\ref{prop-ambiguity}, $\hat P_N\to P^c$ and $\sum_{j=1}^J\hat r_{j,N}\to 0$ almost surely as $N\to\infty$. Thus we prove the claim.
\end{proof}

With Theorem~\ref{thm-bc1}, we can further show the convergence of $\hat{V}_N^*$ to $V^*$ in the probability sense.

\begin{proposition}\label{prop-N-finite}
Suppose Assumption~\ref{ass1} holds.
For any fixed $\varepsilon>0$ and $\delta\in(0,1)$, there exists $N(J,\varepsilon,\delta)>0$ such that for all $N\ge N(J,\varepsilon,\delta)$,
\bgeqn 
\mathbb P \left( \big\|\hat V_N^*-V^*\big\|_w<\varepsilon \right) \ge 1-\delta.
\edeqn 
\end{proposition}

\begin{proof}
By \eqref{eq:weighted_master} together with the bound on
$\|\hat{\mathcal L}_N V^*-\mathcal L V^*\|_w$ derived in the proof of
Theorem~\ref{thm-bc1}, and using Remark~\ref{rem-1}, we obtain 
\begin{equation}\label{eq:master}
\big\|\hat V_N^*-V^*\big\|_w
 \le \frac{\bar{\mathcal C}_w}{(1-\gamma_w)^2}\left(
 \omega_N \|\hat P_0-P^c\|_1 + (1-\omega_N) \|\hat P_N^{\mathrm E}-P^c\|_1+2\sum_{j=1}^J \hat r_{j,N}
\right).
\end{equation}
Let $N_1(J,\varepsilon) :=\Big(\textstyle\sum_{k=1}^{J}\tau_{k,0}-J\Big)\left(\frac{3 \bar{\mathcal C}_w}{(1-\gamma_w)^2 \varepsilon} \|\hat P_0-P^c\|_1-1\right)$. 
Since $\omega_N=\frac{\sum_{k=1}^{J}\tau_{k,0}-J}{\sum_{k=1}^{J}\tau_{k,0}-J+N}$, 
then 
for all $N \ge N_1(J,\varepsilon)$, 
we have 
\begin{equation}
\label{eq-prop-2-1}
 \frac{\bar{\mathcal C}_w}{(1-\gamma_w)^2} \omega_N \|\hat P_0-P^c\|_1 \le \varepsilon/3.
\end{equation}
For any $\eta\in(0,1)$, it is known from \cite[Proposition 6.6]{van1996weak} that
\[
\mathbb{P}\left( \|\hat P_N^{\mathrm E}-P^c\|_1 \le \sqrt{\frac{2\log(2^{J}/\eta)}{N}} \right) \ge 1-\eta.
\]
Taking $\eta=\delta/2$, with probability at least $1-\delta/2$,
\begin{equation}\label{eq:l1-emp-2}
\|\hat P_N^{\mathrm E}-P^c\|_1 \le \sqrt{\frac{2\log(2^{J+1}/\delta)}{N}}.
\end{equation}
On the event in \eqref{eq:l1-emp-2} which holds with probability at least $1-\delta/2$, when $N \ge \frac{18 \bar{\mathcal C}_w^{ 2}}{(1-\gamma_w)^4}\cdot\frac{(J+1)\log 2+\log(1/\delta)}{\varepsilon^{2}}
=: N_2(J,\varepsilon,\delta)$, we have $\sqrt{\frac{2\log(2^{J+1}/\delta)}{N}} \le \frac{\varepsilon(1-\gamma_w)^2}{3 \bar{\mathcal C}_w}$, which leads to
\begin{equation}\label{eq-prop-2-2}
\frac{\bar{\mathcal C}_w}{(1-\gamma_w)^2} (1-\omega_N) \|\hat P_N^{\mathrm E}-P^c\|_1 \le \varepsilon/3.
\end{equation}
Proposition~\ref{prop-ambiguity} implies
\begin{equation}\label{eq-rjn}
\sum_{j=1}^J\hat r_{j,N}
=
\frac{z_{1-\alpha'/2}}{\sqrt{\sum_{k=1}^J\tau_{k,0}-J+N}}
\sum_{j=1}^J\sqrt{\hat p_{j,N}(1-\hat p_{j,N})}.
\end{equation}
Since $\sum_{j=1}^J\hat p_{j,N}=1$, by the Cauchy--Schwarz inequality,
\[
\sum_{j=1}^J\sqrt{\hat p_{j,N}(1-\hat p_{j,N})}
\le
\sqrt{\Big(\sum_{j=1}^J\hat p_{j,N}\Big)\Big(\sum_{j=1}^J(1-\hat p_{j,N})\Big)}
=
\sqrt{J-1}.
\]
Hence
\begin{equation}
\sum_{j=1}^J\hat r_{j,N}
\le
\frac{z_{1-\alpha'/2}\sqrt{J-1}}{\sqrt{\sum_{k=1}^J\tau_{k,0}-J+N}}.
\end{equation}
In particular, if
\[
N \ge N_3(J,\varepsilon)
:=
\left(
\frac{6 \bar{\mathcal C}_w z_{1-\alpha'/2}\sqrt{J-1}}
{(1-\gamma_w)^2 \varepsilon}
\right)^2
-\Big(\sum_{k=1}^J\tau_{k,0}-J\Big),
\]
then
\[
\sum_{j=1}^J\hat r_{j,N}
\le
\frac{\varepsilon(1-\gamma_w)^2}{6 \bar{\mathcal C}_w},
\]
which leads to
\begin{equation}\label{eq-prop-2-3}
\frac{2\bar{\mathcal C}_w}{(1-\gamma_w)^2}\sum_{j=1}^J\hat r_{j,N}\le \varepsilon/3.
\end{equation}
With \eqref{eq-prop-2-1}, \eqref{eq-prop-2-2} and \eqref{eq-prop-2-3}, we conclude that if
\[
N \ge N(J,\varepsilon,\delta):=\max\{N_1, N_2, N_3\},
\]
then on the event in \eqref{eq:l1-emp-2} we have
\[
\big\|\hat V_N^*-V^*\big\|_w<\varepsilon.
\]
Since the event in \eqref{eq:l1-emp-2} holds with probability at least $1-\delta/2$, it follows a fortiori that
\[
\mathbb P \left( \big\|\hat V_N^*-V^*\big\|_w<\varepsilon \right) \ge 1-\delta.
\]
This completes the proof.
\end{proof}

\begin{theorem}[Convergence of the optimal policies]
\label{thm-bc-2}
Suppose that Assumption \ref{ass1} holds, then for any given $\bar{s} \in \mathcal{S}$, the set of optimal actions $\hat{\mathcal{A}}_N^{*}(\bar{s})$ of \eqref{drmdp} is nonempty. Moreover, the distance from the optimal action $\hat{\pi}_N^{*}(\bar{s})$ to the true optimal action set ${\mathcal{A}}^{*}(\bar{s})$ of \eqref{true} converges to zero almost surely as $N \to\infty$. 
In particular, 
if ${\mathcal{A}}^{*}(\bar{s})=\left\{\pi^{*}(\bar{s})\right\}$ is a singleton, then 
$\hat{\pi}_N^{*}(\bar{s})$ converges to $\pi^{*}(\bar{s})$ almost surely.
\end{theorem}
\begin{proof}

Fix $\bar s\in\mathcal S$ and write $\eta:=(\hat P,\hat r)$ with $\mathfrak P(\eta):=\mathfrak P(\hat P,\hat r)$. Define
\[
F(a,\eta):=\sup_{P\in\mathfrak P(\eta)}\mathbb E_P\left[\mathcal C(\bar s,a,\xi)+\gamma V^*\big(g(\bar s,a,\xi)\big)\right],
\ 
\Psi(\eta):=\arg\min_{a\in\mathcal A}F(a,\eta).
\]
By Lemma~\ref{lem-amb-cont}, $\eta\mapsto\mathfrak P(\eta)$ has nonempty, compact values and is Hausdorff continuous. Since the integrand is continuous in $(a,P)$, Berge’s maximum theorem \cite{aliprantis2006infinite} implies that $F$ is continuous in $(a,\eta)$ and $\Psi$ is nonempty, compact-valued, and upper semicontinuous. Moreover, for $\eta^*:=(P^c,0)$ we have $\Psi(\eta^*)={\mathcal{A}}^*(\bar s)$.
Let $\eta_N:=(\hat P_N,\hat r_N)$ and $\eta^*:=(P^c,0)$. By the risk-averse reformulation in Theorem~\ref{thm-reformulation},
\[
F(a,\eta_N)=\rho_N\left(\mathcal C(\bar s,a,\xi)+\gamma V^*\big(g(\bar s,a,\xi)\big)\right).
\]
Since $\rho_N$ is coherent risk measure, it is $1$-Lipschitz. Together with Theorem~\ref{thm-bc1},
\begin{equation}\label{eq-th3-1}
\sup_{a\in\mathcal{A}}\Big|\rho_N\big(\mathcal C(\bar s,a,\xi)+\gamma \hat V_N^*(g(\bar s,a,\xi))\big)
 -F(a,\eta_N)\Big|
\le \gamma_w w(\bar s)\|\hat V_N^*-V^*\|_w \to 0, \text{ a.s.}
\end{equation}

Now, fix an open neighborhood $U$ of ${\mathcal{A}}^*(\bar s)$. By the upper semicontinuity of $\Psi$ at $\eta^*$ and compactness of $\mathcal A$, there exist a neighborhood $B$ of $\eta^*$ and a constant $\delta_U>0$ such that for all $\eta\in B$,
\[
\inf_{a\in\mathcal A\setminus U}\big(F(a,\eta)-v(\eta)\big)\ge2\delta_U,
\ \text{where } v(\eta):=\min_{a\in\mathcal A}F(a,\eta).
\]
By Proposition~\ref{prop-ambiguity} and Lemma~\ref{lem-amb-cont}, $\eta_N\to\eta^*$ almost surely, thus we have $\eta_N\in B$ for all sufficiently large $N$. Combining this with \eqref{eq-th3-1}, for all $N$ large enough,
\[
\varepsilon_N:=\sup_{a\in\mathcal A}\Big|\rho_N\big[\mathcal C(\bar s,a,\xi)+\gamma \hat V_N^*(g(\bar s,a,\xi))\big]-F(a,\eta_N)\Big|<\delta_U.
\]
Thus for any $a\in\mathcal A\setminus U$,
\[
\begin{aligned}
\rho_N\big[\mathcal C(\bar s,a,\xi)+\gamma \hat V_N^*(g(\bar s,a,\xi))\big]
&\ge F(a,\eta_N)-\varepsilon_N \\
&\ge v(\eta_N)+2\delta_U-\varepsilon_N
> v(\eta_N)+\varepsilon_N \\
&\ge \min_{a'\in\mathcal A}\rho_N\big[\mathcal C(\bar s,a',\xi)+\gamma \hat V_N^*(g(\bar s,a',\xi))\big].
\end{aligned}
\]
Hence $\arg\min_{a\in\mathcal A}\rho_N[\mathcal C(\bar s,a,\xi)+\gamma \hat V_N^*(g(\bar s,a,\xi))]\subseteq U$ for all sufficiently large $N$. Since $U$ is arbitrary, it follows that, with $\dd_E$ denoting the Euclidean metric on $\cal A$,
\[
\dd_E\big(\hat\pi_N^*(\bar s),{\mathcal{A}}^*(\bar s)\big)\to0, \quad \text{a.s.}
\]
If ${\mathcal{A}}^*(\bar s)=\{\pi^*(\bar s)\}$ is a singleton, then taking $U:=\{a:\dd_E(a,\pi^*(\bar s))<\delta\}$ yields $\hat\pi_N^*(\bar s)\to\pi^*(\bar s)$ almost surely.
\end{proof}

Under mild conditions, the episodic Bayesian DROC value functions $\hat V_N^*$ and the associated policies $\hat\pi_N^*$ will
converge to the true value function and optimal policy of the SOC model under $P^c$. In practice, obtaining an infinite number of observations is infeasible. This raises the question of how the optimal policy derived from the episodic Bayesian DROC performs when applied in the true environment with a finite number of observations.
To address this, we analyze the finite-sample performance of the episodic optimal policy $\hat{\pi}^*_N$ in the following subsection. 

\subsection{Finite-sample posterior credibility guarantee for the associated policy value}
Recall that a potential drawback of the empirical SOC \eqref{SAA} is that its optimal policy may lack robustness when the actual problem instance deviates from that formulated with the historical observations.
This sensitivity to distribution shifts underscores the importance of establishing finite-sample guarantees, which provide statistical assurance that the obtained solution will maintain specified performance with high probability \cite{yang2020wasserstein}. While existing research has primarily focused on finite-sample guarantees for DRO problems (e.g., \cite{bertsimas2018robust,gao2023finite,mohajerin2018data}), it is important to extend such guarantees to SOC problems, particularly to DROC settings. In our Bayesian setting, the resulting finite-sample guarantee takes the form of a posterior credibility guarantee for the associated policy value.

In the following, we establish a finite-sample posterior credibility guarantee for the value of the optimal policy $\hat{\pi}^*_N$ associated with the episodic Bayesian DROC \eqref{drmdp}. Specifically, we show that the robust value function $\hat V_N^*$ serves as a posterior-credible conservative upper bound on the true value of the associated policy. In particular, once the posterior ambiguity set is constructed to have credibility level $(1-\alpha)$ for the one-step uncertainty model, the same credibility level is inherited by the infinite-horizon discounted problem, without any horizon-dependent adjustment. 
To this end, we first prove the monotonicity of the Bellman operator $\mathcal{L}^\pi$.
\begin{lemma}\label{lem-mon}
 Let $\mathcal{L}^\pi$ be defined as in 
 \eqref{eq-pi}, then $\mathcal{L}^\pi$ 
 is monotone, 
 i.e.,
 $V \geq V^{\prime}$ implies that $\mathcal{L}^\pi V \geq \mathcal{L}^\pi V^{\prime}$.
\end{lemma}
\begin{proof}
By definition
$$
\left(\mathcal{L}^\pi V\right)(s)=\mathbb{E}_{P^c}\left[\mathcal{C}(s, \pi(s), \xi)+\gamma V\left(g(s,\pi(s),\xi)\right)\right].
$$
If two value functions 
$V$ and $V'$
satisfy $V\left(g(s,\pi(s),\xi)\right)\geq V'\left(g(s,\pi(s),\xi)\right)$ for all $s\in\mathcal{S}$ and $\xi\in\Xi$, then 
we can deduce $\left(\mathcal{L}^\pi V\right)(s)\geq\left(\mathcal{L}^\pi V'\right)(s)$ by monotonicity of expectation.
\end{proof}
We are now ready to state the main posterior credibility result of this subsection.
\begin{theorem}[Finite-sample posterior credibility guarantee]
\label{thm-finite-sample-guarantee}
With respect to any fixed finite sample set $\vec{\hat{\xi}}_N$, we have
	\begin{equation}
	 	\mathbb{P}_{\mu_N}\left({\hat{V}^*_N\geq V^{\hat{\pi}^*_N}}\right)\geq 1-\alpha+o(1).\label{eq-finite-guarantee} 
	\end{equation}
	\end{theorem}

\begin{proof}
 Since $\hat{\pi}^*_N$ is the optimal policy corresponding to $\hat{V}^*_N$, then 
\bgeqn 
\hat{V}^*_N(s)=\hat{\mathcal{L}}_N(\hat{V}^*_N)(s) = \sup_{P\in\mathfrak{P}(\hat{P}_{N},\hat{r}_N)} \mathbb{E}_P\left[ \mathcal{C}(s,\hat{\pi}^*_N(s), \xi) + \gamma \hat{V}^*_N(g(s, \hat{\pi}^*_N(s), \xi)) \right]. 
\edeqn 
On the event $E_N:=\{P^c\in \mathfrak{P}(\hat{P}_{N},\hat{r}_N)\}$, which holds with posterior probability at least $1-\alpha+o(1)$ by Proposition \ref{prop-ambiguity}, we can deduce that $$\hat{V}^*_N(s)=\hat{\mathcal{L}}_N(\hat{V}^*_N)(s)\geq \mathbb{E}_{P^c}\left[\mathcal{C}(s,\hat{\pi}^*_N(s), \xi) + \gamma \hat{V}^*_N(g(s, \hat{\pi}^*_N(s), \xi))\right]=\mathcal{L}^{\hat{\pi}^*_N}(\hat{V}^*_N)(s)$$
for all $s\in\mathcal{S}$.
Thus, we have	
\begin{equation}\label{eq-thm-finite-sample-guarantee}
 \hat{V}^*_N(s)\geq \mathcal{L}^{\hat{\pi}^*_N}(\hat{V}^*_N)(s).
\end{equation}
By applying $\mathcal{L}^{\hat{\pi}^*_N}$ recursively to both sides of \eqref{eq-thm-finite-sample-guarantee} and invoking the monotonicity property from Lemma \ref{lem-mon}, we deduce that $\hat{V}^*_N \geq\left(\mathcal{L}^{\hat{\pi}^*_N}\right)^K \hat{V}^*_N$ for all $ K\in\mathbb{N}_+$.

Next we show that $(\mathcal{L}^{\hat{\pi}^*_N})^K \hat{V}^*_N$ converges to $V^{\hat{\pi}^*_N}$ as $K\to\infty$.
Under Assumption~\ref{ass1}, the operator $\mathcal{L}^{\hat{\pi}^*_N}$ is a $\gamma_w$-contraction in the weighted supremum norm $\|\cdot\|_w$,
hence
\[
\|(\mathcal{L}^{\hat{\pi}^*_N})^K \hat{V}^*_N - V^{\hat{\pi}^*_N}\|_w
\le \gamma_w^K \|\hat{V}^*_N - V^{\hat{\pi}^*_N}\|_w \to 0.
\]
In particular, for each fixed $s\in\mathcal S$,
\[
|(\mathcal{L}^{\hat{\pi}^*_N})^K \hat{V}^*_N(s) - V^{\hat{\pi}^*_N}(s)|
\le w(s) \|(\mathcal{L}^{\hat{\pi}^*_N})^K \hat{V}^*_N - V^{\hat{\pi}^*_N}\|_w \to 0.
\]
Taking $K\to\infty$ in $\hat{V}^*_N(s)\ge (\mathcal{L}^{\hat{\pi}^*_N})^K \hat{V}^*_N(s)$ yields
\[
\hat{V}^*_N(s)\ge V^{\hat{\pi}^*_N}(s),\quad \forall s\in\mathcal S,
\]
on the event $E_N$.
Therefore,
\[
\mathbb{P}_{\mu_N}\left(\hat{V}^*_N\geq V^{\hat{\pi}^*_N}\right)
\ge \mathbb{P}_{\mu_N}(E_N)
=\int_{P\in{\mathfrak{P}(\hat{P}_{N},\hat{r}_N)}} \mu(P|\vec{\hat{\xi}}_N) dP
\ge 1-\alpha+o(1),
\]
where the last inequality follows from Proposition~\ref{prop-ambiguity}.
\end{proof}
This finite-sample posterior credibility guarantee complements the asymptotic convergence results established above and provides a finite-data justification for the conservative nature of the proposed episodic Bayesian DROC formulation.

\section{Stability Analysis}
The existing research on SOC has largely focused on the case where 
the ambiguity sets are constructed from noise-free sample data. 
However, in practice, sample data may be 
perturbed due to measurement and/or recording errors 
and these errors 
may affect the quality of decisions, see 
e.g.~\cite{pichler2022quantitative,xu2021quantitative}.
In this section, we investigate how data contamination may affect ambiguity-set construction and, in turn, the reliability of the resulting optimal values and policies for episodic Bayesian DROC.
To the best of our knowledge, relatively few works
\cite{kern2020first,yang2025stability} have
specifically investigated the sensitivity and stability of optimal values in Markov decision models. There remains a notable lack of literature in the DROC domain that quantitatively characterizes the performance of optimal values and policies in the presence of sample perturbations.

Departing from the previous assumption, we now consider a contaminated-data setting in which the observed samples $\vec{\tilde{\xi}}_N:=(\tilde{\xi}_1, \ldots, \tilde{\xi}_N)$ are generated from a perturbed distribution \(\tilde P^c\), rather than from the true distribution \(P^c\). We regard $\tilde{\xi}_1, \ldots, \tilde{\xi}_N$ as perturbed samples deviated from $\hat{\xi}_1, \ldots, \hat{\xi}_N$, and treat $\tilde{P}^c$ as a deviation from $P^c$. For simplicity, we continue to treat $\tilde{\xi}_1, \ldots, \tilde{\xi}_N$ as i.i.d. samples, aligning with the methodology in \cite{pichler2022quantitative,xu2021quantitative}.
With respect to the Bayesian prior $\mu_0$, the nominal probability vectors corresponding to $\vec{\hat \xi}_{N}$ and $\vec{\tilde{\xi}}_N$ are denoted as $\hat{P}_{N}:=(\hat{p}_{j,N})_{j=1}^J$ and $\tilde{P}_{N}:=(\tilde{p}_{j,N})_{j=1}^J$,	with $$\hat{p}_{j,N}=\frac{\tau_{j,0}-1+\sum_{i=1}^N\mathds{1}_{\{\hat{\xi}_i=\xi^j\}}}{\sum_{k=1}^J\tau_{k,0}-J+N} \text{ and } \tilde{p}_{j,N}=\frac{\tau_{j,0}-1+\sum_{i=1}^N\mathds{1}_{\{\tilde{\xi}_i=\xi^j\}}}{\sum_{k=1}^J\tau_{k,0}-J+N},$$ respectively. 
The tolerance (radius) vectors $\hat{r}_N$ and $\tilde{r}_N$ are similarly defined as $$\hat{r}_{j,N}:=z_{1-\frac{\alpha'}{2}}\sqrt{\frac{\hat{p}_{j,N}(1-\hat{p}_{j,N})}{\sum_{k=1}^J\tau_{k,0}-J+N}}\text{ and } \tilde{r}_{j,N}:=z_{1-\frac{\alpha'}{2}}\sqrt{\frac{\tilde{p}_{j,N}(1-\tilde{p}_{j,N})}{\sum_{k=1}^J\tau_{k,0}-J+N}}.$$

\subsection{Quantitative stability with fixed samples}
We begin with stability analysis of how
perturbation of the ambiguity set may affect
the optimal value and optimal solutions of the episodic Bayesian DROC model by deriving some explicit error bounds.
In this part of the analysis, $\vec{\hat \xi}_{N}$ and $\vec{\tilde{\xi}}_N$ are fixed and so are
$\hat{P}_{N}$ and $\tilde{P}_{N}$.	
Let $\tilde{\mathcal{L}}_N$ be the counterpart Bellman operator built from $\mathfrak{P}(\tilde{P}_N,\tilde{r}_N)$:
\begin{equation}\label{eq-dro-perturb}
	\tilde{\mathcal{L}}_N(V)(s) := \inf_{a\in\mathcal{A}}\sup_{P\in\mathfrak{P}(\tilde{P}_N,\tilde{r}_N)}\mathbb{E}_{P}\left[\mathcal{C}(s,a,\xi)+\gamma V(g(s,a,\xi))\right], \quad \forall V \in \mathscr{V}(\mathcal S).
\end{equation}
Let $\tilde{V}_N^*$ denote the optimal value function of Bellman equation \eqref{eq-dro-perturb} and $\tilde{\mathcal{A}}_N^*(\bar{s})$ denote the corresponding optimal action set at state $\bar{s}$. 
We first consider the quantitative stability between problem \eqref{drmdp}
and its perturbed problem \eqref{eq-dro-perturb}
when the Bayesian ambiguity set $\mathfrak{P}(\hat{P}_N,\hat{r}_N)$ is replaced by $\mathfrak{P}(\tilde{P}_N,\tilde{r}_N)$. 
To this end, we introduce the appropriate metric we adopt for discussion as follows.

\begin{definition}[R{\"o}misch \cite{romisch2003stability}]
	A metric with $\zeta$-structure is defined as 
	\begin{equation}
	\dd_{\mathcal{G}}({P}, Q):=\sup _{g \in \mathcal{G}}\left|\mathbb{E}_{P}[g(\xi)]-\mathbb{E}_Q[g(\xi)]\right|,\label{dis}
	\end{equation}
	where $\mathcal{G}$ denotes a set of real-valued measurable functions. 
\end{definition}

The pseudo-metric \eqref{dis} encompasses a wide array of metrics within probability theory, including those outlined in \cite{guo2021statistical}. For instance, if we choose $\mathcal{G}$ as
\begin{equation}
	\mathcal{G}_{TV}:=\left\{g: \text{conv}\; \Xi \to\mathbb{R}: \sup_{\xi\in \text{conv}\; \Xi}|g(\xi)|\leq 1\right\},\label{TV}
\end{equation}
the metric $\dd_{\mathcal{G}}({P}, Q)$ reduces to the total variation metric, denoted as $\dd_{TV}({P}, Q)$.
Moreover, when the set of functions $\mathcal{G}$ is the class of all 1-Lipschitz mappings on $\Xi$ specified as
\begin{equation}
	\mathcal{G}_{\mathrm{Lip}}:=\left\{g: \text{conv}\; \Xi \to\mathbb{R}: |g(\xi)-g(\xi')|\leq\|\xi-\xi'\|, \forall \xi,\xi'\in \text{conv}\; \Xi\right\},\label{disl}
\end{equation}
the resulting metric $\dd_{\mathcal{G}}({P}, Q)$ 
recovers the classical Kantorovich metric, denoted by 
$\dd_{K}({P}, Q)$. 
To establish the desired quantitative stability, we derive an intermediate technical result which may be viewed as a generalization of the well-known H\"ormander formula.

\begin{lemma}[Generalized H\"ormander formula] 
\label{lema:G-Hormander}
Let ${\cal G}$ be a class of measurable functions defined over $\text{conv} \; \Xi$. Consider
\bgeq 
A_N:=\sup_{g\in {\cal G}}\left|\sup_{{P}\in\mathfrak{P}(\hat{P}_N,\hat{r}_N)} \mathbb{E}_{P}[g(\xi)] - \sup_{{P}\in\mathfrak{P}(\tilde{P}_N,\tilde{r}_N)} \mathbb{E}_{P}[g(\xi)]\right|.
\edeq 
Then
\bgeqn 
A_N \leq \mathbb{H}(\mathfrak{P}(\hat{P}_N,\hat{r}_N), \mathfrak{P}(\tilde{P}_N,\tilde{r}_N); \dd_{\mathcal{G}}).
\label{eq:gen-Hormander}
\edeqn 
In the case when ${\cal G}={\cal G}_{TV}$,
\bgeqn 
A_N \leq \sup_{g\in {\cal G}}\|g\|_{\xi,\infty}
\mathbb{H}(\mathfrak{P}(\hat{P}_N,\hat{r}_N), \mathfrak{P}(\tilde{P}_N,\tilde{r}_N);\dd_{TV}).
\edeqn 
Equality holds when 
${\cal G}/\sup_{g\in {\cal G}}\|g\|_{\xi,\infty}$ restricted to $\Xi$
is a unit ball of $\R^J$.
In the case when ${\cal G}$ is a class of Lipschitz functions with modulus being bounded by $L$,
\bgeqn 
A_N \leq L\mathbb{H}(\mathfrak{P}(\hat{P}_N,\hat{r}_N), \mathfrak{P}(\tilde{P}_N,\tilde{r}_N);\dd_{K}).
\edeqn 
Equality holds when 
${\cal G}/L$ consists of all Lipschitz continuous functions defined over $\text{conv}\; \Xi$ with modulus being bounded by $1$. 
\end{lemma}

\begin{proof}
Observe that
\bgeq 
\sup_{P\in\mathfrak P(\hat P_N,\hat r_N)}\mathbb E_P[g(\xi)]
-\sup_{P'\in\mathfrak P(\tilde P_N,\tilde r_N)}\mathbb E_{P'}[g(\xi)]
&=&
\sup_{P\in\mathfrak P(\hat P_N,\hat r_N)}
\inf_{P'\in\mathfrak P(\tilde P_N,\tilde r_N)}
\big(\mathbb E_P[g(\xi)]-\mathbb E_{P'}[g(\xi)]\big)\\
&\leq & \sup_{{P}\in\mathfrak{P}(\hat{P}_N,\hat{r}_N)}
\inf_{{P}'\in\mathfrak{P}(\tilde{P}_N,\tilde{r}_N)}
|\mathbb{E}_{P}[g(\xi)] - \mathbb{E}_{P'}[g(\xi)]|\\
&\leq & \sup_{{P}\in\mathfrak{P}(\hat{P}_N,\hat{r}_N)} \dd_{{\cal G}}(P,\mathfrak{P}(\tilde{P}_N,\tilde{r}_N))\\
&=& \mathbb{D}(
\mathfrak{P}(\hat{P}_N,\hat{r}_N)
,\mathfrak{P}(\tilde{P}_N,\tilde{r}_N);\dd_{\mathcal G}).
\edeq
By swapping the positions of the two ambiguity sets, 
and then taking supremum with respect to $g$ over ${\cal G}$, we immediately obtain \eqref{eq:gen-Hormander}. The rest of the conclusions follow directly from the definition of the pseudo-metric and H\"ormander formula (see e.g.,
Theorem II.~18 of \cite{castaing1977convex}). 
\end{proof}

We are now ready to state the first stability result 
on the optimal value.

\begin{theorem}[Quantitative stability of the optimal value]
\label{thm-stab-value}
Let $C^{N}_{s,a}(\xi):=\mathcal C(s,a,\xi)+\gamma \hat V_N^*(g(s,a,\xi))$ and 
\[\mathcal H_w:=\left\{\frac{C^{N}_{s,a}(\cdot)}{w(s)}:(s,a)\in\mathcal S\times\mathcal A\right\}.
\]
Under Assumption \ref{ass1}, 
\begin{equation}\label{eq:stab-value-w-abstract}
 \|\hat{V}_N^*-\tilde{V}_N^*\|_w
 \leq \frac{1}{1-\gamma_w}
 \mathbb{H}\left(\mathfrak{P}(\hat{P}_N,\hat{r}_N),\mathfrak{P}(\tilde{P}_N,\tilde{r}_N); \dd_{\mathcal{H}_w}\right).
\end{equation}
In particular, 
\begin{equation}\label{eq:stab-value-w-TV}
\begin{aligned}
 \|\hat{V}_N^*-\tilde{V}_N^*\|_w
\leq \frac{\bar{\mathcal C}_w}{(1-\gamma_w)^2}
\mathbb{H}\left(\mathfrak{P}(\hat{P}_N,\hat{r}_N),\mathfrak{P}(\tilde{P}_N,\tilde{r}_N); \dd_{TV}\right)\\
\le \frac{J\bar{\mathcal C}_w}{(1-\gamma_w)^2}
\big(\|\hat P_N-\tilde P_N\|_\infty+\|\hat r_N-\tilde r_N\|_\infty\big).
\end{aligned}
\end{equation}
\end{theorem}

\begin{proof}

Since $\hat{\mathcal L}_N$ and $\tilde{\mathcal L}_N$ are $\gamma_w$-contractions on $(\mathscr{V}(\mathcal S),\|\cdot\|_w)$, similar to \eqref{eq:weighted_master},
we have
\begin{equation}\label{eq:stab-step1}
\|\hat V_N^*-\tilde V_N^*\|_w
\le
\frac{1}{1-\gamma_w}\|\hat{\mathcal L}_N\hat V_N^*-\tilde{\mathcal L}_N\hat V_N^*\|_w.
\end{equation}
For any $s\in\mathcal S$,
\[
\begin{aligned}
&\frac{1}{w(s)}
\big|(\hat{\mathcal L}_N\hat V_N^*)(s)-(\tilde{\mathcal L}_N\hat V_N^*)(s)\big|\\
\le&
\sup_{a\in\mathcal A}
\frac{1}{w(s)}
\left|
\sup_{P\in\mathfrak P(\hat P_N,\hat r_N)}\mathbb E_P[C^N_{s,a}(\xi)]
-
\sup_{P\in\mathfrak P(\tilde P_N,\tilde r_N)}\mathbb E_P[C^N_{s,a}(\xi)]
\right|\\
\le&
\sup_{h\in\mathcal H_w}
\left|
\sup_{P\in\mathfrak P(\hat P_N,\hat r_N)}\mathbb E_P[h(\xi)]
-
\sup_{P\in\mathfrak P(\tilde P_N,\tilde r_N)}\mathbb E_P[h(\xi)]
\right|\\
&\le
\mathbb{H}\left(\mathfrak{P}(\hat{P}_N,\hat{r}_N),\mathfrak{P}(\tilde{P}_N,\tilde{r}_N); \dd_{\mathcal{H}_w}\right),
\end{aligned}
\]
where the last inequality follows from 
Lemma~\ref{lema:G-Hormander}.
Taking the supremum over $s$ 
gives
\[
\|\hat{\mathcal L}_N\hat V_N^*-\tilde{\mathcal L}_N\hat V_N^*\|_w
\le
\mathbb{H}\left(\mathfrak{P}(\hat{P}_N,\hat{r}_N),\mathfrak{P}(\tilde{P}_N,\tilde{r}_N); \dd_{\mathcal{H}_w}\right),
\]
and 
together with \eqref{eq:stab-step1} proves \eqref{eq:stab-value-w-abstract}.

Next, we show $\sup_{h\in\mathcal H_w}\|h\|_\infty\le \bar{\mathcal C}_w/(1-\gamma_w)$.
By Assumption~\ref{ass1},
$|\mathcal C(s,a,\xi)|\le \bar{\mathcal C}_w w(s)$ and
$w(g(s,a,\xi))\le \kappa_w w(s)$.
Moreover, since $\hat V_N^*$ is the unique fixed point of the $\gamma_w$-contraction $\hat{\mathcal L}_N$,
it satisfies $\|\hat V_N^*\|_w\le \bar{\mathcal C}_w/(1-\gamma_w)$.
Therefore,
\[
\frac{|C^N_{s,a}(\xi)|}{w(s)}
\le
\bar{\mathcal C}_w+\gamma\|\hat V_N^*\|_w\kappa_w
\le
\bar{\mathcal C}_w+\gamma_w\frac{\bar{\mathcal C}_w}{1-\gamma_w}
=
\frac{\bar{\mathcal C}_w}{1-\gamma_w},
\]
which implies $\sup_{h\in\mathcal H_w}\|h\|_\infty\le \bar{\mathcal C}_w/(1-\gamma_w)$.
We can renormalize $\mathcal{H}_w$ by multiplying each measurable $h\in\mathcal{H}_w$ by $(1-\gamma_w)/\bar{\mathcal C}_w$, ensuring the entire class of functions is uniformly bounded by 1.
Then \eqref{eq:stab-value-w-TV} follows by applying Lemma~\ref{lema:G-Hormander} and Lemma~\ref{lem-amb-cont}.
The proof is completed.
\end{proof}

When the stage cost is uniformly bounded, i.e.,
$\sup_{s\in\mathcal S,a\in\mathcal A,\xi\in\Xi}|\mathcal C(s,a,\xi)|\le \bar{\mathcal C}<\infty$,
one may choose $w (\cdot)\equiv 1$, so that $\|\cdot\|_w=\|\cdot\|_\infty$, $\bar {\mathcal{C}}_w=\bar {\mathcal{C}}$ and $\gamma_w=\gamma$.
In this bounded-cost case, the stability bound reduces to the classical sup-norm estimate.
Next, we discuss the stability of optimal policies.

\begin{theorem}\label{thm-sta-policy} 
Under Assumption \ref{ass1}, the following assertions hold.
\begin{itemize}
 
\item[(i)] For any $\epsilon>0$ and fixed $\bar{s}$, there exists $\delta> 0$ (depending on $\epsilon$ and $\bar{s}$) such that 
\[\mathbb{D}(\tilde{\mathcal{A}}_N^*(\bar{s}),\hat{\mathcal{A}}_N^*(\bar{s});\dd_E)<\epsilon\]
when $\mathbb{H}(\mathfrak{P}(\hat{P}_N,\hat{r}_N), \mathfrak{P}(\tilde{P}_N,\tilde{r}_N); \dd_{TV})<\delta.$\\

\item[(ii)] If, in addition, $\sup_{{P}\in\mathfrak{P}(\hat{P}_N,\hat{r}_N)} \mathbb{E}_{P}[C_{s,a}^N(\xi)]$ satisfies the second-order growth condition at $\hat{\mathcal{A}}_N^*$, that is, for any given $\bar{s}$, there exists a positive constant $\nu$ such that
\begin{equation}\label{eq-condition}
 \sup_{{P}\in\mathfrak{P}(\hat{P}_N,\hat{r}_N)} \mathbb{E}_{P}[C_{\bar{s},a}^N(\xi)]\geq \hat{V}_N^*(\bar{s})+\nu \dd_E(a,\hat{\mathcal{A}}_N^*(\bar{s}))^2, \forall a\in\mathcal{A},
\end{equation}
then
\bgeqn 
\mathbb{D}(\tilde{\mathcal{A}}_N^*(\bar{s}),\hat{\mathcal{A}}_N^*(\bar{s});\dd_E)\leq\sqrt{\frac{3\bar{\cal C}_w w(\bar s)}{\nu(1-\gamma_w)^2}\mathbb{H}(\mathfrak{P}(\hat{P}_N,\hat{r}_N), \mathfrak{P}(\tilde{P}_N,\tilde{r}_N); \dd_{TV})}.
\edeqn 
 
\end{itemize}
\end{theorem}

\begin{proof}
The arguments follow the spirit of Lemma 3.8 in \cite{liu2013stability}. The key difference is that the nominal and perturbed problems are characterized by different Bellman operators, hence their minimax objectives are evaluated over different ambiguity sets and with different unique optimal functions $\hat{V}^*_N$ and $\tilde{V}^*_N$.

Part (i). Let $\epsilon$ be a fixed small positive number and $\hat{V}^*_N$ be the optimal value function of \eqref{drmdp}. For any fixed $\bar{s}\in\mathcal{S}$, define
\bgeqn 
R_{\bar{s}}(\epsilon):=\inf _{\left\{a \in \mathcal{A}, \dd_E\left(a, \hat{\mathcal{A}}_N^*(\bar{s})\right) \geq \epsilon\right\}} \sup_{{P}\in\mathfrak{P}(\hat{P}_N,\hat{r}_N)} \mathbb{E}_{P}[C_{\bar{s},a}^N(\xi)]-\hat{V}_N^*(\bar{s}) .
\edeqn 
Then $R_{\bar{s}}(\epsilon)>0$. Let $\delta:=\frac{(1-\gamma_w)^2}{ 3\bar{\mathcal{C}}_w w(\bar s)}R_{\bar{s}}(\epsilon) $ and $\mathbb{H}(\mathfrak{P}(\hat{P}_N,\hat{r}_N), \mathfrak{P}(\tilde{P}_N,\tilde{r}_N); \dd_{TV})<\delta$. Define $\tilde{C}^{N}_{s,a}(\xi):=\mathcal C(s,a,\xi)+\gamma \tilde V_N^*(g(s,a,\xi))$. 
By Lemma~\ref{lema:G-Hormander} and Theorem~\ref{thm-stab-value},
\[
\begin{aligned}
 & \sup_{a\in \mathcal{A}} \left|\sup_{{P}\in\mathfrak{P}(\hat{P}_N,\hat{r}_N)} \mathbb{E}_{P}[C_{\bar{s},a}^N(\xi)] - \sup_{{P}\in\mathfrak{P}(\tilde{P}_N,\tilde{r}_N)} \mathbb{E}_{P}[\tilde C_{\bar{s},a}^N(\xi)] \right|
 \\
 \leq& \sup_{a\in \mathcal{A}} \left|\sup_{{P}\in\mathfrak{P}(\hat{P}_N,\hat{r}_N)} \mathbb{E}_{P}[C_{\bar{s},a}^N(\xi)] -\sup_{{P}\in\mathfrak{P}(\tilde{P}_N,\tilde{r}_N)} \mathbb{E}_{P}[ C_{\bar{s},a}^N(\xi)] +\sup_{{P}\in\mathfrak{P}(\tilde{P}_N,\tilde{r}_N)} \mathbb{E}_{P}[ C_{\bar{s},a}^N(\xi)] - \sup_{{P}\in\mathfrak{P}(\tilde{P}_N,\tilde{r}_N)} \mathbb{E}_{P}[\tilde C_{\bar{s},a}^N(\xi)] \right|
 \\
 \leq &\frac{\bar{\mathcal{C}}_w w(\bar s)}{1-\gamma_w}\mathbb{H}(\mathfrak{P}(\hat{P}_N,\hat{r}_N), \mathfrak{P}(\tilde{P}_N,\tilde{r}_N); \dd_{TV})+\frac{\gamma_w\bar{\mathcal{C}}_w w(\bar s)}{(1-\gamma_w)^2}\mathbb{H}(\mathfrak{P}(\hat{P}_N,\hat{r}_N), \mathfrak{P}(\tilde{P}_N,\tilde{r}_N); \dd_{TV})<\frac{\bar{\mathcal{C}}_w w(\bar s)}{(1-\gamma_w)^2}\delta.
\end{aligned}
\]
Therefore, for any $a \in \mathcal{A}$ with $\dd_E\left(a,\hat{\mathcal{A}}_N^*(\bar{s})\right) \geq \epsilon$ and any fixed $a^* \in \hat{\mathcal{A}}_N^*(\bar{s})$, we have
$$
\begin{aligned}
 & \sup_{{P}\in\mathfrak{P}(\tilde{P}_N,\tilde{r}_N)} \mathbb{E}_{P}[\tilde C_{\bar{s},a}^N(\xi)] -\sup_{{P}\in\mathfrak{P}(\tilde{P}_N,\tilde{r}_N)} \mathbb{E}_{P}[\tilde C_{\bar{s},a^*}^N(\xi)] \\
 \geq &\sup_{{P}\in\mathfrak{P}(\hat{P}_N,\hat{r}_N)} \mathbb{E}_{P}[C_{\bar{s},a}^N(\xi)]-\sup_{{P}\in\mathfrak{P}(\hat{P}_N,\hat{r}_N)} \mathbb{E}_{P}[C_{\bar{s},a^*}^N(\xi)]-\frac{2\bar{\mathcal{C}}_w w(\bar s)}{(1-\gamma_w)^2} \delta \geq R_{\bar{s}}(\epsilon) / 3>0,
\end{aligned}
$$
which implies that $a$ is not an optimal solution to \eqref{eq-dro-perturb}. This is equivalent to $\dd_E\left(a, \hat{\mathcal{A}}_N^*(\bar{s})\right)<\epsilon$ for all $a\in \tilde{\mathcal{A}}_N^*(\bar{s})$, that is, $\mathbb{D}\left(\tilde{\mathcal{A}}_N^*(\bar{s}), \hat{\mathcal{A}}_N^*(\bar{s});\dd_E\right) \leq \epsilon$.

Part (ii). Recall that the condition \eqref{eq-condition} leads to $R_{\bar{s}}(\epsilon)\geq\nu\epsilon^2$. Let
\[\epsilon=\sqrt{\frac{3\bar{\cal C}_w w(\bar s)}{\nu(1-\gamma_w)^2}\mathbb{H}(\mathfrak{P}(\hat{P}_N,\hat{r}_N), \mathfrak{P}(\tilde{P}_N,\tilde{r}_N); \dd_{TV})},\]
we immediately arrive at the desired result.
\end{proof}

It can be seen that the second-order growth condition \eqref{eq-condition} for $\sup_{{P}\in\mathfrak{P}(\hat{P}_N,\hat{r}_N)} \mathbb{E}_{P}[C_{s,a}^N(\xi)]$ 
is critical for quantifying the stability of our episodic Bayesian DROC formulation. To obtain concrete, verifiable criteria ensuring this second-order growth, we assume that for each fixed $s$ and $\xi$, $\mathcal{C}(s,\cdot,\xi)$ is strongly convex in $a$ with modulus $\kappa_s(\xi)$. Then, under some mild conditions \cite{ma2024bayesian}, $\hat{V}_N^*(g(s,\cdot,\xi))$ is convex, thus $C_{s,a}^N(\xi)$ remains strongly convex in $a$ with modulus $\kappa_s(\xi)$. Hence, there exist integrable functions $\eta_s(\xi)$ and $\kappa_s(\xi)$ such that for any fixed $a\in \mathcal{A}$,
$$
C_{s,a^{\prime}}^N(\xi) \geq C_{s,a}^N(\xi)+\eta_s(\xi)^{\top}\left(a^{\prime}-a\right)+\kappa_s(\xi)\left\|a^{\prime}-a\right\|^2, \quad \forall a^{\prime} \in\mathcal{A}, \xi \in \Xi,
$$
where $\kappa_s(\xi)$ is a positive function satisfying $\inf _{{P}\in\mathfrak{P}(\hat{P}_N,\hat{r}_N)} \mathbb{E}_{P}[\kappa_s(\xi)]>0$. Define the auxiliary function:
$$
\phi_{s,a}\left(a^{\prime}\right):=\sup _{{P}\in\mathfrak{P}(\hat{P}_N,\hat{r}_N)}\left(\mathbb{E}_{P}[C_{s,a}^N(\xi)]+\mathbb{E}_{P}[\eta_s(\xi)]^{\top}\left(a^{\prime}-a\right)\right)
$$
and let $v_s:=\inf _{{P}\in\mathfrak{P}(\hat{P}_N,\hat{r}_N)} \mathbb{E}_{P}[\kappa_s(\xi)]$. Then we obtain 
$$
\sup _{{P}\in\mathfrak{P}(\hat{P}_N,\hat{r}_N)} \mathbb{E}_{P}\left[C_{s,a^{\prime}}^N(\xi)\right] \geq \phi_{s,a}\left(a^{\prime}\right)+v_s\left\|a^{\prime}-a\right\|^2, \quad \forall a^{\prime} \in\mathcal{A} .
$$

Moreover, $\phi_{s,a}(\cdot)$ is convex and satisfies $\phi_{s,a}(a)=\sup _{{P}\in\mathfrak{P}(\hat{P}_N,\hat{r}_N)} \mathbb{E}_{P}[C_{s,a}^N(\xi)]$. Consequently, there exists some deterministic vector $\hat{\eta}$ (depending on $a$) such that
$$
\sup _{{P}\in\mathfrak{P}(\hat{P}_N,\hat{r}_N)} \mathbb{E}_{P}\left[C_{s,a^{\prime}}^N(\xi)\right] \geq \sup _{{P}\in\mathfrak{P}(\hat{P}_N,\hat{r}_N)} \mathbb{E}_{P}[C_{s,a}^N(\xi)]+\hat{\eta}^{\top}\left(a^{\prime}-a\right)+v_s\left\|a^{\prime}-a\right\|^2, \quad \forall a^{\prime} \in\mathcal{A} .
$$
This inequality holds uniformly for all $a$, demonstrating that $\sup _{{P}\in\mathfrak{P}(\hat{P}_N,\hat{r}_N)} \mathbb{E}_{P}\left[C_{s,a^{\prime}}^N(\xi)\right]$ is strongly convex, which implies that the optimal action set $\hat{\mathcal{A}}_N^*(s)$ reduces to a singleton $\left\{\hat{\pi}_N^*(s)\right\}$ for each fixed $s\in\mathcal{S}$. Then, the second-order growth condition \eqref{eq-condition} follows immediately from the above inequality by noting that $\sup _{{P}\in\mathfrak{P}(\hat{P}_N,\hat{r}_N)} \mathbb{E}_{P}\left[C_{s,\hat{\pi}_N^*(s)}^N(\xi)\right]=\hat{V}_N^*(s)$ and selecting $\hat{\eta}=0$ at the optimal policy $\hat{\pi}_N^*(s)$.

\subsection{Quantitative statistical robustness}
We now examine the statistical robustness property of the episodic Bayesian DROC model \eqref{drmdp}. 
The foundational framework for statistical robustness was originally established in \cite{hampel1971general}. Over the past few decades, this concept has gained significant attention, particularly through influential monographs such as \cite{huber2011robust}.
For a detailed discussion on how statistical robustness differs from traditional stability analysis, we refer readers to \cite{guo2021statistical}. 

The notion of statistical robustness can be elucidated by defining \((\Xi)^{\otimes N}\) as the Cartesian product \(\underbrace{ \Xi \times \ldots \times \Xi }_N\) and \(\mathcal{B}(\Xi)^{\otimes N}\) as its Borel $\sigma$-algebra. Let \((P^c)^{\otimes N}\) denote the probability measure on the measurable space \(((\Xi)^{\otimes N}, \mathcal{B}(\Xi)^{\otimes N})\) with marginal \(P^c\) on each \((\Xi ,\mathcal{B}(\Xi))\) and correspondingly \((\tilde{P}^c)^{\otimes N}\) with marginal \(\tilde{P}^c\). 

For each fixed state $s$, we consider a statistical functional \(T^s(\cdot)\) that maps from a subset of \(\mathcal{M} \subseteq \mathscr{P}(\Xi)\times\mathbb{R}_+^J\) to \(\mathbb{R}\), where each input pair \((P,r)\) specifies the ambiguity set \(\mathfrak{P}(P,r)\).
For each \(N \in \mathbb{N}\), \(T_N^s(\hat{\xi}_1, \ldots, \hat{\xi}_N)\) represents \(T^s(\hat{P}_N,\hat{r}_N)\), where \((\hat{P}_N,\hat{r}_N)\) are constructed from the sample set \((\hat{\xi}_1,\ldots,\hat{\xi}_N)\).
Notice that \(T_N^s\) maps from \((\Xi)^{\otimes N}\) to \(\mathbb{R}\) and provides a robust estimator for \(T^s(P^c,0)\).
Our interest is whether \(T^s(\tilde{P}_{N},\tilde r_N)\) is close to \(T^s(\hat{P}_{N},\hat r_N)\) under some appropriate metric in terms of empirical probability distributions as the underlying samples vary.
Here \(T^s(\hat{P}_{N},\hat r_N)\) is interpreted as the corresponding statistical estimator in the absence of sample noise. If \(T^s(\tilde{P}_{N},\tilde r_N)\) closely approximates \(T^s(\hat{P}_{N},\hat r_N)\), 
it validates the use of \(T^s(\tilde{P}_{N},\tilde r_N)\) as a robust estimate of \(T^s(P^c,0)\), given the impracticality of obtaining \(T^s(\hat{P}_{N},\hat r_N)\) in real-world scenarios.
With the above setting, the value functions $\hat{V}^*_N$ and $\tilde{V}^*_N$ corresponding to state $s$ can be seen as two statistical estimators, which respectively correspond to the statistical functional $T^s_N(\cdot)$ under two series $(\hat{\xi}_1, \ldots, \hat{\xi}_N) $ and $(\tilde{\xi}_1, \ldots, \tilde{\xi}_N)$, of the optimal value of episodic Bayesian DROC, i.e., 
\(T^s_N(\hat{\xi}_1, \ldots, \hat{\xi}_N)=\hat{V}^*_N(s)\) and \(T^s_N(\tilde{\xi}_1, \ldots, \tilde{\xi}_N)=\tilde{V}^*_N(s)\).

Since our analysis is based on the setup of the weighted supremum norm (see Assumption~\ref{ass1}),
it is natural to normalize the statistical functionals by the weight:
\[
\bar T_N^s(\cdot):=\frac{T_N^s(\cdot)}{w(s)}
\quad\Big(\text{equivalently, } \bar T_N^s(\hat\xi_1,\ldots,\hat\xi_N)=\frac{\hat V_N^*(s)}{w(s)}\Big).
\]
This normalization allows us to derive Lipschitz properties with constants that do not scale with $w(s)$.
Under the framework, we are interested in the difference between {\em laws of the two estimators}, that is, the difference
between $(P^c)^{\otimes N} \circ (\bar T^s_N)^{-1}$ and $(\tilde{P}^c)^{\otimes N} \circ (\bar T^s_N)^{-1} $. This differs from the stability analysis where 
both the true and contaminated samples are fixed.

\begin{lemma}[Quantitative statistical robustness \cite{xu2021quantitative}]\label{lem-qsr}
Suppose a statistical functional $T_N:(\Xi)^{\otimes N}\to\mathbb{R}$ satisfies
	\begin{equation}
		\left| T_N(\hat{\xi}_1, \ldots, \hat{\xi}_N) - T_N(\tilde{\xi}_1, \ldots, \tilde{\xi}_N) \right| 
		\leq \frac{L}{N} \sum_{i=1}^{N} \left\| \hat{\xi}_i - \tilde{\xi}_i \right\|,
	\end{equation}
then for all \(P^c, \tilde{P}^c \in \mathscr{P}(\Xi)\) and \(N \in \mathbb{N}\), 
	\begin{equation}\label{eq-3.14}
		\dd_{K} \left( (P^c)^{\otimes N} \circ (T_N)^{-1}, (\tilde{P}^c)^{\otimes N} \circ (T_N)^{-1}\right) 
		\leq L \dd_{K}(P^c, \tilde{P}^c).
	\end{equation}
\end{lemma}
Lemma~\ref{lem-qsr} is a direct consequence of~\cite[Theorem~2.1]{xu2021quantitative}.
In the forthcoming discussions, we will use Lemma \ref{lem-qsr} as a template to present the quantitative statistical robustness of the optimal value of the episodic Bayesian DROC problem. The basic idea is to derive Lipschitz continuity of the ambiguity sets with respect to the change of sample data and subsequently demonstrate the Lipschitz continuity of the optimal value function (with the change of sample data).

\begin{lemma}[\cite{ma2024bayesian}, Proposition 9]\label{prop-lip}
Assume that $\mathcal{C}(s, a, \xi)$ and $g(s, a, \xi)$ are Lipschitz continuous in $s\in\cal S$ uniformly for all $(a,\xi) \in \mathcal{A} \times \Xi$, with Lipschitz constants $L_{\mathcal{C}}$ and $L_g$, respectively. If $\gamma L_g<1$,
then the value function $\hat{V}_N^*$ is Lipschitz continuous with a Lipschitz constant $L_{V}:=\frac{L_{\mathcal C}}{1-\gamma L_g}$.
\end{lemma}

\begin{theorem}[Quantitative statistical robustness under the weighted supremum norm]
\label{thm-qsr}
Suppose Assumption~\ref{ass1} holds and the conditions of Lemma~\ref{prop-lip} are satisfied.
Assume further that there exist constants $L_1,L_2>0$ such that for any $s\in\mathcal S$,
\bgeqn 
\label{eq:Lip-C}
\sup_{a \in \mathcal{A}} |\mathcal{C}(s,a, \xi) - \mathcal{C}(s,a, \xi')|
\leq L_1 \|\xi - \xi'\|,
\qquad \forall \xi, \xi' \in \Xi,
\edeqn 
and
\bgeqn 
\label{eq:Lip-g}
\sup_{a\in\mathcal A}\|g(s,a,\xi)-g(s,a,\xi')\|
\le L_{2}\|\xi-\xi'\|,
\qquad \forall \xi, \xi' \in \Xi.
\edeqn 
Then, for all $s\in\mathcal S$, there exists a constant $L>0$ independent of $N$ and $s$ such that
	\bgeqn 
	\dd_{K}\left((P^c)^{\otimes N} \circ (\bar T^s_N)^{-1}, (\tilde{P}^c)^{\otimes N} \circ (\bar T^s_N)^{-1}\right) \leq L \dd_{K}(P^c, \tilde{P}^c).
	\edeqn 
\end{theorem}

\begin{proof}
To demonstrate the Lipschitz continuity of $\bar T_N^s$, we first show that there exists a positive constant $L_0$ such that		
\begin{equation}\label{eq:qsr-lip-C}
\sup_{a\in\mathcal A}\frac{|C_{s,a}^N(\xi)-C_{s,a}^N(\xi')|}{w(s)}
\le L_0\|\xi-\xi'\|.
\end{equation}
By the definition of $C_{s,a}^N$ (see Theorem~\ref{thm-stab-value}),
		$$
	\begin{aligned}
			& 	\sup_{a \in \mathcal{A}} |C_{s,a}^N(\xi)-C_{s,a}^N(\xi')| \\
	=	& 	\sup_{a \in \mathcal{A}} |\mathcal{C}(s,a,\xi)+\gamma \hat{V}^*_N(g(s,a,\xi))- \mathcal{C}(s,a,\xi')-\gamma \hat{V}^*_N(g(s,a,\xi'))| \\
		\leq & 	\sup_{a \in \mathcal{A}} |\mathcal{C}(s,a,\xi)- \mathcal{C}(s,a,\xi')|+\gamma\sup_{a \in \mathcal{A}} |\hat{V}^*_N(g(s,a,\xi))-\hat{V}^*_N(g(s,a,\xi'))|
\end{aligned}
	$$
By \eqref{eq:Lip-C}, the first term at the right-hand side of the inequality is bounded by $L_1 \|\xi - \xi'\|$;
and by Lemma~\ref{prop-lip}, the second term is bounded by 
$\gamma L_V\sup_{a \in \mathcal{A}}\|g(s,a,\xi)-g(s,a,\xi')\|
$.
Consequently we obtain
\bgeqn 	\sup_{a \in \mathcal{A}} |C_{s,a}^N(\xi)-C_{s,a}^N(\xi')| 
 \leq (L_1+\gamma L_VL_2) \|\xi - \xi'\|.
\edeqn 
 Since $w(s)\geq1$, we can deduce \eqref{eq:qsr-lip-C} by setting $L_0=L_1+\gamma L_VL_2$. Recall that
 $\mathcal H_w:=\{\frac{C^{N}_{s,a}(\cdot)}{w(s)}:(s,a)\in\mathcal S\times\mathcal A\}$, we have $\frac{\mathcal{H}_w}{L_0} \subset \mathcal{G}_{\mathrm{Lip}}$, which implies
	$$
	\dd_{\mathcal{H}_w}\left(\hat{P}_{N}, \tilde{P}_{N}\right) \leq L_0 \sup _{h \in \mathcal{G}_{\mathrm{Lip}}}\left|\mathbb{E}_{\hat{P}_{N}}[h(\xi)]-\mathbb{E}_{\tilde{P}_{N}}[h(\xi)]\right|.	
	$$
Consequently,
\begin{align*}
&\|\hat{V}_N^* - \tilde{V}_N^*\|_w\\
\leq 	& \frac{1}{1-\gamma_w}\|\hat{\mathcal{L}}_N\hat{V}_N^*-\tilde{\mathcal{L}}_N\hat{V}_N^*\|_w\\
\leq	& \frac{1}{1-\gamma_w}\underbrace{\left\|\inf_{a\in\mathcal{A}}\sup_{{P}\in\mathfrak{P}(\hat{P}_N,\hat{r}_N)}\mathbb{E}_{P}\left[\mathcal{C}(s,a,\xi)+\gamma\hat{V}^*_N(g(s,a,\xi))\right] - \inf_{a\in\mathcal{A}}\sup_{{P}\in\mathfrak{P}(\tilde{P}_N,\hat{r}_N)}\mathbb{E}_{P}\left[\mathcal{C}(s,a,\xi)+\gamma\hat{V}^*_N(g(s,a,\xi))\right]\right\|_w}_{\dagger}\\
+	& \frac{1}{1-\gamma_w}\underbrace{\left\|\inf_{a\in\mathcal{A}}\sup_{{P}\in\mathfrak{P}(\tilde{P}_N,\hat{r}_N)}\mathbb{E}_{P}\left[\mathcal{C}(s,a,\xi)+\gamma\hat{V}^*_N(g(s,a,\xi))\right] - \inf_{a\in\mathcal{A}}\sup_{{P}\in\mathfrak{P}(\tilde{P}_N,\tilde{r}_N)}\mathbb{E}_{P}\left[\mathcal{C}(s,a,\xi)+\gamma\hat{V}^*_N(g(s,a,\xi))\right]\right\|_w}_{\dagger\dagger}.
	\end{align*}
 For the first part, we obtain by Lemma~\ref{lema:G-Hormander} that 
\begin{equation}
\begin{aligned}\label{dagger}
\dagger	\leq	& \mathbb{H}(\mathfrak{P}(\hat{P}_N,\hat{r}_N), \mathfrak{P}(\tilde{P}_N,\hat{r}_N); \dd_{\mathcal{H}_w})\\
		\leq& \dd_{\mathcal{H}_w}(\hat{P}_N, \tilde{P}_N) \\
	\leq	&L_0\sup_{h \in \mathcal{G}_{\mathrm{Lip}}} \left| \mathbb{E}_{\hat{P}_N}[h(\xi)] - \mathbb{E}_{\tilde{P}_N}[h(\xi)] \right| \\
		=&L_0\sup_{h \in \mathcal{G}_{\mathrm{Lip}}} \left| \frac{1}{\sum_{j=1}^J\tau_{j,0}-J+N} \sum_{i=1}^{N} h(\hat{\xi}_{i}) - \frac{1}{\sum_{j=1}^J\tau_{j,0}-J+N} \sum_{i=1}^{N} h(\tilde{\xi}_{i}) \right|\\
	\leq	&\frac{L_0}{(\sum_{j=1}^J\tau_{j,0}-J+N)}\sum_{i=1}^N\|\hat{\xi}_i-\tilde{\xi}_i\|
	\end{aligned} 
\end{equation}
	where the third inequality from the last is based on Theorem 1 in \cite{pichler2022quantitative}.
Likewise, by Lemma~\ref{lema:G-Hormander}, we have that
 \begin{align*}
\dagger\dagger
\leq	& \mathbb{H}(\mathfrak{P}(\tilde{P}_N,\tilde{r}_N), \mathfrak{P}(\tilde{P}_N,\hat{r}_N); \dd_{\mathcal{H}_w})
\leq L_0\mathbb{H}(\mathfrak{P}(\tilde{P}_N,\tilde{r}_N), \mathfrak{P}(\tilde{P}_N,\hat{r}_N);\dd_{\mathcal{G}_{\mathrm{Lip}}})
\leq L_0D_{\Xi}\sum_{j=1}^J|\tilde{r}_{j,N}-\hat{r}_{j,N}|,
\end{align*}
where $D_{\Xi}:=\max_{j\neq j'}\|\xi^j-\xi^{j'}\|$.
According to the constructions of the tolerances $\hat{r}_{j,N}$ and $\tilde{r}_{j,N}$, we have
$$
\begin{aligned}
\sum_{j=1}^J|\hat{r}_{j,N}-\tilde{r}_{j,N}|
=&\frac{z_{1-\frac{\alpha'}{2}}}{\sqrt{\sum_{k=1}^J\tau_{k,0}-J+N}}
\sum_{j=1}^J
\left|
\sqrt{\hat{p}_{j,N}(1-\hat{p}_{j,N})}
-
\sqrt{\tilde{p}_{j,N}(1-\tilde{p}_{j,N})}
\right|\\
\leq&
\frac{z_{1-\frac{\alpha'}{2}}}{\sqrt{\sum_{k=1}^J\tau_{k,0}-J+N}}
\sum_{j=1}^J
\sqrt{
\left|
\hat{p}_{j,N}(1-\hat{p}_{j,N})
-
\tilde{p}_{j,N}(1-\tilde{p}_{j,N})
\right|
}\\
\leq&
\frac{z_{1-\frac{\alpha'}{2}}}{\sqrt{\sum_{k=1}^J\tau_{k,0}-J+N}}
\sum_{j=1}^J
\sqrt{|\hat{p}_{j,N}-\tilde{p}_{j,N}|}.
\end{aligned}
$$
Moreover, for each $j=1,\ldots,J$, the quantity $|\hat p_{j,N}-\tilde p_{j,N}|$ is an integer multiple of
$\frac{1}{\sum_{k=1}^J\tau_{k,0}-J+N}$, and hence
\[
\frac{1}{\sqrt{\sum_{k=1}^J\tau_{k,0}-J+N}}
\sqrt{|\hat p_{j,N}-\tilde p_{j,N}|}
\le
|\hat p_{j,N}-\tilde p_{j,N}|.
\]
Therefore,
$$
\sum_{j=1}^J|\hat{r}_{j,N}-\tilde{r}_{j,N}|
\le
z_{1-\frac{\alpha'}{2}}
\sum_{j=1}^J|\hat{p}_{j,N}-\tilde{p}_{j,N}|
\le
z_{1-\frac{\alpha'}{2}}\dd_{TV}(\hat{P}_N, \tilde{P}_N) 
\le
\frac{z_{1-\frac{\alpha'}{2}}}{\delta_\Xi}\dd_{\mathcal{G}_{\mathrm{Lip}}}(\hat{P}_N, \tilde{P}_N),
$$
where $\delta_\Xi:=\min_{j\neq j'}\|\xi^j-\xi^{j'}\|$.
Thus
\begin{align}\label{dagger2}
\dagger\dagger
\leq&
\frac{D_\Xi L_0z_{1-\frac{\alpha'}{2}}}{\delta_\Xi}\dd_{\mathcal{G}_{\mathrm{Lip}}}(\hat{P}_N, \tilde{P}_N)
\leq
\frac{D_\Xi L_0z_{1-\frac{\alpha'}{2}}}{\delta_\Xi(\sum_{j=1}^J\tau_{j,0}-J+N)}\sum_{i=1}^N\|\hat{\xi}_i-\tilde{\xi}_i\|.
\end{align}

Combining with \eqref{dagger} and \eqref{dagger2}, we can draw the conclusion that
\[
\left|\frac{\hat{V}_N^*(s) - \tilde{V}_N^*(s)}{w(s)}\right|
\leq
\|\hat{V}_N^* - \tilde{V}_N^*\|_w
\leq
\frac{\delta_\Xi L_0 + D_\Xi L_0z_{1-\frac{\alpha'}{2}}}{\delta_\Xi(1-\gamma_w)(\sum_{j=1}^J\tau_{j,0}-J+N)}
\sum_{i=1}^N\left\|\hat{\xi}_i-\tilde{\xi}_i\right\|.
\]
Thus we have proved that $\bar T_N^s$ satisfies the Lipschitz condition in Lemma~\ref{lem-qsr}.
Applying Lemma~\ref{lem-qsr} with $T_N=\bar T_N^s$ yields the desired robustness bound.
\end{proof}

\begin{remark}
Theorem~\ref{thm-qsr} deals with the case when all sample data are potentially perturbed. 
In practice, 
sample data might come from 
mixed sources with varying levels of qualities. 
Our established theorem is also applicable to such settings.
To model this situation, we adopt Huber's $\varepsilon$-contamination model, where the observed data distribution is given by
\[
\tilde P^c = (1-\varepsilon)P^c+\varepsilon Q,\quad Q\in\mathscr P(\Xi),
\]
meaning that each sample is drawn from the clean distribution \( P^c \) with probability \( 1 - \varepsilon \), and from an arbitrary contaminating distribution \( Q \) with probability \( \varepsilon \). This models the case where high-quality (clean) data are mixed with corrupted or low-quality (bad) samples.
Under this model, the distance between \( P^c \) and \( \tilde P^c \) can be bounded in terms of the Kantorovich metric:
\[
\dd_{K}(P^c, \tilde P^c) = \dd_{K}(P^c, (1-\varepsilon)P^c + \varepsilon Q) \leq \varepsilon \dd_{K}(P^c, Q).
\]
This inequality highlights that the deviation between the clean and contaminated distributions scales linearly with the contamination level \( \varepsilon \), but it also reveals a potential challenge: even when \( \varepsilon \) is small, the product \( \varepsilon \dd_{K}(P^c, Q) \) may still be large if \( Q \) is significantly different from \( P^c \). Consequently, \( \tilde P^c \) may not represent a small perturbation of \( P^c \) under the Kantorovich metric.
Therefore, while our theorem provides a quantitative statistical robustness guarantee under contamination, the practical utility of the bound depends on the nature of the contaminating distribution \( Q \). When \( Q \) is adversarial or far from \( P^c \), the bound may become loose, especially under coarse metrics such as total variation. Nonetheless, for many common function classes \( \mathcal{G} \), especially when \( \Xi \) is finite, \( \dd_{\mathcal{G}}(P^c, Q) \) remains uniformly bounded, yielding an \( O(\varepsilon) \) guarantee even in adversarial settings.
\end{remark}

With Theorems \ref{thm-stab-value}-\ref{thm-qsr},
we can conclude that under appropriate regularity conditions, both the optimal value function and the corresponding policy obtained from the Bayesian distributionally robust Bellman equation maintain their reliability. This means the obtained solutions remain valid as long as the potential data contamination does not cause significant distribution shifts. Our theoretical analyses underscore the inherent robustness of the proposed method, affirming its reliability in real-world scenarios. Consequently, the proposed episodic Bayesian DROC framework emerges as a robust approach for decision-making problems subject to distributional uncertainty and sample contamination.

\section{Bellman-operator cutting-plane algorithm for episodic Bayesian DROC}
As noted at the end of Section 2, it is necessary to solve the distributionally robust Bellman equation \eqref{drmdp2} at each episode in Algorithm \ref{alg:A}, which is computationally challenging and raises tractability issues. To address this challenge, we develop a Bellman-operator cutting-plane (BOCP) method that constructs supporting cuts for the Bellman operator $\hat{\mathcal{L}}_N$.
The method performs operator-level, monotone lower-envelope updates and yields an approximate value iteration with uniform convergence on $\mathcal S$ under suitable assumptions. 
We use $\partial$ for Moreau--Rockafellar subdifferential \cite{rockafellar2015convex} and \(\partial_s\), \(\partial_a\) for block subdifferentials with respect to \(s\) and \(a\), respectively.
To ensure convexity of the value function associated with the Bellman equation \eqref{drmdp2}, we impose the following structural assumption.

\begin{assumption}\label{ass-sddp}
(i) The state space $\mathcal S\subset\mathbb R^m$ and action space $\mathcal{A}\subset\mathbb R^n$ are convex and compact. 
(ii) For each \(\xi\in\Xi\), the cost function $\mathcal{C}(s, a, \xi)$ is convex and continuous in $(s, a) \in \mathcal{S} \times \mathcal{A}$. (iii) The state transition mapping $g(s, a, \xi)$ is affine, i.e.,
$$
g(s, a, \xi)=A(\xi) s+B(\xi) a+b(\xi),
$$
where $A:\Xi\to\mathbb{R}^{m\times m}$, $B:\Xi\to\mathbb{R}^{m\times n}$ and $b:\Xi\to\mathbb{R}^{m}$ are measurable mappings.
\end{assumption}

In this section, for computational purposes we specialize to compact convex $\cal S$ and $\cal A$, so that the stage cost is bounded, i.e., $\bar{\mathcal C}:=\sup_{s\in\mathcal S,\ a\in\mathcal A,\ \xi\in\Xi}|\mathcal C(s,a,\xi)|<\infty$ and we work with the sup-norm (i.e., $w\equiv 1$ and $\gamma_w=\gamma$).
The affine transition structure is standard in the SOC literature; see, e.g., \cite{abeille2018improved,taskesen2023distributionally}.
Under Assumption~\ref{ass-sddp}, the Bellman operators associated with \eqref{true}, \eqref{drmdp}, and \eqref{drmdp2} preserve convexity, and hence the corresponding value functions are convex \cite{shapiro2025episodic}. Thus, the min-max formulation in DROC preserves the convex structure of the problem.
Moreover, Theorem~\ref{thm-reformulation} ensures that the problem can be equivalently reformulated as a tractable mean-risk optimization problem. Specifically, to facilitate the numerical solution of the episodic Bayesian DROC problem, we reformulate \eqref{drmdp2} as follows:
\begin{align*}
	V_N(s)=&\inf_{a\in\mathcal{A}}\rho_N\left[\mathcal{C}(s,a,\xi)+\gamma V_N(g(s,a,\xi))\right]\\	=&\inf_{a\in\mathcal{A},\zeta}\lambda_N\sum_{j=1}^J{p}^l_{j,N}\left[\mathcal{C}(s,a,\xi^j)+\gamma V_N(g(s,a,\xi^j))\right]+(1-\lambda_N)\zeta\\
 &+\frac{1-\lambda_N}{1-\upsilon_N}\sum_{j=1}^Jp^{u-l}_{j,N}\left[\mathcal{C}(s,a,\xi^j)+\gamma V_N(g(s,a,\xi^j))-\zeta\right]^+.
\end{align*}

The cutting-plane method iterates between a master problem step (solve the approximate problem under current cuts) and a separation step (add a new supporting cut at a trial point). We now present an explicit BOCP algorithm to approximate the
optimal value function $\hat V_N^*$.
In our stationary infinite-horizon setting, we maintain a global cut pool for the value function. At each trial state $\bar s$, we (i) solve the mean-risk master problem with the current cuts to obtain an approximate solution, and then (ii) generate an operator-level supporting cut at the current trial state $\bar s$. Let \(\underline V_N\) denote the current piecewise-affine approximation of \(\hat V_N^*\), given by the maximum of the accumulated cuts. By construction, we have $\hat{V}_N^* \geq \underline{{V}}_N$. Given a trial point $\bar{s} \in \mathcal{S}$, we determine the current approximate solution by solving the following master problem, which reformulates the mean-risk objective using auxiliary variables:
$$
\begin{aligned}
 (\bar{a},\bar{\zeta},\bar{y}) \in \underset{a\in\mathcal{A} ,\zeta, y\in\mathbb R_+^J}{\arg\min} & \lambda_N\sum_{j=1}^J{p}^l_{j,N}\left[\mathcal{C}^j(\bar s,a)+\gamma\underline{V}_N(A^j\bar{s}+B^ja+b^j)\right]+(1-\lambda_N)\zeta+\frac{1-\lambda_N}{1-\upsilon_N}\sum_{j=1}^{J}p^{u-l}_{j,N}y_j\\[2mm]
 \text{s.t.}\quad
 & y_j\ge\mathcal{C}^j(\bar s,a)+\gamma\underline{V}_N(A^j\bar{s}+B^ja+b^j)-\zeta,\ j=1,\ldots,J.
\end{aligned}
$$
For notational convenience, we denote $\mathcal{C}^j(s,a):=\mathcal{C}(s,a,\xi^j)$, $A^j:=A(\xi^j)$, $B^j:=B(\xi^j)$ and $b^j:=b(\xi^j)$.
Utilizing the chain rule for subdifferentials, we derive a new cutting hyperplane at the trial point $\bar{s}$ in the form $\ell(s)=\underline{v}+q^{\top}(s-\bar{s})$. Here
$$
\begin{aligned}
\underline{v} =&\lambda_N\sum_{j=1}^J{p}^l_{j,N}\left[\mathcal{C}^j(\bar{s},\bar{a})+\gamma \underline{V}_N({s}_j')\right]+(1-\lambda_N)\bar{\zeta}+\frac{1-\lambda_N}{1-\upsilon_N}\sum_{j=1}^Jp^{u-l}_{j,N}\left[\mathcal{C}^j(\bar{s},\bar{a})+\gamma \underline{V}_N({s}_j')-\bar{\zeta}\right]^+, \\
q\in&\lambda_N\sum_{j=1}^J{p}^l_{j,N}\left[\partial_s\mathcal{C}^j(\bar{s},\bar{a})+\gamma(A^j)^{\top}\partial_s\underline{V}_N({s}_j')\right]+\frac{1-\lambda_N}{1-\upsilon_N}\sum_{j=1}^Jp^{u-l}_{j,N}\partial_s\left[\mathcal{C}^j(\bar{s},\bar{a})+\gamma \underline{V}_N({s}_j')-\bar{\zeta}\right]^+,
\end{aligned}
$$
with ${s}_j'=A^j\bar{s}+B^j\bar{a}+b^j$ for simplicity.
For each scenario $j=1,\ldots,J$, the subdifferential of the positive part is:
$$\partial_s\left[\mathcal{C}^j(\bar{s},\bar{a})+\gamma \underline{V}_N({s}_j')-\bar{\zeta}\right]^+=\left\{\theta_j\left(\partial_s\mathcal{C}^j(\bar{s},\bar{a})+\gamma(A^j)^{\top}\partial_s\underline{V}_N({s}_j')\right)\mid\theta_j\in\Theta_j\right\}$$
with 
\begin{equation}\label{eq-theta-partial}
\Theta_j=
\begin{cases} 
\{1\}, & \text{if } \mathcal{C}^j(\bar{s},\bar{a})+\gamma \underline{V}_N({s}_j')>\bar{\zeta} \\ 
[0,1], & \text{if }\mathcal{C}^j(\bar{s},\bar{a})+\gamma \underline{V}_N({s}_j')=\bar{\zeta}\\ 
\{0\}, & \text{if }\mathcal{C}^j(\bar{s},\bar{a})+\gamma \underline{V}_N({s}_j')<\bar{\zeta}.
\end{cases}
\end{equation}
Here, we choose $\theta_j\in\Theta_j$, $j=1,\ldots,J$ such that $\sum_{j=1}^Jp^{u-l}_{j,N}\theta_j=1-\upsilon_N$. The justification for this construction will be given in Lemma~\ref{lem:support}.
We maintain a global cut pool indexed by $\mathcal I$. 
Each cut $\ell_\eta$ is an affine function generated at an anchor state $\bar s_\eta$, i.e.,
\[
\ell_\eta(s)=v_\eta+q_\eta^\top(s-\bar s_\eta),\qquad \eta\in\mathcal I.
\]
Then, $\underline{V}_N:=\max _{\eta \in \mathcal{I}} \ell_\eta$ represents the current approximation to $\hat V_N^*$ by the global cut pool within episode $N$.
For each scenario $j$, let $\hat{\eta}\in \mathcal{I}$ be such that $\underline{V}_N\left({s}_j'\right)=\ell_{\hat{\eta}}\left({s}_j'\right)$, i.e., $\ell_{\hat{\eta}}$ is a supporting hyperplane of $\underline{V}_N$ at ${s}_j'$. Then $\partial_s \ell_{\hat{\eta}}\left({s}_j'\right)\in\partial_s\underline{V}_N\left({s}_j'\right)$.

With the above preparation, we can then develop the BOCP procedure for solving the mean-risk Bellman equation \eqref{drmdp2} at each specific episode $N$, stated as Algorithm \ref{alg:B}.

\begin{algorithm}[htbp]
	\caption{The BOCP algorithm for episodic Bayesian DROC}
	\begin{algorithmic}[1]\label{alg:B}
		\STATE Input: current episode index $N$; current state $s_N$; posterior $\mu_N$;
 maximum iteration number $K_{\mathrm{max}}$,
 accuracy parameter $\varepsilon$.
		\STATE Set $\bar{s}_0= s_{N}$, $k= 0$, $\Delta_0= +\infty$ and initial lower approximation $\underline{{V}}_N^0\equiv -\bar{\mathcal C}/(1-\gamma)$. 
 \WHILE{$k< K_{\mathrm{max}}$ \textbf{and} $\Delta_k>(1-\gamma)\varepsilon$}
		\STATE \textbf{Master problem at $\bar s_k$:} compute the control and auxiliary variables via the master problem
 \begin{equation}\label{eq:phi-k}
\begin{aligned}
 (\bar{a}_k,\bar{\zeta}_k,\bar y) \in \underset{a\in\mathcal{A},\zeta,y\in\mathbb R_+^J}{\arg \min } &\lambda_N\sum_{j=1}^J{p}^l_{j,N}\left[\mathcal{C}^j(\bar{s}_{k},a)+\gamma\underline{V}_N^k(A^j\bar{s}_{k}+B^ja+b^j)\right]+(1-\lambda_N)\zeta\\
 &+\frac{1-\lambda_N}{1-\upsilon_N}\sum_{j=1}^{J}p^{u-l}_{j,N}y_j,\\[1mm]
 \text{s.t.}\quad
 & y_j\ge\mathcal{C}^j(\bar{s}_{k},a)+\gamma\underline{V}_N^k(A^j\bar{s}_{k}+B^ja+b^j)-\zeta,\ j=1,\ldots,J.
\end{aligned}
\end{equation}

\STATE Let $\mathfrak{a}_j^\star$ and $\mathfrak{b}_j^\star$ be the optimal dual multipliers for the constraints in \eqref{eq:phi-k}, and define $\theta_j:=\frac{(1-\upsilon_N)\mathfrak{a}_j^\star}{(1-\lambda_N)p^{u-l}_{j,N}}$.
\STATE \textbf{Operator-level cut at $\bar s_k$:} compute the intercept and an operator-level subgradient
 \begin{equation}\label{eq:v-k}
\begin{aligned}
\underline{v}_k =&\lambda_N\sum_{j=1}^J{p}^l_{j,N}\left[\mathcal{C}^j(\bar{s}_{k},\bar{a}_k)+\gamma \underline{V}_N^k(A^j\bar{s}_{k}+B^j\bar{a}_k+b^j)\right]+(1-\lambda_N)\bar{\zeta}_k\\
 &\quad+\frac{1-\lambda_N}{1-\upsilon_N}\sum_{j=1}^Jp^{u-l}_{j,N}\left[\mathcal{C}^j(\bar{s}_{k},\bar{a}_k)+\gamma \underline{V}_N^k(A^j\bar{s}_{k}+B^j\bar{a}_k+b^j)-\bar{\zeta}_k\right]^+, 
\end{aligned}
\end{equation}
 \begin{equation}\label{eq:q-k}
\begin{aligned}
q_k\in&\lambda_N\sum_{j=1}^J{p}^l_{j,N}\left[\partial_s\mathcal{C}^j(\bar{s}_{k},\bar{a}_k)+\gamma(A^j)^{\top}\partial_s\underline{V}_N^k(A^j\bar{s}_{k}+B^j\bar{a}_k+b^j)\right]\\
&\quad+\frac{1-\lambda_N}{1-\upsilon_N}\sum_{j=1}^Jp^{u-l}_{j,N}\theta_j\left[\partial_s\mathcal{C}^j(\bar{s}_{k},\bar{a}_k)+\gamma (A^j)^{\top}\partial_s\underline{V}_N^k(A^j\bar{s}_{k}+B^j\bar{a}_k+b^j)\right].
\end{aligned}
\end{equation}
\textbf{Update:} add the supporting cut $\ell_{k+1}(s)=\underline v_k+q_k^\top(s-\bar s_k)$ to approximate the value function and update the global underestimator by $\underline V_N^{k+1}:=\max\{\underline V_N^k,\ell_{k+1}\}$.

\STATE \textbf{Bellman operator residual:} $\displaystyle \Delta_{k+1}=\big\|\hat{\mathcal L}_N(\underline V_N^{k+1})-\underline V_N^{k+1}\big\|_\infty$.
\STATE \textbf{Next trial state:} randomly draw a sample $\xi^j$ and set $\bar{s}_{k+1}=A^j\bar{s}_{k}+B^j\bar{a}_k+b^j$.
\STATE $k\gets k+1$.
\ENDWHILE

\RETURN Approximate optimal value function $\underline{{V}}_N:=\underline{V}_N^{k}$ for the current episode $N$.
	\end{algorithmic}
\end{algorithm}

The BOCP algorithm can be regarded as an approximate value iteration procedure that iteratively refines a piecewise-linear lower approximation of the value function. At each step, it adds a supporting hyperplane to the epigraph of the Bellman operator's value function approximation.

\subsection{Convergence and sample complexity}
Conceptually, BOCP connects distributionally robust optimization with approximate dynamic programming at the operator level. In each iteration $k$, the algorithm evaluates the Bellman operator $\hat{\mathcal{L}}_N$ at ${\bar s}_k$ by solving the mean-risk master problem \eqref{eq:phi-k}. It then constructs a supporting cut, whose coefficients are determined by the optimal Karush-Kuhn-Tucker (KKT) multipliers $\mathfrak{a}^\star$, and incorporates this cut into the global lower envelope via \eqref{eq:v-k}-\eqref{eq:q-k}; see Lemma~\ref{lem:support}.

\begin{lemma}\label{lem:support}
Suppose Assumption \ref{ass-sddp} holds.
Let $(\bar a_k,\bar\zeta_k,\bar y)$ be the optimal solution determined through \eqref{eq:phi-k} and $\underline{v}_k$ be the optimal value at $\bar s_k$ given in \eqref{eq:v-k}. Then there exists a subgradient $q_k$ defined by \eqref{eq:q-k} such that for all $s\in\mathcal S$,
\[
\hat{\mathcal L}_N(\underline V_N^k)(s)\geq \underline{v}_k+q_k^\top(s-\bar s_k).
\]
Equivalently, the affine function $\ell_{k+1}(s):=\hat{\mathcal L}_N(\underline V_N^k)(\bar s_k)+q_k^\top(s-\bar s_k)$ is a supporting hyperplane of $\hat{\mathcal L}_N(\underline V_N^k)$ at $\bar s_k$, i.e., $\ell_{k+1}\le \hat{\mathcal L}_N(\underline V_N^k)$.
\end{lemma}

\begin{proof}
For any fixed episode $N$ and an iteration $k$, define 
\[
F_k(s):=\hat{\mathcal L}_N(\underline V_N^k)(s)
=\inf_{a,\zeta,y}\Phi_k(s,a,\zeta,y):=\varphi_k(s,a,\zeta,y) +\delta_{\mathcal A}(a)+\delta_{\mathcal K}(s,a,\zeta,y),
\]
where 
\[\varphi_k(s,a,\zeta,y):=\lambda_N\sum_{j=1}^J p_{j,N}^l\big[\mathcal C^j(s,a)+\gamma\underline V_N^k(A^j s+B^j a+b^j)\big]
+(1-\lambda_N)\zeta+\frac{1-\lambda_N}{1-\upsilon_N}\sum_{j=1}^J p_{j,N}^{u-l} y_j,
\]
and $\mathcal K:=\Big\{(s,a,\zeta,y):y_j\ge \mathcal G_j(s,a,\zeta):=\mathcal C^j(s,a)+\gamma\underline V_N^k(A^j s+B^j a+b^j)-\zeta,y_j\ge 0,\forall j\Big\}.$
Here $\delta_{\mathcal A}$ and $\delta_{\mathcal K}$ are indicator functions.
Specifically, $\delta_{\mathcal A}(x):=0$ if $x\in\mathcal A$ and $+\infty$ otherwise; $\delta_{\mathcal K}$ is defined analogously.
Let $(\bar a_k,\bar\zeta_k,\bar y)$ be an optimal solution of the master \eqref{eq:phi-k} at $\bar s_k$ and let $(\mathfrak{a}^\star,\mathfrak{b}^\star)\in\mathbb{R}_+^J\times\mathbb{R}_+^J$ be the associated optimal multipliers
such that the following KKT conditions hold:
\begin{align}
\text{\textbf{(stationarity in $a$)}}\quad
&\begin{aligned}
&0\in \partial_a\Big[\lambda_N\sum_{j=1}^J p_{j,N}^l\big(\mathcal C^j(\bar s_k,a)+\gamma \underline V_N^k(A^j\bar s_k+B^ja+b^j)\big)\\
&\quad +\sum_{j=1}^J \mathfrak{a}_j^\star \big(\mathcal C^j(\bar s_k,a)+\gamma \underline V_N^k(A^j\bar s_k+B^ja+b^j)\big)\Big]_{a=\bar a_k}+{\cal N}_{\mathcal A}(\bar a_k)
\end{aligned}\label{kkt-a}\\[4pt]
\text{\textbf{(stationarity in $\zeta$)}}\quad
& 0=(1-\lambda_N)-\sum_{j=1}^J \mathfrak{a}_j^\star \label{kkt-zeta}\\
\text{\textbf{(stationarity in $y$)}}\quad
& 0=\frac{1-\lambda_N}{1-\upsilon_N}p^{u-l}_{j,N}-\mathfrak{a}_j^\star-\mathfrak{b}_j^\star,\quad j=1,\dots,J \label{kkt-t}\\
\text{\textbf{(complementary slackness)}}\quad
& \mathfrak{a}_j^\star\big(\bar y_j-\mathcal G_j(\bar s_k,\bar a_k,\bar\zeta_k)\big)=0,\quad
\mathfrak{b}_j^\star\bar y_j=0,\quad j=1,\dots,J. \label{kkt-cs}
\end{align}
Here ${\cal N}_{\mathcal A}$ denotes the normal cone of $\mathcal A$.
Define
\[
h_j(s,a,\zeta,y):=\mathcal G_j(s,a,\zeta)-y_j,\quad
l_j(s,a,\zeta,y):=-y_j.
\]
Then the normal cone admits the subgradient representation
\[
\begin{aligned}
 {\cal N}_{\mathcal K}(\bar s_k,\bar a_k,\bar\zeta_k,\bar y)
=\Big\{\sum_{j=1}^J \mathfrak{a}_j \nu_j
+\sum_{j=1}^J \mathfrak{b}_j \eta_j: \mathfrak{a},\mathfrak{b}\in\mathbb R_+^J,\ 
\nu_j\in\partial h_j(\bar s_k,\bar a_k,\bar\zeta_k,\bar y),\
\eta_j\in\partial l_j(\bar s_k,\bar a_k,\bar\zeta_k,\bar y),\\
\mathfrak{a}_j h_j(\bar s_k,\bar a_k,\bar\zeta_k,\bar y)=0,\ \mathfrak{b}_j l_j(\bar s_k,\bar a_k,\bar\zeta_k,\bar y)=0\Big\}.
\end{aligned}
\]
For each $j$ and any $d_j\in \partial \underline V_N^k(A^j\bar s_k+B^j\bar a_k+b^j)$, the affine chain rule in \cite[Thm.~23.9]{rockafellar2015convex} yields
\[
\partial_s\big[\underline V_N^k(A^j s+B^j a+b^j)\big]_{(s,a)=(\bar s_k,\bar a_k)}
=(A^j)^\top d_j,\quad
\partial_a\big[\underline V_N^k(A^j s+B^j a+b^j)\big]_{(s,a)=(\bar s_k,\bar a_k)}
=(B^j)^\top d_j.
\]
Therefore, selecting
$\nu_j=\big(\partial_s\mathcal C^j(\bar s_k,\bar a_k)+\gamma (A^j)^{\top} d_j,
 \partial_a\mathcal C^j(\bar s_k,\bar a_k)+\gamma (B^j)^{\top} d_j,
 -1,
 -e_j\big)
\in\partial h_j(\bar s_k,\bar a_k,\bar\zeta_k,\bar y),$ and $\eta_j=(0,0,0,-e_j)\in\partial l_j(\bar s_k,\bar a_k,\bar\zeta_k,\bar y),$
where $e_j$ is the $j$-th unit vector in $\mathbb{R}^J$, we can deduce that
$w^\star:=\sum_{j=1}^J \mathfrak{a}_j^\star\nu_j
+\sum_{j=1}^J \mathfrak{b}_j^\star\eta_j
\in{\cal N}_{\mathcal K}(\bar s_k,\bar a_k,\bar\zeta_k,\bar y)$.

Since $\varphi_k$ is continuous at $(\bar s_k,\bar a_k,\bar\zeta_k,\bar y)$, the Moreau--Rockafellar subdifferential sum rule (see \cite[Thm.~23.8]{rockafellar2015convex}) gives
\[
 \partial \varphi_k(\bar s_k,\bar a_k,\bar\zeta_k,\bar y)
+ \partial \delta_{\mathcal A}(\bar a_k) + \partial \delta_{\mathcal K}(\bar s_k,\bar a_k,\bar\zeta_k,\bar y)\subset\partial \Phi_k(\bar s_k,\bar a_k,\bar\zeta_k,\bar y).
\]
Recall that individual blocks for $x:=(x_s,x_a,x_\zeta,x_y)\in\partial \varphi_k(\bar s_k,\bar a_k,\bar\zeta_k,\bar y)$ are 
\[
\begin{aligned}
x_s&=\lambda_N\sum_{j=1}^J p_{j,N}^l\big[\partial_s\mathcal C^j(\bar s_k,\bar a_k)+\gamma(A^j)^\top d_j\big],\\
x_a&=\lambda_N\sum_{j=1}^J p_{j,N}^l\big[\partial_a\mathcal C^j(\bar s_k,\bar a_k)+\gamma(B^j)^\top d_j\big],\\
x_\zeta&=(1-\lambda_N),\quad
x_y=\frac{1-\lambda_N}{1-\upsilon_N}(p^{u-l}_{1,N},\dots,p^{u-l}_{J,N}).
\end{aligned}
\]
We now show that there exist $x^\star:=(x_s^\star,x_a^\star,x_\zeta^\star,x_y^\star)\in\partial \varphi_k(\bar s_k,\bar a_k,\bar\zeta_k,\bar y)$ and
$u_a^\star\in {\cal N}_{\mathcal A}(\bar a_k)$ such that
\[x^\star+(0,u_a^\star,0,0)+w^\star=(q_k,0,0,0)\in \partial \Phi_k(\bar s_k,\bar a_k,\bar\zeta_k,\bar y).\]
\begin{itemize}
\item[(i)] \textbf{$y$-block:} For each $j$, by \eqref{kkt-t} we have that
$\big(x_y^\star+w^\star_y\big)_j=\frac{1-\lambda_N}{1-\upsilon_N}p^{u-l}_{j,N}-(\mathfrak{a}_j^\star+\mathfrak{b}_j^\star)=0.$
\item[(ii)] \textbf{$\zeta$-block:} By \eqref{kkt-zeta},
$x_\zeta^\star+w^\star_\zeta=(1-\lambda_N)-\sum_{j=1}^J\mathfrak{a}_j^\star=0$.
\item[(iii)] \textbf{$a$-block:} Define
\[
\bar\varphi(a):=\lambda_N\sum_{j=1}^J p_{j,N}^l\Big[\mathcal C^j(\bar s_k,a)+\gamma\underline V_N^k(A^j\bar s_k+B^j a+b^j)\Big]
+\sum_{j=1}^J \mathfrak{a}_j^\star\Big[\mathcal C^j(\bar s_k,a)+\gamma\underline V_N^k(A^j\bar s_k+B^j a+b^j)\Big].
\]
By \eqref{kkt-a}, $0\in\partial_a\bar\varphi(\bar a_k)+{\cal N}_{\mathcal A}(\bar a_k)$.
Hence there exist $r^\star_a\in\partial_a\bar\varphi(\bar a_k)$ and $u^\star_a\in {\cal N}_{\mathcal A}(\bar a_k)$ with $r^\star_a+u^\star_a=0$. This means that there exists $d_j^\star\in\partial \underline V_N^k(A^j\bar s_k+B^j\bar a_k+b^j)$ such that
\[
r^\star_a=\sum_{j=1}^J\big(\lambda_N p_{j,N}^l+\mathfrak{a}_j^\star\big)\big[\partial_a\mathcal C^j(\bar s_k,\bar a_k)+\gamma(B^j)^\top d^\star_j\big].
\]
Thus, we have $x_a^\star+w^\star_a=r^\star_a$ and then $x_a^\star+w^\star_a+u^\star_a=0$.

\item[(iv)] \textbf{$s$-block:}
With the consistent choice of $d^\star$ in (iii),
$x_s^\star+w^\star_s
=\lambda_N\sum_{j=1}^J p_{j,N}^l\big[\partial_s\mathcal C^j(\bar s_k,\bar a_k)+\gamma(A^j)^\top d^\star_j\big]
+\sum_{j=1}^J \mathfrak{a}_j^\star\big[\partial_s\mathcal C^j(\bar s_k,\bar a_k)+\gamma(A^j)^\top d^\star_j\big].$
Define $\theta_j:=\frac{(1-\upsilon_N)\mathfrak{a}_j^\star}{(1-\lambda_N)p^{u-l}_{j,N}}\in[0,1]$, then we have $\sum_{j=1}^J p^{u-l}_{j,N}\theta_j=1-\upsilon_N$ by \eqref{kkt-zeta}.
Thus 
\[
x_s^\star+w^\star_s
=\lambda_N\sum_{j=1}^J p_{j,N}^l\big[\partial_s\mathcal C^j(\bar s_k,\bar a_k)+\gamma(A^j)^\top d^\star_j\big]
+\frac{1-\lambda_N}{1-\upsilon_N}\sum_{j=1}^J p^{u-l}_{j,N}\theta_j\big[\partial_s\mathcal C^j(\bar s_k,\bar a_k)+\gamma(A^j)^\top d^\star_j\big]
=:q_k.\]
\end{itemize} 
Combining (i)-(iv) gives
$(q_k,0,0,0)\in\partial \Phi_k(\bar s_k,\bar a_k,\bar\zeta_k,\bar y).$ Therefore, for all $(s,a,\zeta,y)$, we have
$\Phi_k(s,a,\zeta,y)\ge\Phi_k(\bar s_k,\bar a_k,\bar\zeta_k,\bar y)+q_k^\top(s-\bar s_k).$
Taking the infimum over $(a,\zeta,y)$ yields
\[
F_k(s)=\inf_{a,\zeta,y}\Phi_k(s,a,\zeta,y)\ge\Phi_k(\bar s_k,\bar a_k,\bar\zeta_k,\bar y)+q_k^\top(s-\bar s_k)
=F_k(\bar s_k)+q_k^\top(s-\bar s_k).
\]
Therefore, $q_k\in\partial_s F_k(\bar s_k)=\partial_s\hat{\mathcal L}_N(\underline V_N^k)(\bar s_k)$. Since $\underline v_k:=F_k(\bar s_k)$, we obtain 
\[
\hat{\mathcal L}_N(\underline V_N^k)(s)\ge\underline v_k+q_k^\top(s-\bar s_k),\quad \forall s\in\mathcal S.
\]
This completes the proof.
\end{proof}

By Lemma~\ref{lem:support}, the new cut satisfies
$\ell_{k+1}\le \hat{\mathcal L}_N(\underline V_N^k)$.
Since $\hat{\mathcal L}_N$ is monotone and $\underline V_N^k\le \hat V_N^*$ (by induction),
we have
\[
\underline V_N^{k+1}=\max\{\underline V_N^k,\ell_{k+1}\}
\le \max\{\hat V_N^*, \hat{\mathcal L}_N(\underline V_N^k)\}
\le \hat V_N^*,
\]
so the envelope remains a global underestimator.

\begin{lemma}[Monotone lower bounds]\label{lem:valid}
Suppose Assumption \ref{ass-sddp} holds. Then the BOCP update produces a sequence $\{\underline V_N^k\}_{k\ge0}$ such that, for all $k$,
\[
\underline V_N^k\leq\underline V_N^{k+1}\leq \hat V_N^*.
\]
In particular, the pointwise limit $\underline V_N^\infty(s):=\lim_{k\to\infty}\underline V_N^k(s)$ exists and satisfies $\underline V_N^\infty\le \hat V_N^*$.
\end{lemma}

\begin{proof}
 We prove the monotonicity by induction from $k=0$. 
 Since $\underline V_N^0\le \hat V_N^*$,
by monotonicity of $\hat{\mathcal L}_N$, we have 
\[
\ell_{1}\le\hat{\mathcal L}_N(\underline V_N^0)\leq\hat{\mathcal L}_N(\hat V_N^*)=\hat V_N^*.
\]
where the first inequality follows from Lemma \ref{lem:support}.
Hence $\underline V_N^{1}=\max\{\underline V_N^0,\ell_{1}\}\le \hat V_N^*$ and $\underline V_N^{1}\ge \underline V_N^0$ pointwise. Analogous to the above proof for $k=0$, we can derive the result for all $k$.
Since 
$\underline V_N^k$ is nondecreasing in $k$ and bounded above by $\hat V_N^*$ for each $s$, the pointwise limit $\underline V_N^\infty(s):=\lim_{k\to\infty}\underline V_N^k(s)$ exists.
\end{proof}

To show the convergence of the BOCP algorithm, we need the following assumption.

\begin{assumption}\label{ass:dense}
For any fixed episode $N$, assume that the sequence $\{\bar s_k\}$ generated in Algorithm~\ref{alg:B} satisfies that for every $s\in\mathcal S$, $\varepsilon>0$ and $K\in\mathbb{N}$, there exists an index $k\geq K$ such that $\|\bar s_k-s\|<\varepsilon$ a.s.
\end{assumption}

Assumption \ref{ass:dense} ensures that cuts are generated throughout the state space, guaranteeing uniform convergence. When $\mathcal{S}$ is finite, Assumption \ref{ass:dense} reduces to the standard requirement that each trial state is visited infinitely many times a.s., which is commonly used in \cite{fullner2025stochastic,philpott2008convergence}.

\begin{theorem}[Convergence of the BOCP algorithm]\label{thm:cp-episode}
Suppose that Assumptions~\ref{ass-sddp}-\ref{ass:dense} hold. For any fixed episode $N$, if $\underline V_N^\infty$ is continuous on $\mathcal S$, then $\{\underline V_N^k\}$ generated by Algorithm~\ref{alg:B} converges to $\hat V_N^*$, that is, as $k\to\infty$,
\[
\underline V_N^k(s)\uparrow \hat V_N^*(s)\quad \text{for all } s\in\mathcal S,
\text{ and }
\|\underline V_N^k-\hat V_N^*\|_\infty\to0.
\]
\end{theorem}

\begin{proof}
Let $\mathfrak C^\infty$ be the collection of all generated cuts, i.e., $\mathfrak C^\infty=\{\ell_{k+1}\}_{k\ge0}$. By construction,
\[
\underline V_N^k=\max\big\{\underline V_N^0,\ell_1,\dots,\ell_k\big\},\ 
\underline V_N^\infty=\sup\big\{\underline V_N^0,\ell:\ell\in\mathfrak C^\infty\big\}.
\]
For each $\ell_{k+1}\in\mathfrak C^\infty$, we have $\ell_{k+1}\le \hat{\mathcal L}_N(\underline V_N^k)$ by Lemma~\ref{lem:support}. Moreover, by monotonicity of $\hat{\mathcal L}_N$ and $\underline V_N^k\le \underline V_N^\infty$, it follows that
\[
\ell_{k+1}\le\hat{\mathcal L}_N(\underline V_N^k)\le\hat{\mathcal L}_N(\underline V_N^\infty).
\]
Taking the supremum over all cuts with $\underline V_N^0\le \hat{\mathcal L}_N(\underline V_N^\infty)$ yields
\begin{equation}\label{eq:thm-alg-con-1}
\underline V_N^\infty\le\hat{\mathcal L}_N(\underline V_N^\infty).
\end{equation}

At each iteration $k$, by the definition of $\ell_{k+1}$, we have $\ell_{k+1}(\bar s_k)=\hat{\mathcal L}_N(\underline V_N^k)(\bar s_k)$,
hence
\begin{equation}\label{eq:thm-alg-con-2}
 \underline V_N^{k+1}(\bar s_k)=\max\{\underline V_N^k(\bar s_k),\ell_{k+1}(\bar s_k)\}\ge\hat{\mathcal L}_N(\underline V_N^k)(\bar s_k).
\end{equation}
For any fixed $s\in\mathcal S$, Assumption~\ref{ass:dense} ensures that there exists a subsequence $\{\bar s_{k_j}\}$ such that $\bar s_{k_j}\to s$ as $j\to\infty$. 
Taking $k=k_j$ in \eqref{eq:thm-alg-con-2} and using $\underline V_N^{k_j+1}\le \underline V_N^\infty$ yields
\begin{equation}\label{eq-5-11}
 \underline V_N^\infty(\bar s_{k_j})
\ge \underline V_N^{k_j+1}(\bar s_{k_j})
\ge \hat{\mathcal L}_N(\underline V_N^{k_j})(\bar s_{k_j}). 
\end{equation}
Next we justify that $\hat{\mathcal L}_N(\underline V_N^{k_j})(\bar s_{k_j})\to \hat{\mathcal L}_N(\underline V_N^\infty)(\bar s_{k_j})$.
Since each $\underline V_N^k$ is continuous and
$\underline V_N^k\uparrow \underline V_N^\infty$ pointwise on the compact set $\mathcal S$,
with $\underline V_N^\infty$ continuous by assumption, Dini's theorem implies
\[
\|\underline V_N^k-\underline V_N^\infty\|_\infty\to0.
\]
Moreover, $\hat{\mathcal L}_N$ is a $\gamma$-contraction under $\|\cdot\|_\infty$, hence 
\begin{equation}\label{eq-5-12}
 \|\hat{\mathcal L}_N(\underline V_N^{k})-\hat{\mathcal L}_N(\underline V_N^\infty)\|_\infty
\le \gamma \|\underline V_N^{k}-\underline V_N^\infty\|_\infty\to0.
\end{equation}
Combining \eqref{eq-5-11} and \eqref{eq-5-12} with the continuity of both $\underline V_N^\infty$ and $\hat{\mathcal L}_N(\underline V_N^\infty)$ on $\mathcal S$, letting $j\to\infty$ gives
$\underline V_N^\infty(s)\ge \hat{\mathcal L}_N(\underline V_N^\infty)(s)$.
Since $s$ is arbitrary, this yields
\begin{equation}\label{eq:thm-alg-con-3}
 \underline V_N^\infty\ge\hat{\mathcal L}_N(\underline V_N^\infty).
\end{equation}
From \eqref{eq:thm-alg-con-1} and \eqref{eq:thm-alg-con-3}, we have $\underline V_N^\infty=\hat{\mathcal L}_N(\underline V_N^\infty)$, i.e., $\underline V_N^\infty$ is a fixed point of $\hat{\mathcal L}_N$. By Lemma~\ref{lem-beo-3}, $\hat{\mathcal L}_N$ admits a unique fixed point $\hat V_N^*$. Therefore $\underline V_N^\infty=\hat V_N^*$. Recalling that we have already shown $\|\underline V_N^k-\underline V_N^\infty\|_\infty\to0$ by Dini's theorem, we conclude that
$\|\underline V_N^k-\hat V_N^*\|_\infty\to0.$
\end{proof}

\begin{remark}
 A sufficient condition for the continuity of $\underline V_N^\infty$ is a uniform bound on cut slopes:
assume there exists $L<\infty$ such that $\|q_\eta\|\le L$ for all cuts $\ell_\eta(s)=v_\eta+q_\eta^\top(s-\bar s_\eta)$.
Then each $\underline V_N^k=\max_{\eta\le k}\ell_\eta$ is $L$-Lipschitz on $\mathcal S$, hence the pointwise limit
$\underline V_N^\infty=\sup_k \underline V_N^k$ is also $L$-Lipschitz and therefore continuous.
Such a bound is standard under Lipschitz regularity: if $\mathcal C^j(\cdot,a)$ is $L_{\cal C}$-Lipschitz in $s$ uniformly in $a$ and
$\|A^j\|\le L_g$ with $\gamma L_g<1$, then the fixed point $\hat V_N^*$ is Lipschitz with constant $\frac{L_{\cal C}}{1-\gamma L_g}$ (see Lemma \ref{prop-lip}),
and the BOCP subgradients $q_k$ remain uniformly bounded.
\end{remark}

As shown in Lemma \ref{lem:support}, the BOCP algorithm builds operator-level cuts for the Bellman operator $\hat{\cal L}_N$ and maintains a single global cut pool for the stationary value function.
Theorem~\ref{thm:cp-episode} demonstrates that BOCP realizes a fixed-point computation of the Bayesian DROC Bellman equation since the global lower envelope increases monotonically and converges uniformly on $\cal S$ to the unique fixed point $\hat{V}_N^*$. Thus BOCP provides a convergent and computationally tractable method for solving the infinite-horizon Bayesian DROC Bellman equation.

Recall that Algorithm~\ref{alg:B} provides the solution of \eqref{drmdp2} for fixed episode $N$, we now consider the sample complexity of Algorithm~\ref{alg:A} by determining a stopping episode $N_{\mathrm{max}}$, which is based on Proposition~\ref{prop-N-finite}.

\begin{proposition}[Sample complexity]\label{prop-sample-complexity}
Suppose Assumption~\ref{ass-sddp} holds. For any fixed $\varepsilon>0$ and $\delta\in(0,1)$, 
when $N_{\mathrm{max}}:= O\left(\frac{J\log J+\log(1/\delta)}{\varepsilon^2}\right)$,
we have
\[
\mathbb P \left( \big\|\hat V_{N_{\mathrm{max}}}^*-V^*\big\|_\infty<\varepsilon \right) \ge 1-\delta.
\]
\end{proposition}
\begin{proof}
By Proposition~\ref{prop-N-finite},
for any fixed $\varepsilon>0$ and $\delta\in(0,1)$, there exist $c_1,\dots,c_3>0$ independent of $J,\varepsilon,\delta$ such that
\[
 N(J,\varepsilon,\delta):= \max\Big\{
\underbrace{c_1\tfrac{J}{\varepsilon}}_{ N_1(J,\varepsilon)},
\ \underbrace{c_2\tfrac{J+\log(1/\delta)}{\varepsilon^2}}_{ N_2(J,\varepsilon,\delta)},
\ \underbrace{c_3\tfrac{J\log J}{\varepsilon^2}}_{N_3(J,\varepsilon)}
\Big\}
\]
ensures $\mathbb P \left( \big\|\hat V_N^*-V^*\big\|_\infty<\varepsilon \right) \ge 1-\delta
$ for all $N\ge N(J,\varepsilon,\delta)$.
Therefore, picking $N_{\mathrm{max}}= O\left(\frac{J\log J+\log(1/\delta)}{\varepsilon^2}\right)$ yields the claim.
\end{proof}

\begin{remark}
The stopping criterion 
$\Delta_{k}:=\big\|\hat{\mathcal L}_N(\underline V_N^{k})-\underline V_N^{k}\big\|_\infty \le (1-\gamma)\varepsilon$ 
ensures
\(
\|\underline V_N^{k}-\hat V_N^*\|_\infty
\le \|\hat{\mathcal L}_N(\underline V_N^{k})-\underline V_N^{k}\|_\infty/(1-\gamma)\le \varepsilon
\). Combining this with Proposition~\ref{prop-sample-complexity} gives
\[
\mathbb P \left( \big\|\underline V_{N_{\mathrm{max}}}^{k}-V^*\big\|_\infty<2\varepsilon \right) \ge 1-\delta.
\]
This stopping test is an ideal (function-space) criterion on continuous state spaces. In practice, one typically evaluates the Bellman residual on a discrete subset of $\cal S$. The quantitative stability results for this kind of discretization can be found in \cite{ma2024bayesian}.
\end{remark}

To evaluate the performance of the corresponding policy induced by $\underline{{V}}_N(s)$, we compute state-control pairs sequentially in forward time starting with initial state $\hat{s}_1$. At stage $t=1, \ldots$, given the current state $\hat{s}_t$, the corresponding control $\hat{a}_t$ is determined by
$$
\begin{aligned}
 (\hat{a}_t,\hat{\zeta}_t) \in \underset{a\in\cal A,\zeta}{\arg \min } &\lambda_N\sum_{j=1}^J{p}^l_{j,N}\left[\mathcal{C}^j(\hat{s}_t,a)+\gamma \underline{V}_N(A^j\hat{s}_t+B^ja+b^j)\right]+(1-\lambda_N)\zeta\\
 &+\frac{1-\lambda_N}{1-\upsilon_N}\sum_{j=1}^Jp^{u-l}_{j,N}\left[\mathcal{C}^j(\hat{s}_t,a)+\gamma \underline{V}_N(A^j\hat{s}_t+B^ja+b^j)-\zeta\right]^+.
\end{aligned}
$$
By drawing an i.i.d. sample path $\{\hat{\xi}_t, t=1, \ldots\}$, we can determine the next state $\hat{s}_{t+1}=A\left(\hat{\xi}_t\right) \hat{s}_t+B\left(\hat{\xi}_t\right) \hat{a}_t+b\left(\hat{\xi}_t\right)$.
With this sample path,
an unbiased estimator of the policy value is:
\begin{equation}\label{policy-evaluate}
 \sum_{t=1}^{\infty} \gamma^{t-1} \mathcal{C}\left(\hat{s}_t, \hat{a}_t, \hat{\xi}_t\right). 
\end{equation}
In practice, we can use the truncation at a finite horizon $T$ to approximate the infinite sum in \eqref{policy-evaluate}, where the truncation error satisfies:
$$
\left|\sum_{t=1}^\infty \gamma^{t-1} \mathcal{C}\left(\hat{s}_t, \hat{a}_t, \hat{\xi}_t\right)-\sum_{t=1}^T \gamma^{t-1} \mathcal{C}\left(\hat{s}_t, \hat{a}_t, \hat{\xi}_t\right)\right| \leq(1-\gamma)^{-1} \gamma^T \bar{\mathcal{C}} .
$$
Equivalently, it suffices to choose $T\geq \frac{\log (\epsilon(1-\gamma)/\bar{\mathcal{C}})}{\log \gamma} ,$ which guarantees that the truncation error does not exceed the prescribed tolerance $\epsilon$.

\subsection{Episodic warm start}\label{subsec:warmstart}
Algorithm~\ref{alg:B} solves the Bayesian DROC Bellman equation for a fixed episode $N$, where the posterior-induced risk functional $\rho_N$
and hence the Bellman operator $\hat{\mathcal L}_N$ are fixed.
In the multi-episode setting, instead of restarting BOCP from scratch at each episode, it is natural to warm-start episode \(N{+}1\) using the cut pool obtained at episode \(N\).

The key difficulty is that the Bellman operator changes across episodes:
$\hat{\mathcal L}_N \neq \hat{\mathcal L}_{N+1}$ because $(\lambda_N,\upsilon_N,P_N^l,P_N^{u-l})$ depend on the updated posterior.
Therefore, a lower approximation $\underline V_N$ that underestimates $\hat V_N^*$ does not automatically remain a valid lower bound for $\hat V_{N+1}^*$.
To safely reuse cuts, we first provide the following standard result for the episodic Bellman operator $\hat{\mathcal L}_{N+1}$.

\begin{lemma}[Theorem 1, \cite{ma2024bayesian}]\label{lem:warm-sub}
If a bounded function $V$ satisfies $
V \le \hat{\mathcal L}_{N+1}(V),$ then $V \le \hat V_{N+1}^*$.
Similarly, if $V \ge \hat{\mathcal L}_{N+1}(V)$, then $V \ge \hat V_{N+1}^*$.
\end{lemma}

Let $\ell(s)=q^\top s + v$ be an affine cut generated in episode $N$.
We call $\ell$ $(N{+}1)$-valid if it is a sub-solution of episode $N{+}1$:
\[
\ell(s)\le \hat{\mathcal L}_{N+1}(\ell)(s),\quad \forall s\in\mathcal S.
\]
Equivalently, it suffices to check the nonnegativity of the worst violation:
\begin{equation}\label{eq:warmcheck-def}
\Delta_{N\to N+1}(\ell)
:=\min_{s\in\mathcal S}\ \big(\hat{\mathcal L}_{N+1}(\ell)(s)-\ell(s)\big)\ \ge 0.
\end{equation}
Under the mean-risk reformulation \eqref{drmdp2}, the minimization in \eqref{eq:warmcheck-def} admits an explicit convex-program form:
\begin{equation}\label{eq:warmcheck-cvx}
\begin{aligned}
\Delta_{N\to N+1}(\ell)=
\min_{s\in\mathcal S,\ a\in\mathcal A,\ \zeta,\ y\in\mathbb R_+^J}\ 
&\lambda_{N+1}\sum_{j=1}^J p^{l}_{j,N+1}\Big(\mathcal C^j(s,a)+\gamma \ell(A^j s+B^j a+b^j)\Big)
+(1-\lambda_{N+1})\zeta\\
&\quad+\frac{1-\lambda_{N+1}}{1-\upsilon_{N+1}}\sum_{j=1}^J p^{u-l}_{j,N+1} y_j
-\ell(s)\\
\text{s.t.}\quad 
& y_j\ge \mathcal C^j(s,a)+\gamma \ell(A^j s+B^j a+b^j)-\zeta,\quad j=1,\dots,J.
\end{aligned}
\end{equation}
Under Assumption~\ref{ass-sddp}, \eqref{eq:warmcheck-cvx} is a convex optimization problem. Furthermore, if $\mathcal S,\mathcal A$ are polyhedral
and $\mathcal C^j$ is piecewise-linear convex, then \eqref{eq:warmcheck-cvx} reduces to a linear program.

Let $\mathfrak C_N$ be the cut pool obtained at episode $N$, and define the subset of $(N{+}1)$-valid cuts by
\[
\mathfrak C_N^{\mathrm{valid}}
:=\{\ell\in\mathfrak C_N:\Delta_{N\to N+1}(\ell)\ge 0\}.
\]
Then the maximum of valid cuts forms a valid warm-start lower bound from the following proposition.

\begin{proposition}\label{prop:warm-max}
Let $\underline V_{N+1}^0(s):=\max\big\{-\bar{\mathcal C}/(1-\gamma),\ \max_{\ell\in\mathfrak C_N^{\mathrm{valid}}}\ell(s)\big\}$.
Then $\underline V_{N+1}^0$ satisfies $\underline V_{N+1}^0 \le \hat{\mathcal L}_{N+1}(\underline V_{N+1}^0)$, and hence
$\underline V_{N+1}^0 \le \hat V_{N+1}^*$.
\end{proposition}

\begin{proof}
For each $\ell\in\mathfrak C_N^{\mathrm{valid}}$ we have $\ell\le \hat{\mathcal L}_{N+1}(\ell)$ by definition.
Let $\underline V(s):=\max_{\ell\in\mathfrak C_N^{\mathrm{valid}}}\ell(s)$.
Then for any $\ell$ we have $ \ell\leq \underline V$, so by monotonicity
$ \ell\leq\hat{\mathcal L}_{N+1}(\ell)\leq \hat{\mathcal L}_{N+1}(\underline V)$.
Taking the maximum over $\ell$ yields $\underline V\leq \hat{\mathcal L}_{N+1}(\underline V)$.
The baseline constant $-\bar{\mathcal C}/(1-\gamma)$ is also a valid global lower bound, and the same monotonicity argument shows that
$\underline V_{N+1}^0=\max\{-\bar{\mathcal C}/(1-\gamma),\underline V\}$ remains a sub-solution.
The claim follows from Lemma~\ref{lem:warm-sub}.
\end{proof}

In summary, this section develops the BOCP algorithm as a computationally tractable method for solving episodic Bayesian DROC problems. Its convergence guarantees, together with the proposed warm-start mechanism across posterior updates, provide a practical framework for infinite-horizon distributionally robust control under episodic Bayesian learning.

\section{Numerical results}
In this section, we illustrate the theoretical results and evaluate Algorithms~\ref{alg:A} and \ref{alg:B} on a single-product inventory control problem under the episodic Bayesian DROC formulation. All experiments were conducted on a 64-bit PC with 12 GB RAM and a 3.20 GHz processor. 

We consider an infinite-horizon stationary inventory problem discussed in \cite{shapiro2025episodic}:
$$
\begin{array}{cl}
\min_{\pi \in \Pi} & \mathbb{E}^{\pi}_{P^c} \left[\sum_{t=1}^{\infty} \gamma^{t-1}\left(c a_t+\psi\left(s_t+a_t, \xi_t\right)\right)\right] \\
\text { s.t. } & s_{t+1}=s_t+a_t-\xi_t,\ a_t \geq 0
\end{array}
$$
where
$\psi(y, d):=b[d-y]^++h[y-d]^+,$ $c,b,h$ are
the per-unit ordering cost, backorder penalty cost, holding cost, respectively, satisfying $b>c>0$. Here,
$s_t$ represents the inventory level at time $t$, $a_t$ denotes the order quantity, and the demand $\xi_t \geq 0$ is an i.i.d. random variable. The inventory problem admits an optimal base-stock policy $\pi^*(s)=\left[s^*-s\right]^+$, where the optimal threshold $s^*=F^{-1}(\kappa)$ is determined by
$$
F(\nu):=P^c(\xi \leq \nu)=\mathbb{E}_{P^c}\left[\mathds{1}_{\{\xi\in(-\infty, \nu]\}}\right],\kappa=\frac{b-(1-\gamma) c}{b+h}
$$
as in \cite[Section 1.2]{shapiro2021lectures}.
According to \cite{shapiro2025episodic}, for $s \leq s^*$, we have
\begin{equation}\label{eq-numer-true}
V^*(s)=-c s+(1-\gamma)^{-1} \mathbb{E}_{P^c}\left[\gamma c \xi+(1-\gamma) c s^*+\psi\left(s^*, \xi\right)\right]. 
\end{equation}
We adopt the following parameter configuration: $c=1$, $h=2$, $b=10$, and $\xi$ follows an exponential distribution $\mathrm{Exp}(10)$, which is assumed to be the \emph{clean} environment.
To match the finite-support assumption, we truncate the demand at \(U=50\) (close to $F^{-1}(0.99)$) and split \([0,U]\) into \(J=50\) equal-width intervals with edges \(u_j:=jU/J=j\) and grid points \(\xi^j=j-\frac{1}{2}\). The normalized bin probabilities are \(p_j=\big(F(u_j)-F(u_{j-1})\big)/F(U)\).
Using the closed-form solution \eqref{eq-numer-true},  we approximate the benchmark value function \(V^*\) with \(10^5\) samples.

In episode $N$, we update the Bayesian posterior by \eqref{posterior} using the observed data and construct the ambiguity set \eqref{ambiguity} corresponding to the posterior distribution.
We then compute the episode-wise optimal value function $\hat{V}_N^*$ from the following Bellman equation:
\begin{equation}\label{eq-num-droc}
 V_N(s) = \inf_{a \in \mathcal{A}}\sup_{{P}\in \mathcal{P}_{\mu_N}} \mathbb{E}_{P} [c a+\psi\left(s+a, \xi\right) + \gamma V_{N}(s+a-\xi)].
\end{equation}
Here, we choose the credible level $\alpha=0.2$ to construct $\mathcal{P}_{\mu_N}$ and the discount factor $\gamma=0.95$. 
We first plot the integrated gap between $\hat{V}_N^*$ and $V^*$, computed as 
$\int_{\mathcal S}\big|\hat V_N^*(s)-V^*(s)\big|d\varpi^*(s)$, where $\varpi^*$ is the
stationary distribution under true optimal policy $\pi^*$ of the original problem. 
The procedure is replicated 100 times to reduce sampling variability.
\begin{figure}[h] 
	\centering \includegraphics[width=0.7\textwidth]{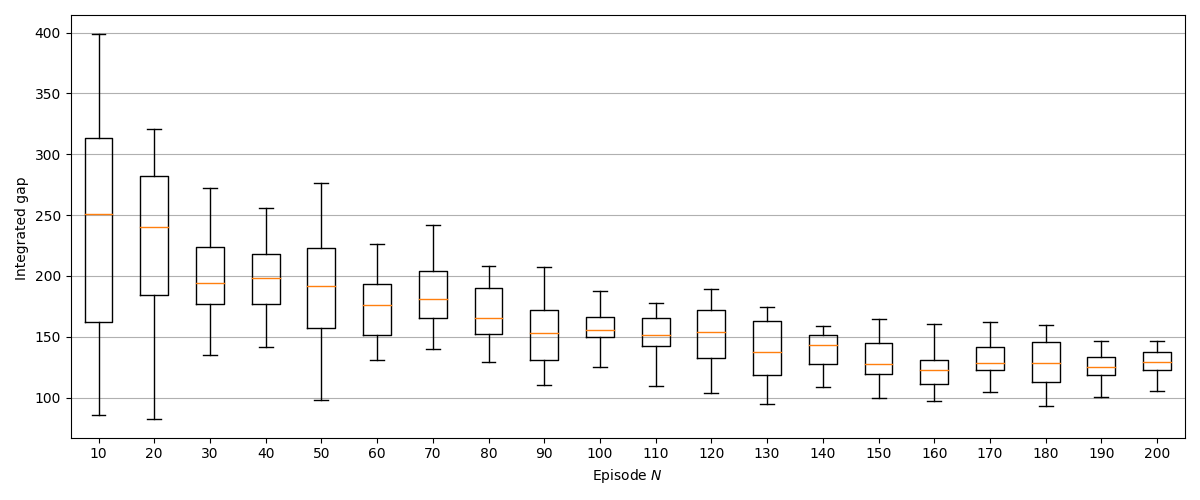} 
	\caption{Convergence of the integrated gap of the optimal value function across episodes.} 
	\label{fig:3} 
\end{figure}

As shown in Figure~\ref{fig:3}, the episodic Bayesian DROC approach yields an optimal value function that converges to $V^*$ as the episode index increases.
Next, we empirically validate the quantitative statistical robustness result under a Huber-type perturbation of the clean environment $P^c=\mathrm{Exp}(10)$. In this experiment, $T_N^{s_0}$ denotes the optimal value at the reference state $s_0=0$ obtained from the Bayesian DROC model constructed using $N=500$ observed samples.
For each shift level $\varepsilon\in[0,0.3]$, we define a perturbed distribution
$\tilde P^c_\varepsilon:=(1-\varepsilon)\mathrm{Exp}(10)+\varepsilon\mathrm{Exp}(40)$.
Note that $\tilde P^c_\varepsilon$ is a mixture distribution rather than a single exponential distribution. Hence, the perturbed environment is no longer within the nominal exponential family, and the experiment reflects a genuine distribution shift away from the clean environment.
We quantify the input perturbation by the one-dimensional Kantorovich distance $\dd_K(P^c,\tilde P^c_\varepsilon)$, which increases approximately linearly with the shift level $\varepsilon$ under this perturbation family.
We generate $R=1000$ independent datasets $\vec{\hat\xi}_N^{(r)}\sim (P^c)^{\otimes N}$ and
$\vec{\tilde\xi}_N^{(r)}\sim (\tilde P^c_\varepsilon)^{\otimes N}$.
This produces two empirical samples $T_N^{s_0}(\vec{\hat\xi}_N^{(r)})$ and $T_N^{s_0}(\vec{\tilde\xi}_N^{(r)})$. We then estimate the empirical Kantorovich distance on $\mathbb R$ as
\[
\dd_K\left((P^c)^{\otimes N}\circ (T_N^{s_0})^{-1},(\tilde P^c_\varepsilon)^{\otimes N}\circ (T_N^{s_0})^{-1}\right)
\approx\frac{1}{R}\sum_{r=1}^{R}\big|T_{(r)}-\tilde T_{(r)}\big|,
\]
where $\{T_{(r)}\}$ and $\{\tilde T_{(r)}\}$ are the sorted samples of
$\{T_N^{s_0}(\vec{\hat\xi}_N^{(r)})\}_{r=1}^R$ and $\{T_N^{s_0}(\vec{\tilde\xi}_N^{(r)})\}_{r=1}^R$.
The left panel of Figure~\ref{fig:qsr-num} depicts a linear relationship between $\dd_K(P^c,\tilde P^c_\varepsilon)$ and $\varepsilon$ as theoretically guaranteed. Based on the approximate linear relationship,
we plot the right panel with $\varepsilon$ as horizontal axis. The figure shows a sublinear relationship between $\dd_K\big((P^c)^{\otimes N}\circ (T_N^{s_0})^{-1},(\tilde P^c_\varepsilon)^{\otimes N}\circ (T_N^{s_0})^{-1}\big)$ and $\varepsilon$.
The latter relationship is consistent with Theorem~\ref{thm-qsr}, which predicts a controlled growth of the estimator-law deviation as the underlying distribution shift increases.

\begin{figure}[h]
\centering
\includegraphics[width=0.95\textwidth]{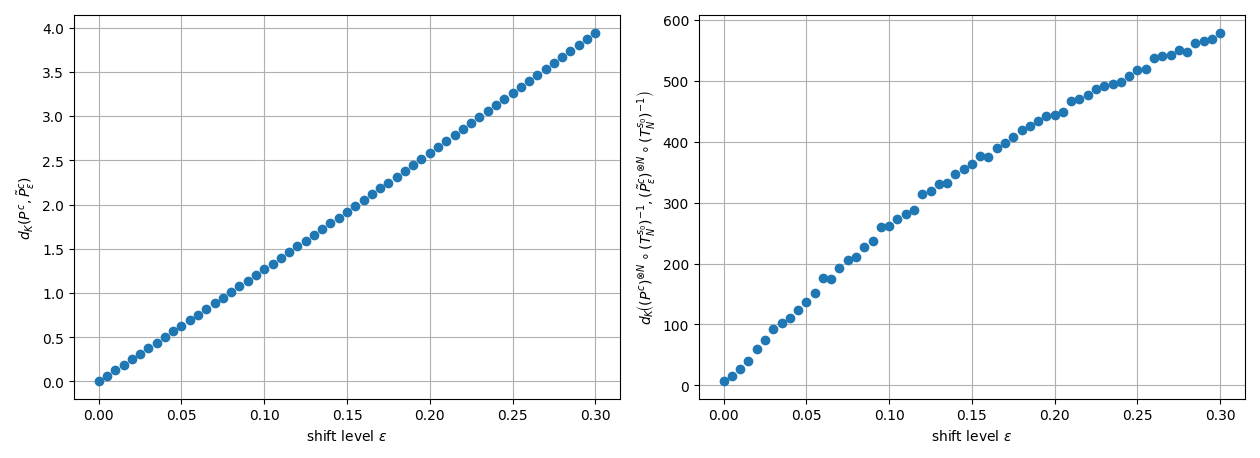}
\caption{Quantitative statistical robustness experiment at $s_0=0$.
}
\label{fig:qsr-num}
\end{figure}

We then demonstrate the computational performance of Algorithm~\ref{alg:B} on this problem.
We consider two variants of Algorithm~\ref{alg:B}. Both variants solve the episodic Bayesian DROC Bellman equation via the BOCP scheme, but they differ in whether the lower approximation is warm-started by reusing cuts from the previous episode. The concrete results are plotted in Figure~\ref{fig:BOCP-within}. The left panel in Figure~\ref{fig:BOCP-within} plots the integrated gap between the episode-wise optimal value function $\hat V_N^*$ and the BOCP lower approximation $\underline V_N^{k}$ under the stationary distribution $\varpi_N$ induced by the episodic optimal policy $\hat \pi^*_N$, i.e.,
$\int_{\mathcal S}\big(\hat V_N^*(s)-\underline V_N^{k}(s)\big)d\varpi_N(s).$
We drop the absolute value since $\underline V_N^{k}$ is a lower bound of $\hat V_N^*$ by construction.
The vertical dashed lines indicate the switches between successive episodes. 
The right panel in Figure~\ref{fig:BOCP-within} reports the within-episode convergence of BOCP in the episode $N=100$ by plotting the Bellman residual
$\mathfrak{R}_k=\|\hat{\cal L}_N(\underline V_N^{k})-\underline V_N^{k}\|_\infty$ against the BOCP iteration index $k$ (log scale).
Overall, Figure~\ref{fig:BOCP-within} demonstrates that warm-start yields a much smaller initial gap at the beginning of each episode and accelerates convergence, illustrating the benefit of safely reusing valid cuts across posterior updates.

\begin{figure}[h]
 \centering
 \begin{subfigure}[b]{0.48\textwidth}
 \centering
 \includegraphics[width=\textwidth]{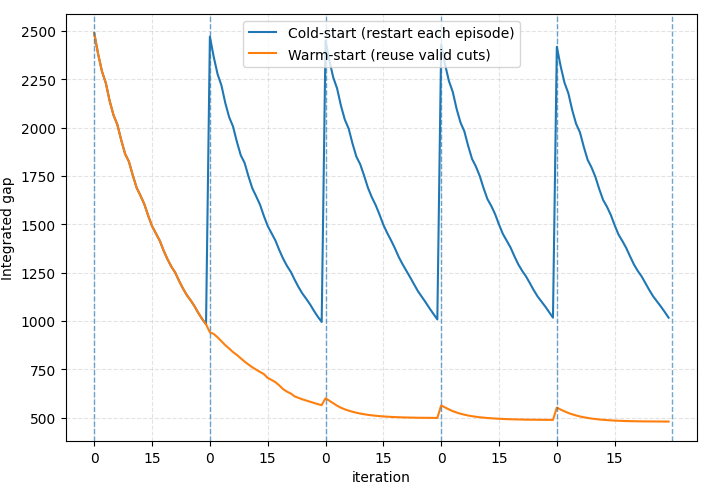}
 \caption{Integrated gap across episodes $N=1,\ldots,5$.}
 \end{subfigure}%
 \hfill
 \begin{subfigure}[b]{0.48\textwidth}
 \centering
 \includegraphics[width=\textwidth]{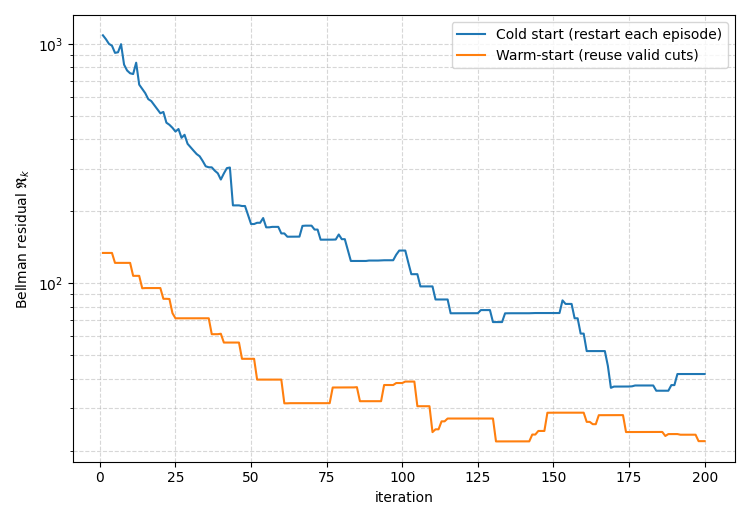}
 \caption{Bellman residual at episode $N=100$.}
 \end{subfigure}
 \caption{BOCP warm-start vs.\ cold-start for the episodic Bayesian DROC. }
 \label{fig:BOCP-within}
\end{figure}

In the following, we compare the proposed episodic Bayesian DROC approach with episodic Bayesian SOC in \cite{shapiro2025episodic} and an
extension of distributionally robust stochastic control (denoted by DRSC) proposed in \cite{shapiro2012minimax}.
Notably, DRSC only considers a single episode, meaning that the ambiguity set $\mathcal{P}^{DRSC}$ is only constructed once with some pre-collected historical data set. Following the setting in \cite{shapiro2025episodic}, we implement an episodic extension: at the beginning of each episode, we reconstruct the ambiguity set using all data collected so far and shrink its size at the standard rate $O(t^{-1/2})$, where $t$ denotes the cumulative number of observations up to the current episode. 
Specifically, we set the DRSC ambiguity set as a $L_\infty$ ball centered at the nominal distribution $\hat P_N^{\mathrm E}$, i.e.,
$\mathcal P^{\mathrm{DRSC}}_t=\{Q\in\mathscr{P}(\Xi): \|Q-\hat P_t^{\mathrm E}\|_\infty\le \frac{1}{10\sqrt t}\}$
and the robust policy is obtained by solving the corresponding distributionally robust Bellman equation under $\mathcal P^{\mathrm{DRSC}}_{t}$.

\begin{figure}[htbp]
 \centering

 \begin{subfigure}[t]{0.6\textwidth}
 \centering
 \includegraphics[width=\textwidth]{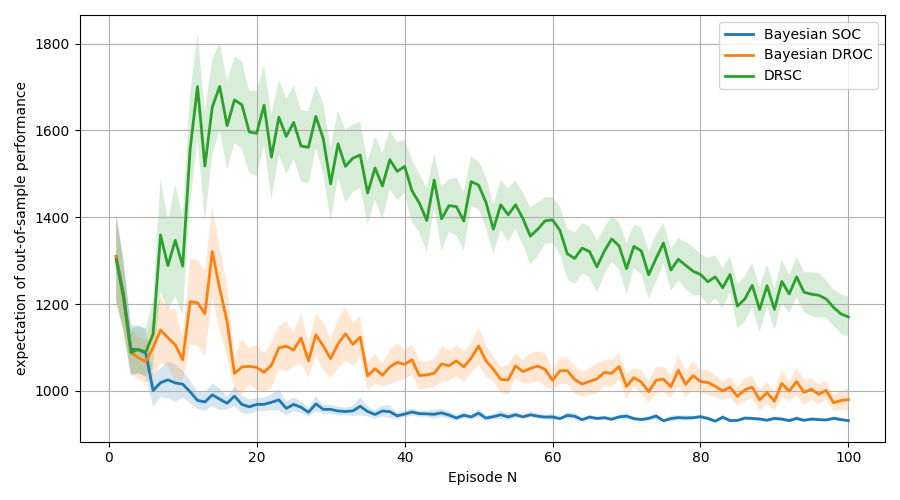}
 \caption{Expected cost.}
 \end{subfigure}
 
 \begin{subfigure}[t]{0.6\textwidth}
 \centering
 \includegraphics[width=\textwidth]{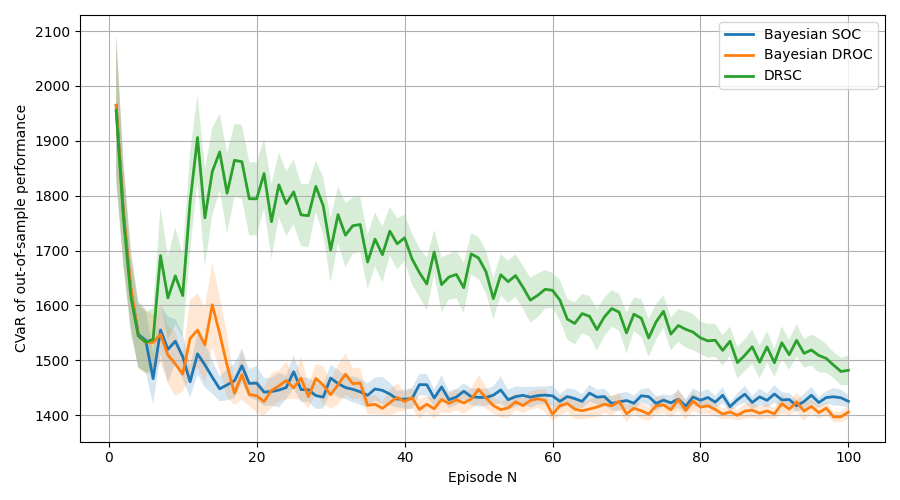}
 \caption{$\mathrm{CVaR}^{0.95}$.}
 \end{subfigure}
 
 \begin{subfigure}[t]{0.6\textwidth}
 \centering
 \includegraphics[width=\textwidth]{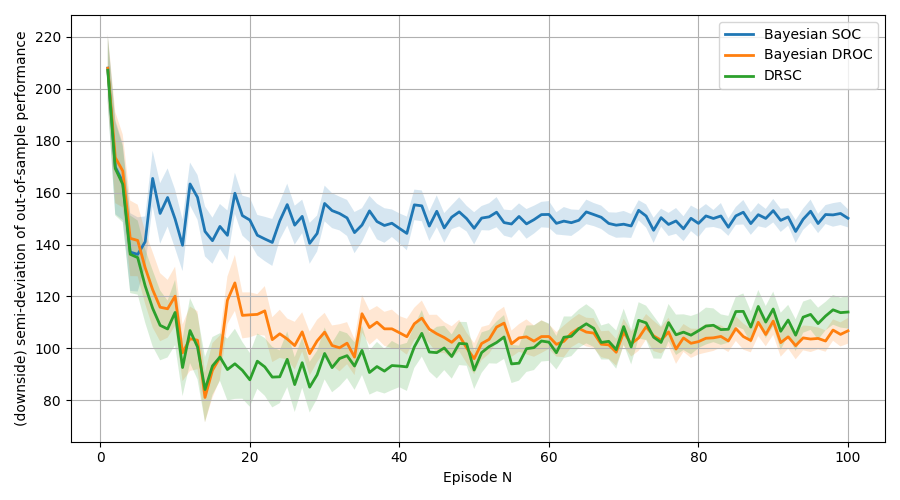}
 \caption{Downside semi-deviation.}
 \end{subfigure}

 \caption{Out-of-sample discounted cost under the true environment versus episode index $N$.}
 \label{fig:2}
\end{figure}
Figure \ref{fig:2} compares the out-of-sample performance of the learned policies under three different risk measures 
(mean, $\mathrm{CVaR}^{0.95}$, and downside semi-deviation) across episodes under a clean environment $\mathrm{Exp}(10)$.
For each episode index $N$, we evaluate the episode-wise optimal policy $\hat\pi_N^*$ under the true demand distribution $P^c$ using 2000 independent rollouts. Each rollout produces a discounted cost sample
$\sum_{t=1}^{T}\gamma^{t-1}\big(ca_t+\psi(s_t+a_t,\xi_t)\big)$,
where $a_t=\hat\pi_N^*(s_t)$ with truncation $T=250$. We then compute the $\mathrm{CVaR}^{0.95}$ and the downside semi-deviation of the discounted total cost from the resulting cost samples. To account for randomness in the training data and posterior updates, we repeat the entire procedure over 100 independent replications. The curves report the replication mean, and the shaded bands show $95\%$ confidence intervals based on the standard error across replications.

\begin{figure}[h]
 \centering
 \begin{subfigure}[t]{0.9\textwidth}
 \centering
 \includegraphics[width=\textwidth]{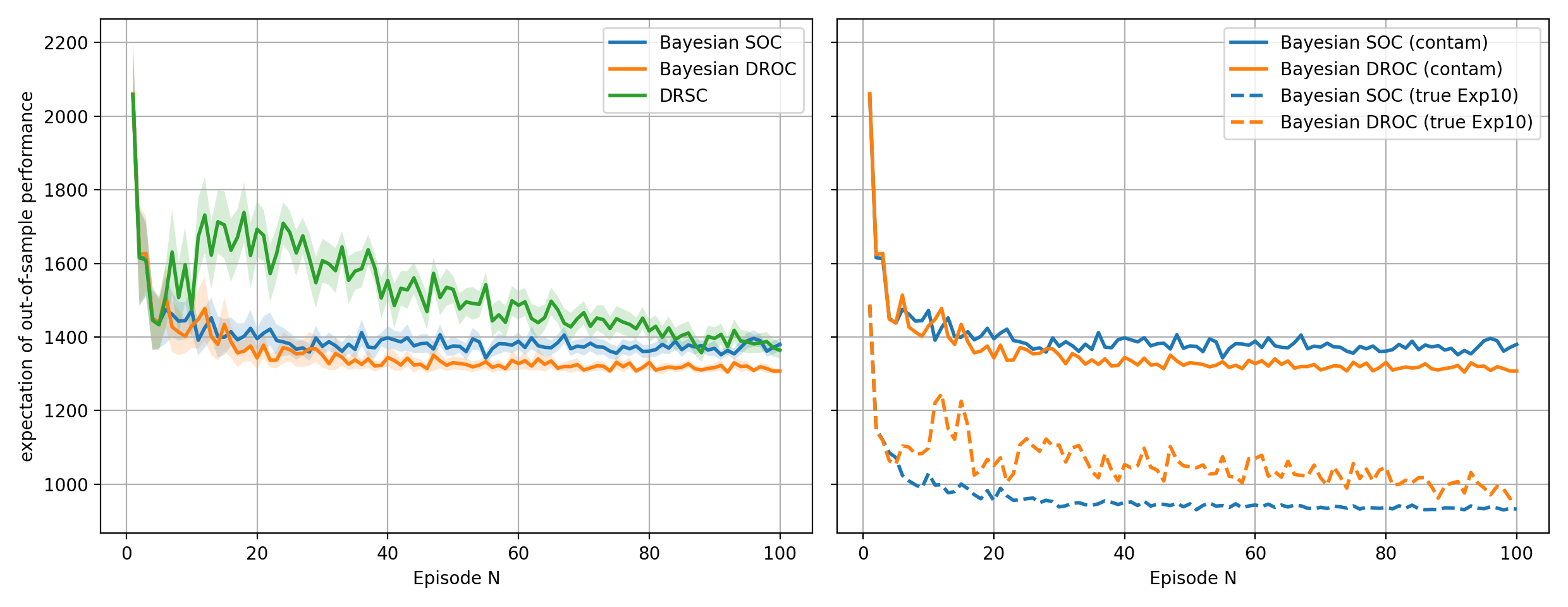}
 \caption{Huber mixture case.}
 \end{subfigure}
 \begin{subfigure}[t]{0.9\textwidth}
	\centering \includegraphics[width=\textwidth]{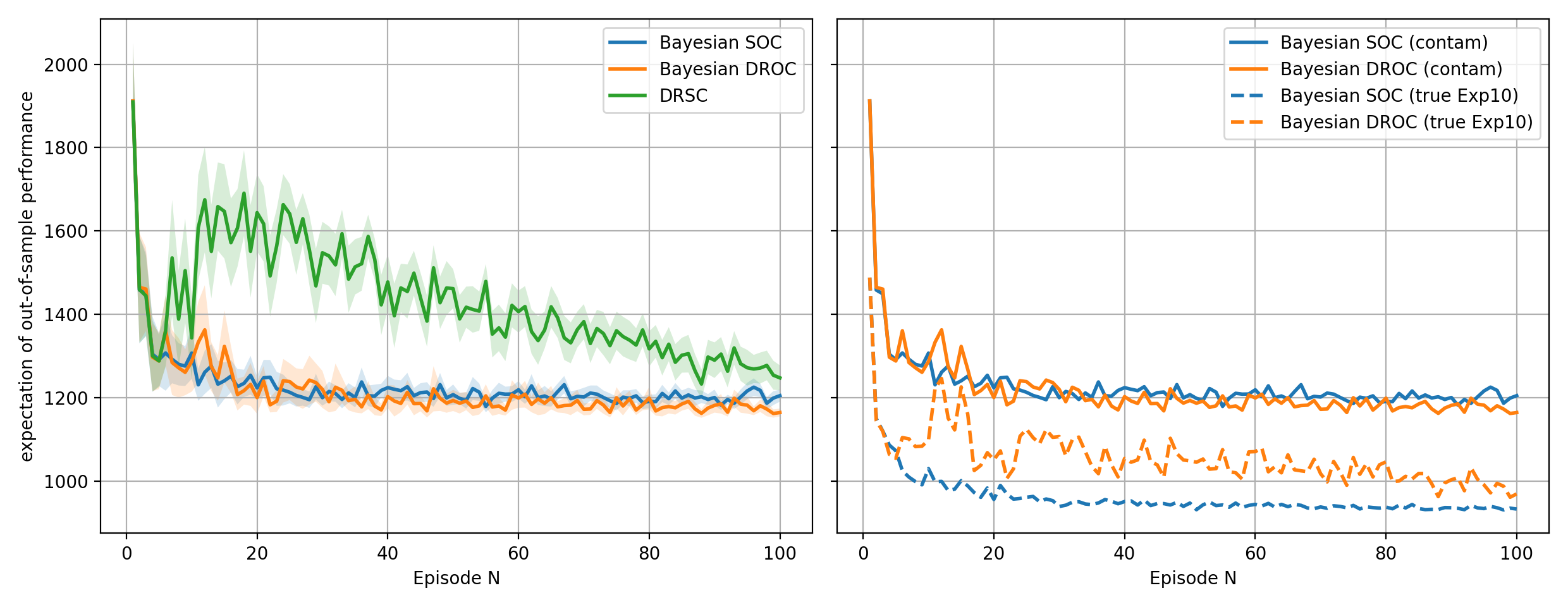} 
 \caption{Scale shift case.}
 \end{subfigure}
\caption{Out-of-sample discounted cost under the contaminated environment versus episode index $N$.}
	\label{fig:contaminated} 
\end{figure}

Figure~\ref{fig:contaminated} further compares the expected out-of-sample discounted cost of the learned policies across episodes under contaminated test environments. Specifically, in each episode $N$, the training data are generated from the clean environment $P^c$ and the learned policy is then evaluated under the shifted test environment $\tilde P^c_{\varepsilon}$ with $\varepsilon=0.3$, considering two shift mechanisms:
(i) Huber mixture $\tilde P^c_{\varepsilon}=(1-\varepsilon)\mathrm{Exp}(10)+\varepsilon\mathrm{Exp}(30)$;
(ii) scale shift $\tilde P^c_{\varepsilon}=\mathrm{Exp}(10(1+\varepsilon))$.
The left panels show the average performance of all methods when deployed in the contaminated test environment $\tilde P^c_{\varepsilon}$.
In contrast, the right panels focus on the two Bayesian methods and disentangle the effect of distribution shift by evaluating the learned policies under both the clean environment $P^c$ and the contaminated environment $\tilde P^c_{\varepsilon}$ within the same plot.
This visualization highlights robustness under contamination or distribution shift when the learned Bayesian policies are deployed in shifted environments.
It should be noted that this experiment is different from the quantitative statistical robustness experiment in Figure~\ref{fig:qsr-num}: there the episode index $N$ is fixed and the perturbation level $\varepsilon$ varies, whereas here $\varepsilon$ is fixed and the episode index $N$ increases.
Nevertheless, the two experiments are consistent: for a fixed distribution shift, the induced estimation error remains uniformly controlled with respect to $N$.
This helps explain why the gap between the clean-test and shifted-test performance curves remains contained as the episode index increases.

Overall, Figures~\ref{fig:2}--\ref{fig:contaminated} report the out-of-sample discounted costs of the learned policies
as the episode $N$ increases.
In the {clean} environment (Figure~\ref{fig:2}), the episodic Bayesian SOC method tends to achieve a lower expected cost.
A possible explanation is that it optimizes a risk-neutral posterior-averaged objective under the epistemic uncertainty.
In contrast, the episodic Bayesian DROC method exhibits a consistently stronger control of tail and
downside risks, reflected by smaller $\mathrm{CVaR}^{0.95}$ and downside semi-deviation across episodes.
This observation is consistent with Theorem~\ref{thm-reformulation}, which shows that the Bayesian DROC Bellman equation admits an equivalent mean--CVaR reformulation and therefore embeds an explicit risk-averse preference for adverse demand realizations.
Hence, Bayesian DROC may trade a small amount of mean performance for a more stable out-of-sample behavior,
especially in the upper-tail events that dominate service-level and backorder risks in inventory systems.

More importantly, under contaminated environments (Figure~\ref{fig:contaminated}), where the true demand distribution
deviates from the nominal exponential family used for posterior learning, Bayesian DROC shows stronger robustness to distribution shift.
Specifically, in the left panels, Bayesian DROC attains lower expected out-of-sample discounted cost than the competing methods under the shifted test environment.
Moreover, in the right panels, the gap between the clean-test and contaminated-test performance curves is generally smaller for Bayesian DROC than for episodic Bayesian SOC, indicating that its performance deteriorates less severely when deployed under model mismatch.
Together, these observations suggest that the ambiguity set $\mathcal P_{\mu_N}$ provides a principled uncertainty envelope and that the robust Bellman equation~\eqref{eq-num-droc} produces policies that generalize more reliably under distribution shift.

As noted in \cite{shapiro2025episodic}, there is currently limited methodological guidance on how to adaptively shrink
ambiguity sets for distributionally robust approaches in dynamic settings, and DRSC-type methods typically rely on a
pre-specified radius schedule, which can be overly conservative in finite samples.
In contrast, our episodic Bayesian DROC constructs $\mathcal P_{\mu_N}$ directly from posterior credible regions and updates it
episode by episode using all data collected so far, so the effective ambiguity size contracts automatically as the
posterior concentrates. This posterior-driven update mitigates the conservatism caused by hand-crafted shrinkage
rules and explains why Bayesian DROC outperforms the DRSC baseline in Figures~\ref{fig:2}--\ref{fig:contaminated}.

\section{Concluding remarks}
In this paper, we combine distributionally robust optimization with Bayesian adaptive learning to propose an episodic distributionally robust optimal control (DROC) framework with posterior-induced ambiguity sets. The new model systematically integrates the decision-maker's belief with sequentially revealed observations across episodes.
Our analysis reveals that the proposed framework possesses several distinctive advantages: (i) an explicit mean-risk (expectation--CVaR) reformulation of the robust Bellman operator, which facilitates both computation and interpretation, (ii) asymptotic convergence of the episodic value functions and policies, together with a finite-sample guarantee, and (iii) quantitative stability and statistical robustness guarantees under suitable regularity conditions and data perturbations.
We further develop an operator-level cutting-plane (BOCP) algorithm for the episodic Bayesian DROC Bellman equation. Numerical results show improved out-of-sample tail-risk performance and stronger robustness under distribution shift.

A simplifying assumption in our analysis is that the episodic Bayesian DROC problem is solved exactly in each episode, which may be impractical in some applications. This suggests several directions for future research.
One is to investigate the trade-off between estimation error, arising from Bayesian learning with finite data, and optimization error, arising from approximate solutions of episodic Bayesian DROC. Another is to study the more general adaptive Bayesian SOC model with an augmented state space that explicitly includes posterior beliefs, together with more efficient solution algorithms.


\section*{Acknowledgments}
The first and second authors are supported by the National Key R\&D Program of China (2022YFA1004000, 2022YFA1004001).

	\bibliographystyle{abbrv}
	\bibliography{ref}

@book{pflug2014multistage,
  title={Multistage Stochastic Optimization},
  author={Pflug, Georg Ch and Pichler, Alois},
  volume={1104},
  year={2014},
  publisher={Springer}
}

@article{li2020confidence,
  title={Confidence interval based distributionally robust real-time economic dispatch approach considering wind power accommodation risk},
  author={Li, Peng and Yang, Ming and Wu, Qiuwei},
  journal={IEEE Transactions on Sustainable Energy},
  volume={12},
  number={1},
  pages={58--69},
  year={2020},
  publisher={IEEE}
}

@phdthesis{nilim2004robust,
  title={Robust markov decision processes with uncertain transition matrices},
  author={Nilim, Arnab and El Ghaoui, Laurent},
  year={2004},
  school={University of California, Berkeley}
}

@book{castaing1977convex,
  title        = {Convex Analysis and Measurable Multifunctions},
  author       = {Castaing, Charles and Valadier, Michel},
  year         = {1977},
  publisher    = {Springer},
}

@article{guo2019distributionally,
  title={Distributionally robust shortfall risk optimization model and its approximation},
  author={Guo, Shaoyan and Xu, Huifu},
  journal={Mathematical Programming},
  volume={174},
  number={1},
  pages={473--498},
  year={2019},
  publisher={Springer}
}

@article{yang2025stability,
  title={Stability Analysis of An Integrated Multistage Stochastic Programming and {M}arkov Decision Process Problem},
  author={Yang, Zhiyao and Chen, Zhiping and Xu, Huifu},
  journal={arXiv preprint arXiv:2509.22194},
  year={2025}
}

@article{wessels1977markov,
  title={Markov programming by successive approximations with respect to weighted supremum norms},
  author={Wessels, Jaap},
  journal={Journal of mathematical analysis and applications},
  volume={58},
  number={2},
  pages={326--335},
  year={1977},
  publisher={Elsevier}
}

@inproceedings{li2022risk,
  title={Risk-aware model predictive control enabled by {B}ayesian learning},
  author={Li, Yingke and Lin, Yifan and Zhou, Enlu and Zhang, Fumin},
  booktitle={2022 American Control Conference (ACC)},
  pages={108--113},
  year={2022},
  organization={IEEE}
}

@article{li2025bayesian,
  title={Bayesian Risk-averse Model Predictive Control with Consistency and Stability Guarantees},
  author={Li, Yingke and Lin, Yifan and Zhou, Enlu and Zhang, Fumin},
  journal={arXiv preprint arXiv:2511.21871},
  year={2025}
}

@article{dibiasi2021measuring,
  title={Measuring {K}nightian uncertainty},
  author={Dibiasi, Andreas and Iselin, David},
  journal={Empirical Economics},
  volume={61},
  number={4},
  pages={2113--2141},
  year={2021},
  publisher={Springer}
}

@article{chen2024bayesian,
  title={A {B}ayesian approach to data-driven multi-stage stochastic optimization},
  author={Chen, Zhiping and Ma, Wentao},
  journal={Journal of Global Optimization},
  pages={1--28},
  year={2024},
  publisher={Springer}
}

@article{wang2016likelihood,
	title={Likelihood robust optimization for data-driven problems},
	author={Wang, Zizhuo and Glynn, Peter W and Ye, Yinyu},
	journal={Computational Management Science},
	volume={13},
	pages={241--261},
	year={2016},
	publisher={Springer}
}

@inproceedings{strens2000bayesian,
	title={A {B}ayesian framework for reinforcement learning},
	author={Strens, Malcolm},
	booktitle={International Conference on Machine Learning},
	volume={2000},
	pages={943--950},
	year={2000}
}

@article{osband2013more,
	title={{(More)} efficient reinforcement learning via posterior sampling},
	author={Osband, Ian and Russo, Daniel and Van Roy, Benjamin},
	journal={Advances in Neural Information Processing Systems},
	volume={26},
	year={2013}
}

@book{bertsekas2012dynamic,
	title={Dynamic Programming and Optimal Control: Volume I},
	author={Bertsekas, Dimitri},
	volume={4},
	year={2012},
	publisher={Athena Scientific}
}

@article{lam2019recovering,
  title={Recovering best statistical guarantees via the empirical divergence-based distributionally robust optimization},
  author={Lam, Henry},
  journal={Operations Research},
  volume={67},
  number={4},
  pages={1090--1105},
  year={2019},
  publisher={INFORMS}
}

@article{hanasusanto2015distributionally,
  title={A distributionally robust perspective on uncertainty quantification and chance constrained programming},
  author={Hanasusanto, Grani A and Roitch, Vladimir and Kuhn, Daniel and Wiesemann, Wolfram},
  journal={Mathematical Programming},
  volume={151},
  number={1},
  pages={35--62},
  year={2015},
  publisher={Springer}
}

@article{delage2010distributionally,
  title={Distributionally robust optimization under moment uncertainty with application to data-driven problems},
  author={Delage, Erick and Ye, Yinyu},
  journal={Operations research},
  volume={58},
  number={3},
  pages={595--612},
  year={2010},
  publisher={INFORMS}
}

@book{efron2021computer,
  title={Computer Age Statistical Inference, Student Edition: Algorithms, Evidence, and Data Science},
  author={Efron, Bradley and Hastie, Trevor},
  volume={6},
  year={2021},
  publisher={Cambridge University Press}
}

@article{pfeiffer2018two,
  title={Two approaches to stochastic optimal control problems with a final-time expectation constraint},
  author={Pfeiffer, Laurent},
  journal={Applied Mathematics \& Optimization},
  volume={77},
  pages={377--404},
  year={2018},
  publisher={Springer}
}

@article{carpentier2012dynamic,
  title={Dynamic consistency for stochastic optimal control problems},
  author={Carpentier, Pierre and Chancelier, Jean-Philippe and Cohen, Guy and De Lara, Michel and Girardeau, Pierre},
  journal={Annals of Operations Research},
  volume={200},
  pages={247--263},
  year={2012},
  publisher={Springer}
}

@book{huber2011robust,
  title={Robust Statistics},
  author={Huber, Peter J and Ronchetti, Elvezio M},
  year={2011},
  publisher={John Wiley \& Sons}
}

@article{hampel1971general,
  title={A general qualitative definition of robustness},
  author={Hampel, Frank R},
  journal={The Annals of Mathematical Statistics},
  volume={42},
  number={6},
  pages={1887--1896},
  year={1971},
  publisher={Institute of Mathematical Statistics}
}

@article{kern2020first,
  title={{First-order sensitivity of the optimal value in a Markov decision model with respect to deviations in the transition probability function}},
  author={Kern, Patrick and Simroth, Axel and Z{\"a}hle, Henryk},
  journal={Mathematical Methods of Operations Research},
  volume={92},
  number={1},
  pages={165--197},
  year={2020},
  publisher={Springer}
}

@article{tzortzis2019infinite,
  title={Infinite horizon average cost dynamic programming subject to total variation distance ambiguity},
  author={Tzortzis, Ioannis and Charalambous, Charalambos D and Charalambous, Themistoklis},
  journal={SIAM Journal on Control and Optimization},
  volume={57},
  number={4},
  pages={2843--2872},
  year={2019},
  publisher={SIAM}
}

@article{shapiro2012minimax,
  title={Minimax and risk averse multistage stochastic programming},
  author={Shapiro, Alexander},
  journal={European Journal of Operational Research},
  volume={219},
  number={3},
  pages={719--726},
  year={2012},
  publisher={Elsevier}
}

@book{bertsekas1996stochastic,
  title={Stochastic Optimal Control: The Discrete-time Case},
  author={Bertsekas, Dimitri and Shreve, Steven E},
  volume={5},
  year={1996},
  publisher={Athena Scientific}
}

@article{wang2024aleatoric,
  title={Aleatoric and epistemic discrimination: Fundamental limits of fairness interventions},
  author={Wang, Hao and He, Luxi and Gao, Rui and Calmon, Flavio},
  journal={Advances in Neural Information Processing Systems},
  volume={36},
  year={2024}
}

@article{ma2024bayesian,
  title={A {B}ayesian Composite Risk Approach for Stochastic Optimal Control and {M}arkov Decision Processes},
  author={Ma, Wentao and Chen, Zhiping and Xu, Huifu},
  journal={arXiv preprint arXiv:2412.16488},
  year={2024}
}

@article{liu2013stability,
  title={Stability analysis of stochastic programs with second order dominance constraints},
  author={Liu, Yongchao and Xu, Huifu},
  journal={Mathematical Programming},
  volume={142},
  pages={435--460},
  year={2013},
  publisher={Springer}
}

@article{shapiro2025episodic,
  title={Episodic {B}ayesian optimal control with unknown randomness distributions},
  author={Shapiro, Alexander and Zhou, Enlu and Lin, Yifan and Wang, Yuhao},
  journal={Operations Research},
  year={2025},
  publisher={INFORMS}
}

@article{xu2010distributionally,
	title={Distributionally robust Markov decision processes},
	author={Xu, Huan and Mannor, Shie},
	journal={Advances in Neural Information Processing Systems},
	volume={23},
	year={2010}
}

@article{guigues2023risk,
	title={Risk-Averse Stochastic Optimal Control: an efficiently computable statistical upper bound},
	author={Guigues, Vincent and Shapiro, Alexander and Cheng, Yi},
	journal={Operations Research Letters},
	volume={51},
	number={4},
	pages={393--400},
	year={2023},
	publisher={Elsevier}
}

@article{li2024discretization,
  title={Discretization and quantification for distributionally robust optimization with decision-dependent ambiguity sets},
  author={Li, Manlan and Tong, Xiaojiao and Sun, Hailin},
  journal={Optimization Methods and Software},
  pages={1--30},
  year={2024},
  publisher={Taylor \& Francis}
}

@article{wang2023bayesian,
  title={Bayesian risk-averse {Q}-learning with streaming observations},
  author={Wang, Yuhao and Zhou, Enlu},
  journal={Advances in Neural Information Processing Systems},
  volume={36},
  pages={75967--75992},
  year={2023}
}

@article{rahimian2022effective,
  title={Effective scenarios in multistage distributionally robust optimization with a focus on total variation distance},
  author={Rahimian, Hamed and Bayraksan, Guzin and De-Mello, Tito Homem},
  journal={SIAM Journal on Optimization},
  volume={32},
  number={3},
  pages={1698--1727},
  year={2022},
  publisher={SIAM}
}

@book{shapiro2021lectures,
	title={Lectures on Stochastic Programming: Modeling and Theory},
	author={Shapiro, Alexander and Dentcheva, Darinka and Ruszczynski, Andrzej},
	year={2021},
	publisher={SIAM}
}

@book{puterman2014markov,
	title={Markov Decision Processes: Discrete Stochastic Dynamic Programming},
	author={Puterman, Martin L},
	year={2014},
	publisher={John Wiley \& Sons}
}

@article{jiang2018risk,
  title={Risk-averse two-stage stochastic program with distributional ambiguity},
  author={Jiang, Ruiwei and Guan, Yongpei},
  journal={Operations Research},
  volume={66},
  number={5},
  pages={1390--1405},
  year={2018},
  publisher={INFORMS}
}

@article{mehrotra2014models,
  title={Models and algorithms for distributionally robust least squares problems},
  author={Mehrotra, Sanjay and Zhang, He},
  journal={Mathematical Programming},
  volume={146},
  number={1},
  pages={123--141},
  year={2014},
  publisher={Springer}
}

@book{berger2013statistical,
	title={Statistical Decision Theory and Bayesian Analysis},
	author={Berger, James O},
	year={2013},
	publisher={Springer Science \& Business Media}
}

@article{guo2021statistical,
  title={Statistical robustness in utility preference robust optimization models},
  author={Guo, Shaoyan and Xu, Huifu},
  journal={Mathematical Programming},
  volume={190},
  number={1},
  pages={679--720},
  year={2021},
  publisher={Springer}
}

@article{bertsimas2018robust,
  title={Robust sample average approximation},
  author={Bertsimas, Dimitris and Gupta, Vishal and Kallus, Nathan},
  journal={Mathematical Programming},
  volume={171},
  pages={217--282},
  year={2018},
  publisher={Springer}
}

@article{chen2025data,
  title={Data-driven approximation of distributionally robust chance constraints using {B}ayesian credible intervals},
  author={Chen, Zhiping and Ma, Wentao and Ji, Bingbing},
  journal={OR Spectrum},
  volume={47},
  number={3},
  pages={969--1009},
  year={2025},
  publisher={Springer}
}

@book{hernandez2012further,
  title={Further Topics on Discrete-time Markov Control Processes},
  author={Hern{\'a}ndez-Lerma, On{\'e}simo and Lasserre, Jean B},
  volume={42},
  year={2012},
  publisher={Springer Science \& Business Media}
}

@book{bertsekas2022abstract,
  title={Abstract Dynamic Programming},
  author={Bertsekas, Dimitri},
  year={2022},
  publisher={Athena Scientific}
}

@article{gao2023finite,
  title={Finite-sample guarantees for {W}asserstein distributionally robust optimization: {B}reaking the curse of dimensionality},
  author={Gao, Rui},
  journal={Operations Research},
  volume={71},
  number={6},
  pages={2291--2306},
  year={2023},
  publisher={INFORMS}
}

@book{gelman1995bayesian,
	title={{Bayesian Data Analysis}},
	author={Gelman, Andrew and Carlin, John B and Stern, Hal S and Rubin, Donald B},
	year={1995},
	publisher={Chapman and Hall/CRC}
}

@inproceedings{abeille2018improved,
  title={Improved regret bounds for Thompson sampling in linear quadratic control problems},
  author={Abeille, Marc and Lazaric, Alessandro},
  booktitle={International Conference on Machine Learning},
  pages={1--9},
  year={2018},
  organization={PMLR}
}

@article{gupta2019near,
	title={Near-optimal {B}ayesian ambiguity sets for distributionally robust optimization},
	author={Gupta, Vishal},
	journal={Management Science},
	volume={65},
	number={9},
	pages={4242--4260},
	year={2019},
	publisher={INFORMS}
}

@article{haskell2016empirical,
  title={Empirical dynamic programming},
  author={Haskell, William B and Jain, Rahul and Kalathil, Dileep},
  journal={Mathematics of Operations Research},
  volume={41},
  number={2},
  pages={402--429},
  year={2016},
  publisher={INFORMS}
}

@incollection{van1996weak,
  title={Weak convergence},
  author={Van Der Vaart, Aad W and Wellner, Jon A},
  booktitle={Weak convergence and empirical processes: with applications to statistics},
  pages={16--28},
  year={1996},
  publisher={Springer}
}

@book{aliprantis2006infinite,
  title={Infinite Dimensional Analysis: A Hitchhiker’s Guide},
  author={Aliprantis, Charalambos D and Border, Kim C},
  year={2006},
  publisher={Springer}
}

@article{huang2017study,
	title={A study of distributionally robust multistage stochastic optimization},
	author={Huang, Jianqiu and Zhou, Kezhuo and Guan, Yongpei},
	journal={arXiv preprint arXiv:1708.07930},
	year={2017}
}

@article{lin2022bayesian,
	title={Bayesian Risk {M}arkov Decision Processes},
	author={Lin, Yifan and Ren, Yuxuan and Zhou, Enlu},
	journal={Advances in Neural Information Processing Systems},
	volume={35},
	pages={17430--17442},
	year={2022}
}

@article{mohajerin2018data,
  title={Data-driven distributionally robust optimization using the {W}asserstein metric: Performance guarantees and tractable reformulations},
  author={Mohajerin Esfahani, Peyman and Kuhn, Daniel},
  journal={Mathematical Programming},
  volume={171},
  number={1},
  pages={115--166},
  year={2018},
  publisher={Springer}
}

@article{philpott2018distributionally,
	title={Distributionally robust {SDDP}},
	author={Philpott, Andrew B and de Matos, Vitor L and Kapelevich, Lea},
	journal={Computational Management Science},
	volume={15},
	pages={431--454},
	year={2018},
	publisher={Springer}
}

@article{pichler2022quantitative,
	title={Quantitative stability analysis for minimax distributionally robust risk optimization},
	author={Pichler, Alois and Xu, Huifu},
	journal={Mathematical Programming},
	volume={191},
	number={1},
	pages={47--77},
	year={2022},
	publisher={Springer}
}

@article{cooper2012performance,
	title={Performance guarantees for empirical {M}arkov decision processes with applications to multiperiod inventory models},
	author={Cooper, William L and Rangarajan, Bharath},
	journal={Operations Research},
	volume={60},
	number={5},
	pages={1267--1281},
	year={2012},
	publisher={INFORMS}
}

@article{fullner2025stochastic,
  title={Stochastic dual dynamic programming and its variants: A review},
  author={F{\"u}llner, Christian and Rebennack, Steffen},
  journal={SIAM Review},
  volume={67},
  number={3},
  pages={415--539},
  year={2025},
  publisher={SIAM}
}

@article{philpott2008convergence,
  title={On the convergence of stochastic dual dynamic programming and related methods},
  author={Philpott, Andrew B and Guan, Ziming},
  journal={Operations Research Letters},
  volume={36},
  number={4},
  pages={450--455},
  year={2008},
  publisher={Elsevier}
}

@book{rockafellar2015convex,
  title={Convex Analysis},
  author={Rockafellar, Ralph Tyrell},
  year={2015},
  publisher={Princeton university press}
}

@article{ma2025bayesian,
  title={Bayesian distributionally robust variational inequalities: regularization and quantification},
  author={Ma, Wentao and Chen, Zhiping and Chen, Xiaojun},
  journal={arXiv preprint arXiv:2509.16537},
  year={2025}
}

@article{xie2021nonparametric,
	title={A nonparametric {B}ayesian framework for uncertainty quantification in stochastic simulation},
	author={Xie, Wei and Li, Cheng and Wu, Yuefeng and Zhang, Pu},
	journal={SIAM/ASA Journal on Uncertainty Quantification},
	volume={9},
	number={4},
	pages={1527--1552},
	year={2021},
	publisher={SIAM}
}

@article{rieder1975bayesian,
	title={Bayesian dynamic programming},
	author={Rieder, Ulrich},
	journal={Advances in Applied Probability},
	volume={7},
	number={2},
	pages={330--348},
	year={1975},
	publisher={Cambridge University Press}
}

@article{yang2020wasserstein,
	title={Wasserstein distributionally robust stochastic control: {A} data-driven approach},
	author={Yang, Insoon},
	journal={IEEE Transactions on Automatic Control},
	volume={66},
	number={8},
	pages={3863--3870},
	year={2020},
	publisher={IEEE}
}

@article{taskesen2023distributionally,
  title={Distributionally robust linear quadratic control},
  author={Taskesen, Bahar and Iancu, Dan and Ko{\c{c}}yi{\u{g}}it, {\c{C}}a{\u{g}}{\i}l and Kuhn, Daniel},
  journal={Advances in Neural Information Processing Systems},
  volume={36},
  pages={18613--18632},
  year={2023}
}

@article{kim2023distributional,
  title={Distributional robustness in minimax linear quadratic control with {W}asserstein distance},
  author={Kim, Kihyun and Yang, Insoon},
  journal={SIAM Journal on Control and Optimization},
  volume={61},
  number={2},
  pages={458--483},
  year={2023},
  publisher={SIAM}
}

@article{van2015distributionally,
	title={Distributionally robust control of constrained stochastic systems},
	author={Van Parys, Bart PG and Kuhn, Daniel and Goulart, Paul J and Morari, Manfred},
	journal={IEEE Transactions on Automatic Control},
	volume={61},
	number={2},
	pages={430--442},
	year={2015},
	publisher={IEEE}
}

@article{rockafellar2000optimization,
	title={Optimization of conditional value-at-risk},
	author={Rockafellar, R Tyrrell and Uryasev, Stanislav and others},
	journal={Journal of risk},
	volume={2},
	pages={21--42},
	year={2000},
	publisher={Citeseer}
}

@article{xu2021quantitative,
	title={Quantitative statistical robustness in distributionally robust optimization models},
	author={Xu, Huifu and Zhang, Sainan},
	journal={Pacific Journal of Optimization Special Issue},
	year={2021}
}

@incollection{romisch2003stability,
  title     = {Stability of Stochastic Programming Problems},
  author    = {R{\"o}misch, Werner},
  booktitle = {Handbooks in Operations Research and Management Science},
  volume    = {10},
  pages     = {483--554},
  year      = {2003},
  publisher = {Elsevier}
}
	
\end{document}